\documentclass[11pt]{article}
\usepackage[pagebackref,colorlinks=true
,urlcolor=blue,linkcolor=blue,citecolor=blue]{hyperref}

\usepackage{amsthm, amssymb, bm, mathrsfs, amsmath, amsfonts}
\usepackage{soul, color, tikz-cd}
\usepackage[]{graphicx}
\graphicspath{{/EPSF/}{../figures/}{figures/}}
\usepackage[colorinlistoftodos]{todonotes}

\newcommand{\J}[1]{{\color{red}[J: #1]}} 
\newcommand{\F}[1]{{\color{black!40!blue}[F: #1]}} 

\newcommand{\takeout}[1]{}

\newcommand{\np}{}

\input preamble

\def \Dm {{\mathbb D}}
\def \Xm {{\mathbb X}}
\def \Wm {{\mathbb W}}

\def \G {{\mathcal G}}
\def \SS {{\mathcal S}}
\def \rel {\sim_1}
\def \relTT {\sim_2}
\def \relNA {\sim_3}
\def \relAT {\sim_4}

\newcommand{\abs}[1]{\left\lvert #1 \right\rvert}
\newcommand{\aabs}[1]{\left\| #1 \right\|}
\newcommand{\ip}[2]{\left\langle #1,#2\right\rangle}
\newcommand{\iip}[2]{\left( #1,#2\right)}
\newcommand{\vgrad}{\operatorname{\overset{\tt{v}}{\nabla}}}
\newcommand{\hgrad}{\operatorname{\overset{\tt{h}}{\nabla}}}
\newcommand{\vgradE}{\operatorname{\overset{\tt{v}}{\nabla}}{}^E}
\newcommand{\hgradE}{\operatorname{\overset{\tt{h}}{\nabla}}{}^E}
\newcommand{\vdiv}{\operatorname{\overset{\tt{v}}{\mbox{\rm div}}}}
\newcommand{\hdiv}{\operatorname{\overset{\tt{h}}{\mbox{\rm div}}}}
\newcommand{\Z}{{\mathcal Z}}
\newcommand{\Ter}{\beta_{\mathrm{Ter}}}

\usepackage{mathtools}
\mathtoolsset{showonlyrefs}

\title{Integral geometry on manifolds with boundary and applications}
\author{Joonas Ilmavirta\thanks{Department of Mathematics and Statistics, University of Jyv\"askyl\"a, P.O. Box 35, FI-40014 University of Jyv\"askyl\"a, Finland. \href{mailto:joonas.ilmavirta@jyu.fi}{joonas.ilmavirta@jyu.fi}} \and Fran\c{c}ois Monard\thanks{Department of Mathematics, University of California Santa Cruz, 1156 High St., Santa Cruz CA 95064, USA. \href{mailto:fmonard@ucsc.edu}{fmonard@ucsc.edu}}}
\date{}

\begin{document}
\maketitle

\begin{abstract}
    We survey recent results on inverse problems for geodesic X-ray transforms and other linear and non-linear geometric inverse problems for Riemannian metrics, connections and Higgs fields defined on manifolds with boundary.
\end{abstract}

\tableofcontents

\newpage
\section{Introduction}\label{sec:intro}

Johann Radon and his contemporaries formulated several integral geometric problems, not only in linear but also in non-linear settings \cite{Herglotz1905,Wiechert1907}. Such problems, namely {\em travel-time tomography} and {\em boundary rigidity} as later formulated in \cite{Mukhometov1977,Michel1981}, are concerned with recovering a Riemannian metric from the shortest length between any two boundary points. Such problems and their cousins (described below), now make the field of {\em integral geometry}, or how to reconstruct geometric features of a manifold from integral functionals defined over that manifold. 

Nowadays this field forms the basis of several non-invasive approaches to imaging internal properties of materials: seismology \cite{Herglotz1905,Wiechert1907}, or how to reconstruct the density inside the Earth from first arrival times of seismic wavefronts; medical imaging since the development of X-ray Computerized Tomography \cite{Natterer2001,Stotzka2002,Finch2004}; Single-Photon Emission Computerized Tomography using the attenuated X-ray transform \cite{Novikov2002,Natterer2001a,Novikov2002a}; vector tomography in helio-seismology \cite{Kosovichev1996,Plaza2001,Kramar2016}; ocean imaging \cite{Munk1979}; X-ray diffraction strain tomography \cite{Lionheart2015,Desai2016} and tomography in elastic media \cite[Ch. 7]{Sharafutdinov1994}\cite{Sharafutdinov2012}; neutron imaging, as applied to the imaging of vertebrate remains \cite{Schwarz2005} and shales \cite{Chiang2017}. Non-linear integral geometric problems also continue to find new applications: recently, Neutron Spin Tomography \cite{Sales2017} as a means to measure magnetic fields in materials, has arised as a novel method which can be of use in electrical engineering, superconductivity, etc. The transform to invert in this case is a non-linear operator, the so-called ``non-abelian X-ray transform'' of the magnetic field, see Problem \ref{pb3} below. 

Recent breakthroughs have fuelled the field, exploiting a combination of old and new methods. Examples of such methods are: the systematic use of analysis on the unit sphere bundle combining energy methods (also coined ``Pestov identities''), initiated by Mukhometov \cite{Mukhometov1975} and generalized in \cite{Pestov1988,Sharafutdinov1994}, and harmonic analysis on the tangent fibers \cite{Dairbekov2007,Paternain2011a,Paternain2015}; in dimensions three and higher, the discovery in \cite{Uhlmann2015} that the existence of a foliation of the domain by strictly convex hyperfsurfaces, local or global, yields a powerful and robust approach to integral geometric inversions \cite{Stefanov2014,Stefanov2017,Zhou2013,Zhou2016}, via a successful use of Melrose's scattering calculus \cite{Melrose1995}; the systematic use of analytic microlocal analysis to produce 'generic' results, implying the unique identifiability of unknown parameters in an open and dense subset of all cases \cite{Stefanov2005,Homan2017,Zhou2017}; finally, recent results in the context of Anosov flows, leading to positive results for certain geometries with trapped sets \cite{Guillarmou2014,Guillarmou2014a,Guillarmou2015}.  

This review article aims at giving an overview of the arsenal of these methods, and to describe to what extent they help coping with various geometric settings, whose complexity is mainly governed by two features of the flow considered: the presence of conjugate points and/or infinite-length trajectories. 

\paragraph{Scope of the article.} The article will be devoted to manifolds with variable curvature, with less emphasis on homogeneous spaces for which the methods employed in, e.g. \cite{Helgason1999}, exploit homogeneity to a large extent and may not generalize. The emphasis will be put on manifolds with boundary, though many results enjoy counterparts in the realm of closed manifolds. The focus will be on mostly analytic methods, rather than topological or purely geometrical. The integration will be done over rays (no integration over higher-dimensional manifolds, see however the recent preprint involving an integral transform over two-dimensional leaves \cite{Salo2017}). Recent topical reviews have been published on some of the topics covered in what follows \cite{Paternain2012b,Paternain2013,Uhlmann2017}, and we have attempted to minimize overlap. 

It is our hope that this review article does justice to the field and its community, and we apologize in advance for any missing reference which would deserve to be included here. Let us mention that although the following topics are directly related to the current article, lack of time has prevented us to discuss range characterization issues, as provided e.g. in \cite{Pestov2004,Ainsworth2014,Paternain2013a,Sadiq2015,Monard2016d,Assylbekov2017} and cases where the boundary is non-convex, for which recent results appear in~\cite{Guillarmou2017a}.

\paragraph{Notation:}

\begin{itemize}
    \item $(M,g)$, $\partial M$, $TM$, $T^*M$, $SM$, $\partial_{+/-} SM$: a typical Riemannian manifold, its boundary, its tangent, cotangent, unit tangent bundles, and incoming/outgoing boundaries. 
    \item $\varphi_t(x,v) = (\gamma_{x,v}(t), \dot \gamma_{x,v}(t))$: geodesic flow on $SM$. 
    \item $C^\infty(A; B)$: space of smooth sections of a bundle $B\stackrel{\pi}{\to} A$, that is, a smooth map $f:A\to B$ such that $f(x) \in \pi^{-1}(x)$ for every $x\in A$.
    \item $d_g$: boundary distance function of a metric $g$, defined on $\partial M\times \partial M$. 
    \item $\tau:SM\to \Rm$: first exit time of the geodesic $\gamma_{x,v}(\cdot)$ out of $M$. 
    \item $\SS_g$: scattering relation of a metric $g$.
    \item $X$: geodesic vector field on $SM$.
    \item $\I$: ray transform over functions on $SM$.
    \item $I_0$: restriction of $\I$ to functions on $M$.
    \item $I_\perp h := \I [X_\perp h]$: restriction of $\I$ to solenoidal one-forms in two dimensions
    \item $\I_{A,\Phi}$ or $\I_{\nabla,\Phi}$: transform with connection $\nabla$ (associated with connection one-form $A$) and Higgs field $\Phi$ over sections of a bundle $E\to SM$.
    \item $I_{A,\Phi,0}$: restriction of $\I_{A,\Phi}$ to $C^\infty(M; E)$.
    \item $C_{A,\Phi}$: scattering data of the pair $(A,\Phi)$
\end{itemize}

\np

\subsection{Main problems}\label{sec:mainpbs}

\noindent\begin{minipage}{0.6\textwidth} 
We fix $(M,\partial M,g)$ a Riemannian manifold with boundary and $\G$ the set of all geodesics through $M$ and $\nabla$ the Levi-Civita connection. The manifold $(M,g)$ has a unit tangent bundle 
\begin{align}
    SM = \{(x,v)\in TM,\ |v|^2_{g(x)} = 1\}
    \label{eq:SM}
\end{align}
with inward ($+$) and outward ($-$) boundaries 
\[ \partial_\pm SM = \{(x,v)\in SM,\ x\in \partial M, \pm g(v,\nu_x) >0\},  \]
and where the geodesic flow $\varphi_t\colon SM\to SM$ is well-defined, with infinitesimal generator the geodesic vector field $X = \frac{d}{dt} \varphi_t(x,v)|_{t=0}$. 
\end{minipage}
\hfill%
\begin{minipage}{0.4\textwidth}
    \centering
    \includegraphics[width=0.95\linewidth]{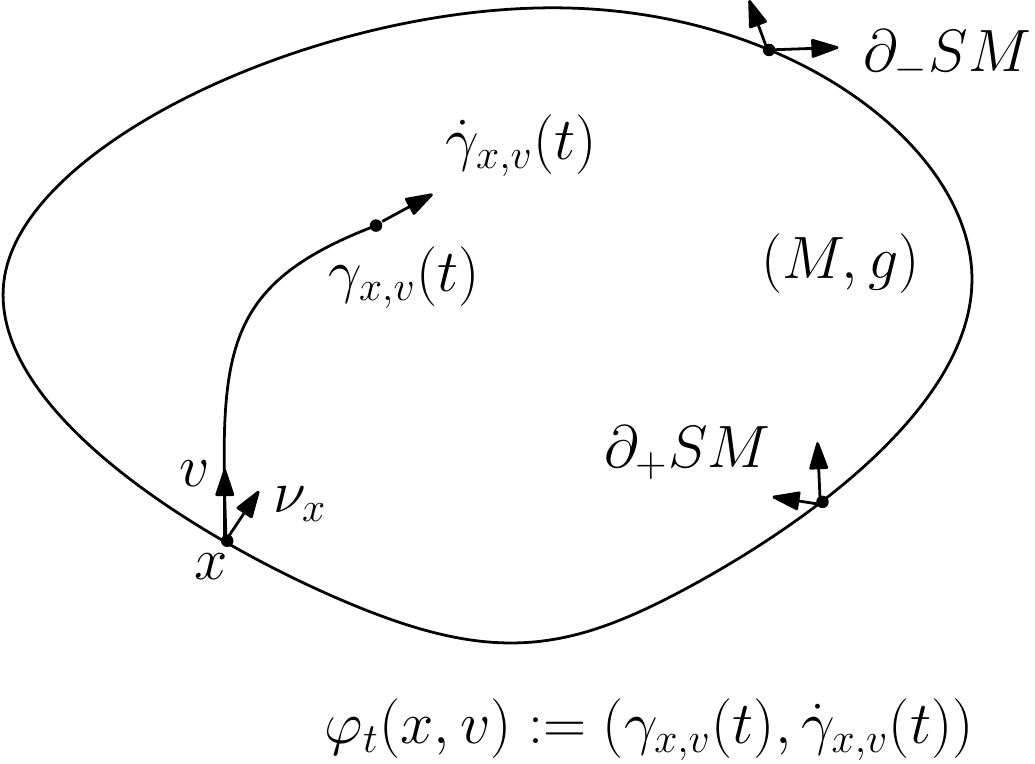}
\end{minipage}%

Given $(x,v) \in SM$, we denote $\tau(x,v)$ the first time $t\ge 0$ for which $\gamma_{x,v} (t) \in \partial M$, and we call $M$ {\bf non-trapping} if $\sup_{SM} \tau$ is finite. We say that $\partial M$ is {\bf strictly convex} if the second fundamental form is positive definite. 

In what follows, symmetric (covariant) tensors of degree $m\ge 0$ will be denoted $S^m(T^* M)$. We will restrict our attention to smooth metrics, unless otherwise explicitly stated.

\subsubsection{Reconstruction of functions, metrics and tensor fields} \label{sec:manifolds}

Given two boundary points $(x,x')\in \partial M\times \partial M$, we define the boundary distance 
\[ d_g(x,x') := \inf_{\gamma} \int |\dot\gamma(t)|_{g(\gamma(t))}\ dt, \]
where the infinimum is taken over all curves in $M$ with endpoints $x,x'$. This defines a {\bf boundary distance function} $d_g\colon\partial M\times \partial M\to [0,\infty)$. We also define the {\bf scattering relation} $\SS_g\colon\partial_+ SM \to \partial_- SM$, given by $\SS_g(x,v) = \varphi_{\tau(x,v)} (x,v)$. 

    Both maps above have a natural invariance: if $\psi\colon M\to M$ is a diffeomorphism fixing every boundary point of $M$, then $d_{\psi^* g} = d_g$ and $\SS_{\psi^* g} = \SS_g$. This invariance is written as an equivalence relation: $g \rel g'$ iff there exists $\psi\colon M\to M$ diffeomorphism fixing $\partial M$ such that $g' = \psi^* g$. We can now formulate three non-linear inverse problems:

\begin{problem}[Boundary, Lens and Scattering Rigidity]\label{pb1} Given $(M,g)$ a Riemannian manifold with boundary:
    \begin{description}
	\item[Boundary Rigidity:] Does $d_g$ determine $g$ modulo $\rel$?
	\item[Lens Rigidity:] Does $(\tau|_{\partial_+ SM},\SS_g)$ determine $g$ modulo $\rel$?
	\item[Scattering Rigidity:] Does $\SS_g$ determine $g$ modulo $\rel$? 
    \end{description}    
\end{problem}
In this article, we will not discuss Scattering Rigidity. In addition, it is well-known that Lens Rigidity is equivalent to Boundary Rigidity for simple manifolds, while Lens Rigidity is a more natural setting in general. 

On to the linear problem, fixing $f$ a symmetric $m$-tensor, the {\bf geodesic X-ray transform} $If\colon\G\to \Rm$ is defined by 
\begin{align}
    If (\gamma) = \int f_{\gamma(t)} (\dot\gamma(t)^{\otimes m})\ dt, \qquad \gamma\in \G.
    \label{eq:Xray}
\end{align}
Such a linear transform has a natural kernel for $m\ge 1$, namely: if $h$ is an $m-1$-tensor vanishing at $\partial M$, and $\sigma$ denotes symmetrization, then $I (\sigma \nabla h) \equiv 0$. This kernel is therefore made of so-called {\em potential tensors}, and we write $f\relTT f'$ iff they differ by a potential tensor field. $\relTT$ is an equivalence relation. In general, the X-ray transform of a function $f\colon SM\to\Rm$ can be defined as
\begin{equation}
\label{eq:SM-XRT}
\I f(\gamma)
=
\int f(\gamma(t),\dot\gamma(t))\ dt.
\end{equation}
This can be seen as a generalization of~\eqref{eq:Xray}; see Section~\ref{sec:tf-sh}.

\begin{problem}[Tensor Tomography (TT($m$))]\label{pb2} Does $If$ determine $f\in S^m(T^* M)$ modulo $\relTT$?  If $m=0$, does $If$ determine $f$?  
\end{problem}

Problem \ref{pb2} for $m = 0$ and $m=2$ arises as a linearization of Problem \ref{pb1}. When TT($m$) is true for $m\ge 1$, we also say that $I$ is {\em solenoidal-injective} (or in short, {\em s-injective}), or injective over solenoidal tensors. This is because by virtue of Sharafudtinov's decomposition, every $m$-tensor $f$ with $L^2$ components is $\relTT$-equivalent to a unique solenoidal tensor field $f^s$ (i.e., satisfying $\delta f^s = 0$ with $\delta$ the formal adjoint of $-\sigma \nabla$), satisfying a continuity estimate of the form $\|f^s\|_{L^2}\le C \|f\|_{L^2}$ for some constant $C(M,m)$.

\begin{figure}[htpb]
    \centering
    \includegraphics[height=0.2\textheight]{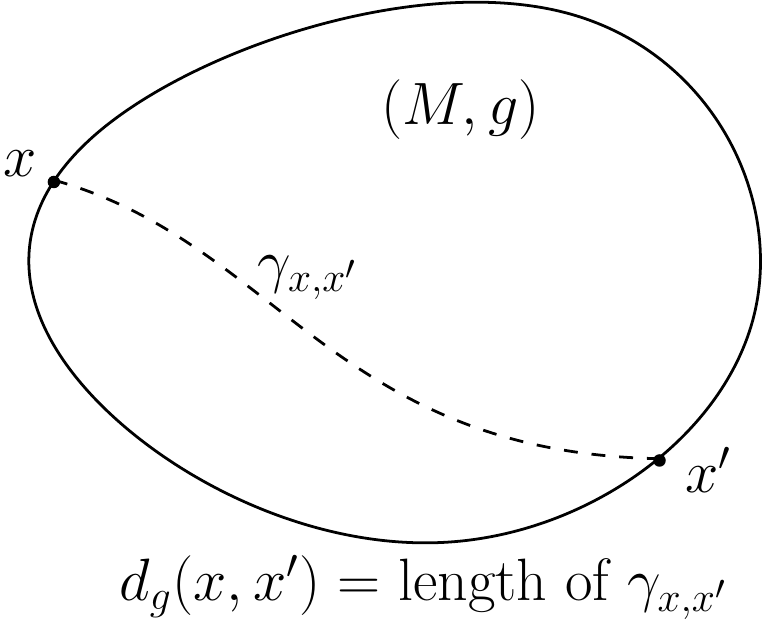} \; \;\; \;\;
    \includegraphics[height=0.2\textheight]{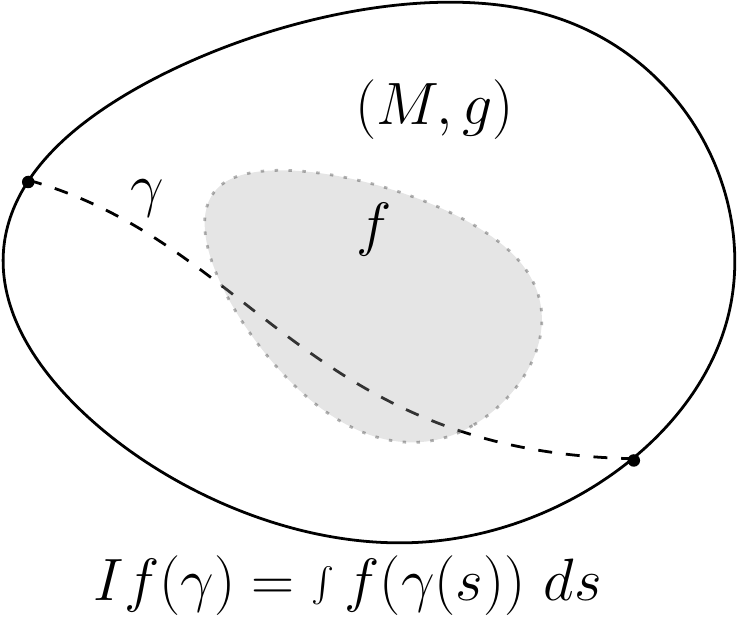}
    \caption{Settings for Problems \ref{pb1} (left) and \ref{pb2} (right).}
    \label{fig:pb1}
\end{figure}

\subsubsection{Reconstruction of connections, Higgs fields, and sections of bundles} \label{sec:bundles}

Now fix an $n$-dimensional vector bundle $E \stackrel{\pi}{\to} M$ and $(A,\Phi)$ a (connection, Higgs field) pair on this bundle, see also Section \ref{sec:ConnectionsHiggs} below. In a local trivialization, $A$ is an $n\times n$ matrix of one-forms and $\Phi$ is an $n\times n$ matrix of functions, and such quantities allow to lift any path $c(t)$ on $M$ into a path $\tilde c(t)$ on $E$ (in the sense that $\pi(\tilde c(t)) = c(t)$ for every $t$) by solving the ODE
\begin{align*}
    \frac{d\tilde c}{dt} + \left( A_{c(t)}(\dot c(t)) + \Phi(c(t))\right) \tilde c(t) = 0.
\end{align*}
If $(x,v)\in \partial_+ SM$, and let $S$ belong to the fiber above $x$. Assuming that the geodesic $\gamma_{x,v}$ exits $M$ for the first time at $\tau>0$ with $x' = \gamma_{x,v}(\tau) \in \partial M$, then the solution $\tilde \gamma_{x,v}$ of the ODE above with curve $c = \gamma_{x,v}$, augmented with the initial condition $\tilde \gamma_{x,v}(0) = S$ allows to uniquely ``parallel-transport'' the state $S$ to the state $\tilde \gamma_{x,v} (\tau)$ above $x'$, which we denote $C_{A,\Phi}(x,v) S$. A natural question is to ask whether the {\bf scattering data} (or {\bf non-abelian ray transform}) $C_{A,\Phi} (x,v)S$, known for all $(x,v,S)\in \partial_+ SM\times \Cm^n$, determines the pair $(A,\Phi)$. To formulate this problem, we first rule out a natural obstruction. 

We write $(A,\Phi) \relNA (B,\Psi)$ if there exists $Q\in C^\infty(M,GL(n,\Cm))$ with $Q|_{\partial M} = Id$, such that $B = Q^{-1} d Q + Q^{-1} A Q$ and $\Phi = Q^{-1} \Psi Q$. When this is true, it is easy to see that $C_{A,\Phi}(x,v) S = C_{B,\Psi}(x,v) S$, since in this case, if $\tilde \gamma_{x,v}$ is the $(A,\Phi)$-lift of $\gamma_{x,v}$, then $Q \tilde \gamma_{x,v}$ is the $(B,\Psi)$-lift of $\gamma_{x,v}$, and both lifts agree at both endpoints. 

\begin{problem}[Non-abelian X-ray transform]\label{pb3} Does $C_{A,\Phi}$ determine $(A,\Phi)$ modulo $\relNA$? 
\end{problem}

In the case $n=1$, it is easy to see that $C_{A,\Phi} (x,v) 1 = \exp (\I[A + \Phi](x,v))$ so that the problem is a usual X-ray transform. In the case $n=3$, problem \ref{pb3} also applies to Neutron Spin Tomography \cite{Desai2016}, a case where $A=0$ and where $\Phi$, valued in the Lie algebra ${\mathfrak so}(3)$, models the unknown magnetic field. 

\takeout{
\J{This is the usual X-ray transform if $A$ and $\Phi$ are real. If not, we are blind to multiples of $2\pi i$. This is not much of an issue with full data, though. Do we want to elaborate?}
}

The linear counterpart of Problem \ref{pb3} is as follows: let $E,A,\Phi$ as above, and fix $m$ a tensor order. If $f$ a section of $E\otimes S^m (T^*M)$ (an $E$-valued symmetric $m$-tensor), for $\gamma\in \G$, the {\bf attenuated X-ray transform}\footnote{Or X-ray transform with connection and Higgs field.} $I_{A,\Phi}f (\gamma)$ is the integral over $\gamma$ of all values of $f$ above each point of $\gamma$ (paired $m$ times with $\dot \gamma$), parallel-transported to a common point via $(A,\Phi)$. If $m\ge 1$, this problem has a natural obstruction and to point it out, it is natural to view the transform as defined over sums of $m$-tensor/$m-1$-tensor denoted by $f_m + f_{m-1}$: we say that $f_m + f_{m-1} \relAT f'_m + f'_{m-1}$ if there exist an $m-1$ tensor $p$ vanishing at $\partial M$ such that $f_m = f'_m + \sigma \nabla p + A p$ and $f_{m-1} = f'_{m-1} + \Phi p$. Whenever $f_m + f_{m-1} \relAT f'_m + f'_{m-1}$, we have that $I_{A,\Phi}(f_m+f_{m-1}) = I_{A,\Phi}(f'_m+f'_{m-1})$. The natural question is therefore:

\begin{problem}[Attenuated X-ray transform]\label{pb4} Does $I_{A,\Phi} (f_m+f_{m-1})$ determine $(f,h)$ modulo $\relAT$? If $m=0$ ($f$ is a section of $E$), does $I_{A,\Phi} f$ determine $f$?
\end{problem}

\begin{figure}[htpb]
    \centering
     \includegraphics[height=0.2\textheight]{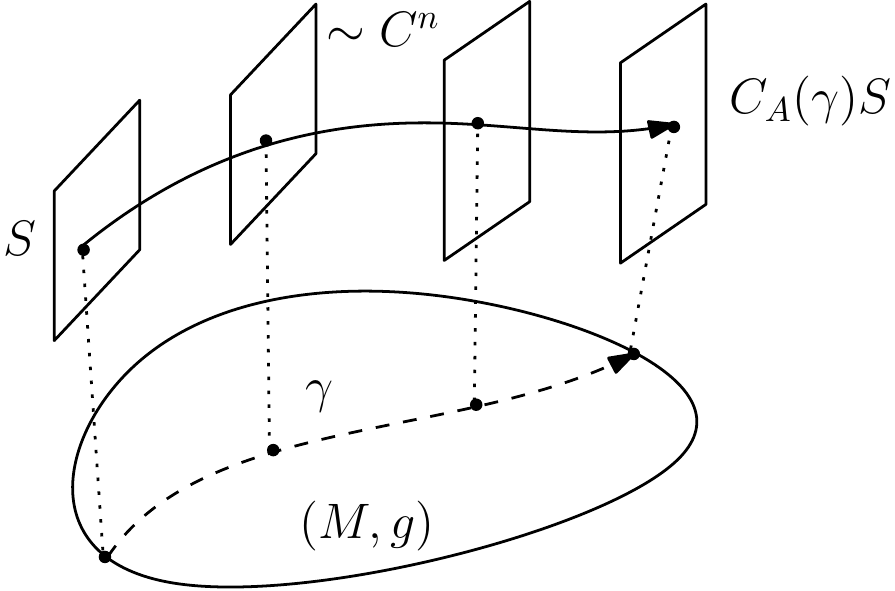} \; \;\; \;\;
    \includegraphics[height=0.2\textheight]{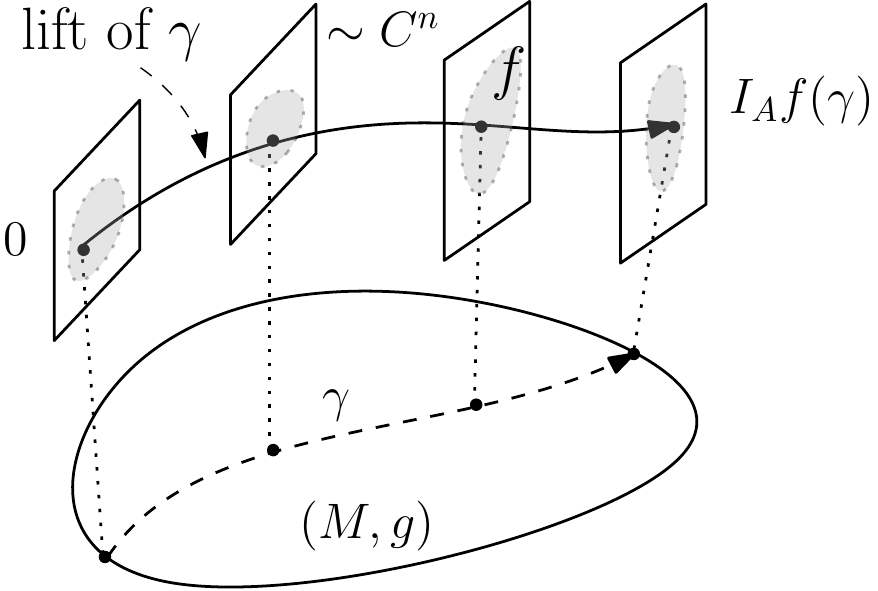}    
    \caption{Settings for Problems \ref{pb3} (left) and \ref{pb4} (right).}
    \label{fig:pb3}
\end{figure}

\subsection{The inverse problems agenda in a geometric context}

For each one of Problems \ref{pb1}--\ref{pb4}, one may ask the typical inverse problems questions: 
\begin{itemize}
    \item[$(i)$] Is the operator injective, modulo the natural obstructions?
    \item[$(ii)$] If yes, in what topology is the inverse continuous?
    \item[$(iii)$] How to explicitly and efficiently invert the operator? 
    \item[$(iv)$] How to characterize the range of the operator?
    \item[$(v)$] In the presence of noisy data, what is a proper regularization approach and how is it statistically optimal?	
\end{itemize}

The answers to questions $(i)-(v)$ strongly depend on the underlying geometric features of the manifold, the geometry and topology of $M$ (namely, the presence or absence of conjugate points\footnote{Two points $x,y$ are {\em conjugate} along a geodesic $\gamma$ if there exists a Jacobi field over $\gamma$ vanishing at both $x$ and~$y$.} and/or trapped geodesics\footnote{The {\em trapped set} of $(M,g)$ is the set of points $(x,v)\in SM$ such that the length of $\gamma_{x,v}\cap M$ is infinite, where $\gamma_{x,v}$ is the unique maximal geodesic satisfying $\gamma(0)=x$ and $\dot\gamma(0) = v$.}), the structure of the connection and Higgs field (rank, structure group, etc.), the dimension of the manifold (including significant differences in the landscapes of results between dimension two, and higher dimensions), and the presence of weights in the transforms. 

Many answers are positive in the case of homogeneous spaces \cite{Funk1916,Radon1917,Helgason1999} and in the case of {\bf simple\footnote{A Riemannian manifold $(M,g)$ is {\em simple} if it is non-trapping, $\partial M$ is strictly convex and $M$ contains no conjugate points.} geometries}: in the case of simple surfaces, it is known that such surfaces are boundary distance rigid \cite{Michel1981,Pestov2005}, and that ray transforms are injective over functions \cite{Mukhometov1977} and solenoidal tensors of any order \cite{Paternain2011a}, also when one includes many types of connections and Higgs fields \cite{Vertgeim1992,Novikov2002b,Eskin2004,Sharafutdinov2000,Finch2001,Paternain2012,Monard2016d}; for higher-dimensional simple manifolds, generic injectivity and stability results are known \cite{Stefanov2005,Zhou2017}, injectivity of X-ray transforms is known over functions and vector fields, and for higher-order tensor fields, the result is true under stronger assumptions on the geometry \cite{Romanov1974,Mukhometov1975,Anikonov1978,Anikonov1997,Sharafutdinov1994,Pestov1988,Sharafutdinov2007,Sharafutdinov1997}. 

In geometries with {\bf conjugate points}, another separation between two- and higher-than-three dimensions occurs: in two dimensions, conjugate points on surfaces unconditionally destroy stability of X-ray transforms \cite{Stefanov2012a,Monard2013b} while the question needs to be refined in higher dimensions and exhibits a tradeoff between the order of conjugate points considered and the dimension of the manifold \cite{Holman2015}. In fact, there is more at play in higher dimensions: the mere existence of a foliation by strictly convex hypersurfaces allows to prove global injectivity and stability \cite{Uhlmann2015,Stefanov2014,Paternain2016a}. Such a criterion allows for conjugate points and some form of trapped geodesics as well, and as such shifts the focus to the following question: which manifolds admit strictly convex foliations? Injectivity questions remain open on surfaces with conjugate points, except for the case of circularly symmetric ones, where injectivity over solenoidal tensor fields is known to hold \cite{Sharafutdinov1997}, and injectivity over piecewise constant functions holds~\cite{Ilmavirta2017}.

In geometries with {\bf trapped geodesics}, one may easily construct counterexamples to injectivity, and thus one must assume some thing about the trapped set. Under the crucial assumption that the trapped set be {\bf hyperbolic} for the geodesic flow (a condition which is always true on manifolds with negative sectional curvatures), injectivity and stability can be restored in many cases \cite{Guillarmou2014,Guillarmou2014a,Guillarmou2015}.

See also the recent topical reviews \cite{Uhlmann2017} on Problems \ref{pb1}--\ref{pb2}, and \cite{Paternain2012b} on Problems \ref{pb3}--\ref{pb4}.

\subsection{Outline}

The remainder of the article is organized as follows. 

We devote section~\ref{sec:geometricSetting} to introductory material and notation, describing the geometrical framework needed to discuss integral geometry. This includes basic geometry of the sphere bundle (on which the geodesic flow lives) and natural operators on it, transport equations, tensor fields, conjugate points, and two-dimensional structure. We also briefly discuss trapping (section~\ref{sec:trapping}) and connections (section~\ref{sec:ConnectionsHiggs}). 

Sections~\ref{sec:Pestov}--\ref{sec:nonlinear} then present results, arranged by methods.
\begin{itemize}
    \item[-] In section~\ref{sec:Pestov} we discuss energy estimates known as Pestov identities. We give the fundamental commutators in section~\ref{sec:comDerSM} before deriving a Pestov identity on simple manifolds (section~\ref{sec:simple}). We then extend the methods to other geometrical settings (section~\ref{sec:PestovOther}), connections and Higgs fields (section~\ref{sec:PestovConnection}) and generalized geodesic flows (section~\ref{sec:PestovMagnetic}).
    \item[-] In section~\ref{sec:reconstruction} we discuss explicit approaches to injectivity and inversion in two dimensions.
    \item[-] Section~\ref{sec:invariant} covers invariant distributions and their relation to tensor tomography, including their construction via iterated Beurling series (section~\ref{sec:Beurling}).
    \item[-] In section~\ref{sec:microlocal} we discuss applications of microlocal analysis to integral geometry. This includes analysis of cases with and without conjugate points, geometry of Fourier Integral Operators, and general families of curves.
    \item[-] In section~\ref{sec:convexity} we turn to layer stripping arguments and methods based on considerations of convexity. These rely on a combination of local support theorems and a global foliation of the manifold. We discuss different methods to obtain local support theorems.
    \item[-] While the results up to this point have been mainly linear, we discuss relations between linear and non-linear problems in section~\ref{sec:nonlinear}.
\end{itemize}
Section~\ref{sec:open} concludes with a small collection of open questions.

\np
\section{Geometric setting and tools}\label{sec:geometricSetting}

A natural reformulation of integral geometric problems involving the integration of objects along curves, is by viewing the integrand as a source term for a 'geometric' transport equation posed on the tangent bundle, and apply various PDE methods (energy identities, etc.) to that equation. Such ideas are not new and form the basis of V.A. Sharafutdinov's pioneering monograph \cite{Sharafutdinov1994}. The main difference of our presentation (which largely follows \cite{Paternain2015}) is in how to represent integrands of tensor field type as natural objects to be integrated over a flow in phase space: in \cite{Sharafutdinov1994}, a section of $S^m(T^*M)$ is identified with a so-called semibasic tensor field on $TM$ (i.e., covariant in horizontal directions and contravariant in vertical ones in a certain sense). Here tensor fields are regarded as scalar functions on the sphere bundle $SM$, whose tensorial nature is encoded in the finite expansions in spherical harmonics on the unit tangent fibers. This latter identification somewhat allows to bypass the proliferation of indices as one increases the tensor order. 

\subsection{The geometry of the unit sphere bundle}

\subsubsection{Vertical and horizontal vectors on the sphere bundle}
\label{sec:v/h}

Given $(M,g)$ a Riemannian manifold, local charts on the tangent bundle $TM$ may be written as $(x,y)$, where the tangent vector looks locally like $y = y^i \frac{\partial}{\partial x_i}$. The double tangent bundle $T(TM)$ admits a horizontal-vertical splitting which embodies whether one is differentiating vertically (along a fiber), or horizontally (along the base, while keeping a tangent vector ``fixed''). Horizontal directions depend on the Riemannian metric, while vertical ones only on the smooth structure. 

Specifically, the {\bf vertical subbundle} of $T(TM)$ is defined so that the fiber $V(x,y)$ at $(x,y)\in TM$ is $V(x,y)=\ker(d_{(x,y)}\pi)$, where $\pi\colon TM\to M$ is the canonical projection. To define the {\bf horizontal subbundle}, we define a connection map $K\colon T(TM)\to TM$ fiber by fiber. Take any $\xi\in T_{(x,y)} TM$ and $\sigma\colon(-\epsilon,\epsilon)\to TM$ a curve with $\sigma(0)=(x,y)$ and $\dot\sigma(0)=\xi$. We may write this curve as $\sigma(t)=(\gamma(t),Z(t))$, where $\gamma$ is a curve on $M$ and $Z$ a vector field along it. Upon defining $K_{(x,y)} \xi := (\nabla_\gamma Z)(0)$, the fiber of the horizontal bundle at $(x,y)$ is then $H(x,y):=\ker(K_{(x,y)})$. Each fiber of the $T(TM)$ then decomposes as 
\[ T_{(x,y)} TM=H(x,y)\oplus V(x,y), \qquad (x,y)\in TM. \] 
In local coordinates $V(x,y) = \langle \frac{\partial}{\partial y^i},\ 1\le i\le n\rangle$ while $H(x,y) = \langle \delta_{x_i} := \frac{\partial}{\partial x_i} - \Gamma_{ij}^k y^j \frac{\partial}{ \partial y^k},\ 1\le i\le n\rangle$. With the splitting above, the maps $d_{(x,y)} \pi|_{H(x,y)}\colon H(x,y)\to T_xM$ and $K_{(x,y)}|_{V(x,y)}\colon V(x,y)\to T_xM$ are linear isomorphisms, allowing us to freely identify horizontal and vertical vectors on $T_{(x,y)} TM$ with vectors on $T_xM$. These isomorphisms become isometries (and the splitting, orthogonal) upon introducing the {\bf Sasaki metric} at $(x,y)\in TM$ defined by
\begin{align*}
    \langle \xi, \eta\rangle_{x,y} := g_x (K_{(x,y)}(\xi), K_{(x,y)}(\eta)) + g_x (d_{(x,y)}\pi (\xi), d_{(x,y)} \pi (\eta)), \qquad \xi,\eta\in T_{(x,y)} TM,
\end{align*}
or equivalently in coordinates, with $\xi = X^i \delta_{x_i} + Y^i \frac{\partial}{\partial y^i}$ and $\eta = \widetilde{X}^i \delta_{x_i} + \widetilde{Y}^i \frac{\partial}{\partial y^i}$,
\begin{align*}
    \langle \xi, \eta\rangle_{x,y} = g_{ij} X^i \widetilde{X}^j + g_{ij} Y^i \widetilde{Y}^j.
\end{align*}

The unit sphere bundle $SM$ of a Riemannian manifold $M$ is the subbundle of $TM$ consisting of unit tangent vectors of unit length:
\begin{equation}
SM = \{(x,v);\ x\in M,\ v\in T_xM,\ \abs{v}=1\}.
\end{equation}
There the horizontal-vertical splitting becomes
\begin{align}
    T_{x,v} SM = \Rm X(x,v) \oplus {\cal H}(x,v) \oplus {\cal V}(x,v), \qquad (x,v)\in SM, 
    \label{eq:SMsplitting}
\end{align}
where $\Rm X \oplus {\cal H}(x,v) = H(x,v)$ and ${\cal V}(x,v) = \ker d_{x,v} (\pi|_{SM})$. Elements of ${\cal H}(x,v)$ and ${\cal V}(x,v)$, when identified as vectors of $T_xM$, are both orthogonal to $v$, so smooth sections of ${\cal H}$ and ${\cal V}$ can be isomorphically identified with smooth sections in $\Z := C^\infty(SM,N)$, where we define the bundle $N\to SM$ by
\begin{align}
    N := \bigcup_{\mathclap{(x,v)\in SM}} \{v\}^\perp, \qquad \{v\}^\perp := \{w\in T_x M,\ g_x(w,v) = 0\}.
    \label{eq:N}
\end{align}

According to the decomposition \eqref{eq:SMsplitting}, the total gradient of a scalar function $u$ on the sphere bundle $SM$ consists of three Sasaki-orthogonal components: the geodesic derivative $Xu$ (scalar-valued), and the vertical and horizontal gradients $\vgrad u$ and $\hgrad u$ (each identified with elements of $\Z$). In particular, we have two differential operators
\begin{align*}
    \vgrad\colon C^\infty(SM)\to \Z, \qquad \hgrad\colon C^\infty(SM)\to \Z.    
\end{align*}
Roughly speaking, the vertical gradient of $u(x,v)$ is the gradient of with respect to $v$ and the horizontal gradient is the component of the gradient with respect to $x$ orthogonal to $v$. If $\dim(M)=2$, then these two gradients can be regarded as scalars as done in Section~\ref{sec:2d-SM} below. 

The adjoints are the vertical and horizontal divergences which we denote $-\vdiv$ and $-\hdiv$ with the following mapping properties
\begin{align*}
    \vdiv\colon \Z\to C^\infty(SM), \qquad \hdiv \colon \Z\to C^\infty(SM).
\end{align*}
The geodesic vector field also acts on $\Z$ by covariant differentiation along the geodesic flow.

\subsubsection{The X-ray transform and transport equations on $SM$}\label{sec:transport}

In the framework just described, given $F\in L^2(SM)$, the X-ray transform of $F$ defined on \eqref{eq:Xray} can be viewed as the inward restriction $u|_{\partial_+ SM}$ of the solution $u$ to a transport problem
\begin{align}
    Xu = -F \qquad (SM), \qquad u|_{\partial_- SM} = 0.
    \label{eq:transport}
\end{align}
With this setting in mind, injectivity questions and inversion formulas can be tackled by classical PDE methods on manifolds: for instance, injectivity over functions means: if $Xu = -F(x)$ throughout $SM$ and $u|_{\partial SM} =0$, does this imply $F = 0$?

Similarly, to address tensor tomography, there is a natural way to identify a symmetric $m$-tensor field $f$ on $M$ with a scalar field $\ell_m f$ on $SM$, given by a mapping 
\begin{align}
    \begin{split}
	\ell_m &\colon C^\infty(S^m (T^* M)) \to C^\infty(SM), \qquad \ell_m\colon L^2(S^m(T^* M)) \to L^2(SM) \\
	\ell_m f (x,v) &= f_x (v,\dots,v), \qquad (x,v)\in SM.	
    \end{split}
    \label{eq:ellm}
\end{align}
Via this identification, the X-ray transform of $f$ is again given by $If := u|_{\partial_+ SM}$, where $u$ solves the transport problem \eqref{eq:transport} with right-hand side $\ell_m f$. Whenever the context allows, we will not distinguish $f$ and~$\ell_m f$.

\subsubsection{Tensor fields and spherical harmonics}\label{sec:tf-sh}

The $L^2$ space of every fiber of the sphere bundle can be decomposed into eigenspaces of the vertical Laplacian 
\[ -\vdiv\vgrad\colon C^\infty(SM)\to C^\infty(SM). \]
Namely, on each spherical fiber, the vertical Laplacian coincides with the Laplacian of the function $v\mapsto u(x,v)$ on the manifold $(S_x M,g_x)$, whose spectrum is the same as that of the spherical Laplacian $\Delta_{\Sm^{n-1}}$, given by $\lambda_m  = m(n+m-2)$ for $m=0,1,2\dots$, with eigenfunctions the spherical harmonics. The corresponding eigenspaces induce an $L^2(SM)$-orthogonal decomposition 
\begin{align}
    L^2(SM) = \bigoplus_{m=0}^\infty H_m(M), \qquad H_m := \ker (-\vdiv\vgrad - \lambda_m Id) \cap L^2(SM),
    \label{eq:orthoDecomp}
\end{align}
which on each fiber over $M$ is just the spherical harmonic decomposition in $S^{n-1}$. Let us also set $\Omega_m := H_m(M) \cap C^\infty(SM)$. Then any function $f\in L^2(SM)$ splits as $f=\sum_{m=0}^\infty f_m$ so that for almost every $x\in M$ the function $v\mapsto f_m(x,v)$ is a spherical harmonic of order $m$. The zeroth component $f_0$ of a function on the sphere bundle is the fiberwise average.  

\paragraph{Tensr fields and finite harmonic content.}
In the decomposition above, an $m$th order tensor field $f$, via its identification \eqref{eq:ellm} with $\ell_m f$, can be regarded as a function on $SM$ which only contains spherical harmonics up to order $m$ and of the same parity as $m$. Conversely, if a scalar function $u$ on $SM$ contains spherical harmonics up to a finite order $m$ and they all have the same parity, then there is a tensor field $f$ so that $\ell_m f=u$.

Since $\ell_{m+1} (\sigma\nabla h) = X (\ell_m h)$ and $If=\I (\ell_m f)$ (see equations \eqref{eq:Xray}, \eqref{eq:SM-XRT}, \eqref{eq:ellm}), the tensor tomography problem~\ref{pb2} can be recast as follows:
If $f\colon SM\to\Rm$ only contains spherical harmonics up to order $m$ and integrates to zero over all (lifted) geodesics of $M$, is there a function $h$ with spherical harmonics up to order $m-1$ so that $f=Xh$?
In terms of the transport equation~\eqref{eq:transport}, the question is whether the spherical harmonic expansion of $u$ ends at order $m-1$.

\paragraph{Decomposition of $X$.} The geodesic vector field behaves nicely with respect to the decomposition \eqref{eq:orthoDecomp}: it maps $\Omega_m$ into $\Omega_{m-1} + \Omega_{m+1}$ \cite[Proposition 3.2]{Guillemin1980a}. Hence on $\Omega_m$ we can write 
\[  X = X_+ + X_-, \quad \text{where} \quad X_\pm\colon \Omega_m \to \Omega_{m\pm 1}, \qquad (\text{convention}:\Omega_{-1} \equiv 0) \]
and such that, for $u\in \Omega_m$ and $w\in \Omega_{m+1}$ and one of them vanishes on $\partial SM$, we have
\begin{align*}
    \iip{X_+ u}{v} = - \iip{u}{X_-v}.
\end{align*}

In particular, the transport equation \eqref{eq:transport}, upon projecting onto each harmonic subspace $\Omega_k$, can be equivalently viewed as the tridiagonal system of equations
\begin{align}
    X_+ u_{m-1} + X_- u_{m+1} = - f_m, \qquad m = 0,1,\dots
    \label{eq:tridiagonal}
\end{align}

\subsubsection{Jacobi fields and conjugate points}

Given a geodesic $\gamma$ and $p,q$ two distinct points on it, we say that $p$ and $q$ are {\bf conjugate} along $\gamma$ if there exists a non-trival Jacobi field along $\gamma$ which vanishes at both $p$ and $q$. Specifically, if $p = \gamma(t_1)$ and $q = \gamma(t_2)$ for some $t_1<t_2$, there exists $J(t) \in T_{\gamma(t)}M$ a non-trivial solution of 
\begin{align*}
    D_t^2 J(t) + R(J(t), \dot\gamma(t))\dot\gamma(t) = 0, \qquad J(t_1) =0, \qquad J(t_2) = 0,
\end{align*}
where $D_t$ denotes Levi-Civita covariant differentiation and $R$ denotes the Riemannian curvature tensor. Since a pair of points can be conjugate along more than one geodesic (e.g., antipodal points on a sphere), it can be useful to keep track along which geodesic a pair of points is conjugate. A way to do this is to keep track of the tangent vectors, and to consider conjugate pairs as a subset of $SM\times SM$, see also Section \ref{sec:FIOnD}.

An equivalent definition which is more amenable to generalizing this concept to other flows, is to say that, with $\varphi_t(x,v) = (\gamma_{x,v}(t), \dot\gamma_{x,v}(t))$ denoting the geodesic flow on $SM$, the points $(x,v)$ and $\varphi_t(x,v)$ are conjugate (along the geodesic $\gamma_{x,v}$) if 
\begin{align*}
    \V(x,v) \cap d\varphi_{-t}|_{\varphi_t(x,v)} \V (\varphi_t(x,v)) \ne \{0\}.
\end{align*}
In other words, conjugate points occurs when the differential of the flow maps vertical vectors into vertical vectors. 

As we will see below, many positive results hold in the absence of conjugate points. In their presence, two-dimensional problems usually become unstable, and higher-dimensional ones require further discussion, see in particular Sections \ref{sec:microlocal} and~\ref{sec:convexity}.

\subsubsection{Additional structure in two dimensions} \label{sec:2d-SM}

In two dimensions, the unit circle bundle $SM$ admits a global framing by three global sections of $T(SM)$: a first section is the geodesic vector field $X = \frac{d}{dt}|_{t=0} \varphi_t(x,v)$; a second is the generator of the rotation group on the fibers $V = \frac{d}{dt}|_{t=0} \rho_t(x,v)$ (assuming the surface to be oriented, giving rise to a rotation-by-$\pi/2$ operator $v\mapsto v_\perp$, then $\rho_t(x,v) = (x,(\cos t)v + (\sin t) v_\perp)$); finally, their commutator $X_\perp := [X,V]$ gives the third one. Such vector fields admits the structure equations
\begin{align}
    [X,V] = X_\perp, \qquad [X_\perp,V] = -X, \qquad [X,X_\perp] = -\kappa V \qquad (\kappa\colon\text{ Gauss curvature}),
    \label{eq:structure2D}
\end{align}
encoding the whole geometry. One may define the {\em Sasaki metric} on $SM$, making $(X,X_\perp,V)$ orthonormal, and with volume form the so-called Liouville measure denoted $d\Sigma^3$. 

Locally (or globally, if $M$ is simply connected), $SM$ can be parameterized in {\em isothermal coordinates} $(x,y,\theta)$, where $g = e^{2\lambda(x,y)} (dx^2 + dy^2)$, $\theta$ is the angle between a tangent vector $v$ and $\partial_x$, namely a tangent vector $v$ sitting above $(x,y)$ has the expression $v = e^{-\lambda(x,y)}\binom{\cos\theta}{\sin\theta}$, the Liouville form reads $d\Sigma^3 = e^{2\lambda} dx\ dy\ d\theta$, and the canonical frame reads
\begin{align*}
    X &= e^{-\lambda} (\cos\theta \partial_x + \sin \theta\partial_y + (-\sin\theta \partial_x \lambda + \cos\theta \partial_y \lambda) \partial_\theta), \qquad V = \partial_\theta, \\
    X_\perp &= - e^{-\lambda} (-\sin\theta \partial_x + \cos\theta \partial_y - (\cos\theta \partial_x \lambda + \sin\theta \partial_y \lambda) \partial_\theta ).
\end{align*}

Note that in two dimensions, we identify $\hgrad u = -(X_\perp u) v^\perp$ and $\vgrad u = (Vu) v^\perp$ (smooth sections of $N$ defined in \eqref{eq:N}) with the functions $X_\perp u$ and $Vu$ (smooth functions on $SM$), respectively.

\paragraph{Jacobi fields.} The structure equations \eqref{eq:structure2D} make it convenient to compute Jacobi fields (or variations of the exponential map). For $\xi\in T_{(x,v)}(SM)$, we may decompose $d\varphi_t(\xi)$ along the frame $\{X(t), X_\perp(t), V(t)\}$ at the basepoint $\varphi_t(x,v)$ as 
\[ d\varphi_t(\xi) = \zeta_1(x,v,t) X(t) + \zeta_2(x,v,t) X_\perp (t) + \zeta_3 (x,v,t) V(t).   \]
Equations \eqref{eq:structure2D} provide us a differential system in $t$ for the coefficients $\zeta_j$ (see e.g. \cite[Section 4.2]{Merry2011}): 
\begin{align*}
    \dot\zeta_1 = 0, \qquad \dot\zeta_2 + \zeta_3 = 0, \qquad \zeta_3 - \kappa(\gamma_{x,v}(t)) \zeta_2 = 0.
\end{align*}
In particular, we may express the variation fields $d\varphi_t(X_\perp) = a X_\perp (t) - \dot a V(t)$ and $d\varphi_t(V) = -bX_\perp (t) + \dot b V(t)$ in terms of two functions $a(x,v,t)$, $b(x,v,t)$ defined for $(x,v)\in SM$ and $t\in (-\tau(x,-v), \tau(x,v))$, solving the scalar Jacobi equation
\begin{align}
    \ddot a + \kappa(\gamma_{x,v}(t)) a = \ddot b + \kappa(\gamma_{x,v}(t)) b = 0, \qquad \left[ \begin{array}{cc} a & b \\ \dot a & \dot b \end{array}
    \right](0) = \left[ \begin{array}{cc} 1 & 0 \\ 0 & 1 \end{array}
    \right] 
    \label{eq:scalarJacobi}
\end{align}
Here the function $b(x,v,t)$ is the one that detects conjugate points on $M$. Specifically, if $t>0$ is such that $b(x,v,t) = 0$, then the points $x$ and $x' = \gamma_{x,v}(t)$ are conjugate along $\gamma_{x,v}$ as the Jacobi field $J(t) = b(x,v,t) \dot \gamma_{x,v}^\perp (t) $ vanishes at both $x$ and~$x'$.

\subsection{Connections and Higgs fields}\label{sec:ConnectionsHiggs}

To set the stage similarly to Section \ref{sec:bundles}, let $E\to M$ a hermitian vector bundle\footnote{In the sense that each fiber is a vector space endowed with a hermitian inner product $\iip{\cdot}{\cdot}_E$.} over $M$. We assume that the fiber over each point is a copy of $\Cm^r$, where $r$ is called the {\bf rank} of the bundle. Let $\nabla^E$ a connection on $E$. We say that $\nabla^E$ is {\bf hermitian} (or {\bf unitary}) if the following identity is satisfied
\begin{align*}
    Y \iip{u}{u'}_E = \iip{\nabla_Y^E u}{u'}_E + \iip{u}{\nabla_Y^E u'}_E,
\end{align*}
for all vector fields $Y$ on $M$ and sections $u,u' \in C^\infty(M; E)$. Via the canonical projection $\pi:SM\to M$, such a bundle and its connection can be pulled back into a bundle 
\[ \pi^* E := \{ (x,v; e), (x,v)\in SM, e\in E_x \} \] 
over $SM$ with hermitian connection $\pi^* \nabla^E$, which is where geodesic transport equations will be naturally written\footnote{The notational distinction between $(E,\nabla^E)$ and their pullbacks will be omitted as in \cite{Guillarmou2015}.}. Following the spherical harmonic decomposition on the tangent spheres, one may still decompose an element $u\in C^\infty(SM; E)$ into a sum $u = \sum_{k=0}^\infty u_k$. 

The geodesic vector field can be viewed as acting on sections of $E$ by $\Xm u := \nabla^E_X u$ for a section $u\in C^\infty(SM; E)$, and for $f\in C^\infty(SM; E)$, this incarnation of the X-ray transform is given by $I_\nabla f := u|_{\partial_+ SM}$, where $u$ solves the transport problem 
\begin{align}
    \Xm u = - f \qquad (SM), \qquad u|_{\partial_- SM} = 0.
    \label{eq:trans2}
\end{align}
Note that in a local trivialization, the connection can be represented as a $r\times r$ matrix of one-forms $A$, and then $\Xm$ reads as $\Xm = X + A$, where $X$ acts componentwise.

One may also add a {\bf Higgs field} $\Phi$, that is to say, a smooth section of $End(E)$ such that at every $x\in M$, $\Phi_x$ is a linear operator $\Phi_x\colon E_x\to E_x$. $\Phi$ is called a {\bf skew-hermitian Higgs field} if the endomorphisms $\Phi_x$ on each fiber are skew-hermitian. The Higgs field is the ``matrix'' generalization of a position-dependent attenuation coefficient, and given $f\in C^\infty(SM; E)$, we can then define the attenuated transform $\I_{\nabla,\Phi} f = u|_{\partial_+ SM}$, where $u\in C^\infty(SM; E)$ solves the transport problem
\begin{align}
    (\Xm + \Phi) u = - f \qquad (SM), \qquad u|_{\partial_- SM} = 0.
    \label{eq:trans3}
\end{align}

\subsection{Trapped geodesics and the hyperbolicity condition}\label{sec:trapping}

So far, all metrics considered assumed that all geodesics exit the domain $M$ in finite time. If this is no longer the case, we say that the manifold is {\em trapping}, and define the incoming $(-)$ and outgoing $(+)$ tails 
\begin{align}
    \Gamma_\pm := \{(x,v)\in SM, \tau (x, \mp v) = +\infty \}, 
    \label{eq:tails}
\end{align}
as well as the {\em trapped set} $K = \Gamma_- \cap \Gamma_+$, invariant by the flow, and consisting of those points which are trapped both forward and backward in time (in general, $(x,v)\in \Gamma_-$ is trapped forward and $(x,v)\in\Gamma_+$ is trapped backward). Geodesics cast from $\partial_+ SM \cap \Gamma_-$ never exit the domain and as such cannot be detected. Moreover, the data would blow up at such geodesics. 

Without specific assumption on $K$, a trapped set can easily generate an infinite-dimensional kernel for an X-ray transform, as the following example suggests: glue a hemisphere on top of a Euclidean cylinder to make it simply connected. Any function supported on the cylinder, circularly symmetric, integrating to zero along the longitudinal direction is in the kernel of the X-ray transform.

A dynamical condition which allows to produce positive answers on manifolds with non-trivial topology \cite{Guillarmou2014,Guillarmou2016,Guillarmou2017,Guillarmou2015}, is to assume that the trapped set be {\bf hyperbolic} for the geodesic flow. Namely, one may define the {\bf stable bundle} $E_- \subset T_{\Gamma_-} SM$ and {\bf unstable bundle} $E_+\subset T_{\Gamma_+} SM$ such that 
\begin{align*}
    \forall (x,v)\in \Gamma_-, &\qquad \forall t>0, \qquad \forall w\in E_- (x,v),\qquad \|d\varphi_t (x,v) w\| \le C e^{-\gamma t} \|w\|, \\
    \forall (x,v)\in \Gamma_+, &\qquad \forall t<0, \qquad \forall w\in E_+ (x,v),\qquad \|d\varphi_t (x,v) w\| \le C e^{\gamma t} \|w\|,
\end{align*} 
with $C,\gamma$ some uniform positive constants. Upon defining the same bundles over $K$ by restriction, $E_s := E_-|_{T_K SM}$ and $E_u := E_+|_{T_K SM}$, we say that the set $K$ is hyperbolic if and only if 
\begin{align}
    \forall (x,v)\in K, \qquad T_{(x,v)} SM = \Rm X(x,v) \oplus E_s(x,v) \oplus E_u (x,v). 
    \label{eq:hyperbolic}
\end{align}
The assumption of hyperbolic trapping has the following advantages: 

$(i)$ The Liouville volume of $\Gamma_+ \cup \Gamma_-$ is zero, and so is the measure of $\Gamma_+ \cup \Gamma_-$ on the boundary $\partial SM$. In particular, this gives hope to make the X-ray transform valued in some $L^p (\partial_+ SM)$-spaces. 

$(ii)$ Solving transport equations of the form $Xu = -f$ on $SM$ may develop singularities even when $f$ is smooth, however there is good control over the created singularities. Namely, upon defining the dual bundles $E_\pm^* \subset T^*_{\Gamma_\pm} SM$ by
\begin{align*}
    E_-^* (E_- \oplus \Rm X) = 0, \qquad E_+^* (E_+\oplus \Rm X) = 0, 
\end{align*}
then for each $f\in C^0(SM)$, it is established in \cite[Section 4.2]{Guillarmou2014a} that the boundary value problem \eqref{eq:transport} has a unique solution in $L^1(SM) \cap C^0(SM\backslash \Gamma_-)$ given by $u = R_+ f$ with 
\begin{align*}
    R_+ f (x,v) := \int_0^{\tau(x,v)} f(\varphi_t(x,v))\ dt, 
\end{align*}
and that, if $f\in C_c^\infty(SM)$, then $WF(R_+ f) \subset E_-^*$. Then in the presence of trapping, the X-ray transform $\I$ may be defined as
\begin{align*}
    \I f:= (R_+ f)|_{\partial_+ SM \backslash \Gamma_-}, \qquad \I\colon C^0(SM) \to C^0(\partial_+ SM\backslash \Gamma_-)    
\end{align*}
and extends as a bounded operator 
\begin{align*}
    \I \colon L^p(SM) \to L^2(\partial_+ SM), \qquad \forall p> 2, 
\end{align*}
see \cite[Lemma 3.4, Prop. 2.4]{Guillarmou2014a}. Similary, the transform over bundles $\I_{\nabla,\Phi}$ also makes sense outside $\Gamma_-$, and the results established in \cite{Guillarmou2015} are described in that setting.

\np
\section{Pestov identities} \label{sec:Pestov}

\subsection{Commutators of derivatives on the sphere bundle}
\label{sec:comDerSM}

Recall the definition of the natural derivatives $X$, $\hgrad$, $\vgrad$, $\hdiv$ and $\vdiv$ on the sphere bundle $SM$ as introduced in Section~\ref{sec:v/h}. 

The geodesic vector field $X$ acts as a differential operator $X\colon C^\infty(SM)\to C^\infty(SM)$ by
\begin{equation}
    Xu(x,v) = \partial_tu(\varphi_t(x,v))|_{t=0},
\end{equation}
where $\varphi_t$ is the geodesic flow. The same definition can be used to define also the operator $X\colon\Z\to\Z$, when one uses the covariant derivative along the flow. This gives rise to two incarnations of $X$ given by 
\begin{align*}
    X&\colon C^\infty(SM)\to C^\infty(SM), \qquad X\colon \Z\to \Z,    
\end{align*}
and we will not distinguish between the two in notation. In addition, we define the curvature operator $R\colon\Z\to\Z$ by $R(x,v)w=R_x(w,v)v$, where the $R_x$ on the right-hand side is the Riemann curvature tensor at $x\in M$.

The starting point of deriving Pestov identities is the following commutator formulas. 

\begin{lemma}[{\cite[Lemma 2.1]{Paternain2015}}]
\label{lma:hd-commutator}
The following commutator formulas hold on $C^\infty(SM)$ and $\Z$:
\begin{equation}
\begin{split}
[X,\vgrad]&=-\hgrad, \qquad [X,\vdiv]=-\hdiv, \\
[X,\hgrad]&=R\vgrad, \qquad [X,\hdiv]=-\vdiv R, \qquad \text{and} \\
\hdiv\vgrad-\vdiv\hgrad&=(n-1)X,
\end{split}
\end{equation}
\end{lemma}
To emphasize the commutator nature of the third formula in Lemma~\ref{lma:hd-commutator}, one can write it as
\begin{equation}
\overset{[\tt{h}}{\mbox{\rm div}}
\overset{\!\!,}{\vphantom{\nabla}\,}
\overset{\tt{v}]}{\nabla}
=
(n-1)X.
\end{equation}
We place the commutator symbols around the labels `h' and `v' (they are commuted), not around `$\rm div$' and `$\nabla$' (they remain in the same order).

In two dimensions the horizontal and vertical gradients can be considered as vector fields (globally if the underlying manifold is orientable).
The corresponding commutator formulas were given in equation~\eqref{eq:structure2D}.

In addition to commuting operators, we need to integrate by parts.
Let us denote the inner product of $u,v\in L^2(SM)$ by $\iip{u}{v}$ and similarly for $L^2$ sections of $N$.
If $u,v\in C^\infty(SM)$ and $Z,W\in\Z$, we have
\begin{equation}
\label{eq:ibp}
\begin{split}
\iip{\vgrad u}{Z} &= -\iip{u}{\vdiv Z},\\
\iip{X u}{w} &= -\iip{u}{X w}+\int_{\partial(SM)}uw\ip{v}{\nu},\text{ and}\\
\iip{X Z}{W} &= -\iip{Z}{X W}+\int_{\partial(SM)}\ip{Z}{W}\ip{v}{\nu}.
\end{split}
\end{equation}
We will not integrate by parts with horizontal derivatives, so these formulas will suffice.
For more details on these operators, we refer to~\cite{Paternain2015}.

\subsection{Simple Riemannian manifolds}\label{sec:simple}

A smooth and compact Riemannian manifold with boundary is called simple if the manifold is simply connected, the boundary is strictly convex (the second fundamental form is positive definite), and there are no conjugate points.

\begin{lemma}
\label{lma:simple-if}
If $M$ is simple, then any vector field $W\in\Z$ satisfies
\begin{equation}
\aabs{XW}^2-\iip{RW}{W}\geq0
\end{equation}
and equality holds if and only if $W=0$.
\end{lemma}

To prove the lemma, convert the integral over $SM$ to integrals over individual geodesics using the Santal\'o formula.
The resulting integral along a geodesic is precisely the index form, which is positive definite due to the lack of conjugate points.

\begin{lemma}
\label{lma:pestov}
If $u\in C^\infty(SM)$ vanishes at the boundary, then
\begin{equation}
\aabs{\vgrad Xu}^2
=
\aabs{X\vgrad u}^2-\iip{R\vgrad u}{\vgrad u}+(n-1)\aabs{Xu}^2.
\end{equation}
This is known as the Pestov identity.
\end{lemma}

To prove the lemma, convert $\aabs{\vgrad Xu}^2-\aabs{X\vgrad u}^2$ into inner products in $L^2(SM)$, integrate by parts, use commutator formulas to simplify the resulting operator, and simplify the result.
The same result for closed manifolds with a full proof can be found in~\cite[Proposition 2.2]{Paternain2015}.

\subsubsection{X-ray tomography of scalars and one-forms}

Lemmas \ref{lma:simple-if}--\ref{lma:pestov} lead to an elegant proof of one of the most basic injectivity results on manifolds:

\begin{theorem}[{\cite{Mukhometov1977}}]
    \label{thm:simple,m=0}
    If $M$ is a compact and simple manifold with boundary, then the geodesic X-ray transform on $M$ is injective on $C^\infty(M)$.
\end{theorem}

We refer to~\cite[Section 4.9]{Sharafutdinov1994} for a discussion of the history of this method. The idea of the proof is to recast the injectivity as unique solvability of the PDE~\eqref{eq:transport} or
\begin{equation}
    \begin{cases}
	\vgrad X u&=0\quad\text{in }SM\\
	u&=0\quad\text{in }\partial(SM)
    \end{cases}
\end{equation}
and using the Pestov identity to show that the only solution is indeed $u=0$.

\begin{proof}[Proof of theorem~\ref{thm:simple,m=0}]
    Let $f\in C^\infty(SM)$ be a function with $If=0$.
    Define a function $u\colon SM\to\Rm$ by
    \begin{equation}
	u(x,v)
	=
	\int_0^{\tau(x,v)}f(\gamma_{x,v}(t))dt,
    \end{equation}
    where $\gamma_{x,v}\colon[0,\tau(x,v)]\to M$ is the maximal unit speed geodesic starting at $x$ in the direction $v$.
    Simplicity ensures that no geodesics are trapped.
    Since the boundary is strictly convex, all geodesics exit transversally.
    It is easy to check that therefore $u$ is smooth in the interior of~$SM$.
    
    Because $If=0$, the function $u$ vanishes at the boundary, since $u|_{\partial_-SM}=0$ and $u|_{\partial_+SM}=If$.
    It is also smooth up to the boundary.
    One way to see this is to use boundary determination:
    studying very short geodesics almost tangent to the boundary shows that $f$ and its normal derivatives of all orders must vanish at $\partial M$.
    (A similar argument works for broken rays as well, provided that certain weighted ray transforms on the boundary are injective~\cite{Ilmavirta2014}. The same method can be applied in the present case without the need for using transforms on the boundary.)
    
    The function $u$ satisfies the transport equation $Xu=-f$, where we have identified $C^\infty(M)\ni f=\pi^*f\in C^\infty(SM)$.
    The function $f$ is independent of direction, so $\vgrad Xu=0$.
    Using lemma~\ref{lma:pestov} with $\vgrad Xu=0$ for $u$ and lemma~\ref{lma:simple-if} for $\vgrad u$ gives $\aabs{Xu}=0$, implying $f=0$.
\end{proof}

The same method can also be applied to one-forms:

\begin{theorem}[{\cite{Anikonov1997}}]\label{thm:simple,m=1}
    If $M$ is a compact and simple manifold with boundary, then the geodesic X-ray transform on $M$ is solenoidally injective on smooth one-forms.
    That is, if $f$ is a smooth one-form that integrates to zero over all maximal geodesics, there is $h\in C^\infty(M)$ which vanishes at the boundary and satisfies $f=dh$.
\end{theorem}

\begin{proof}[Proof sketch]
    The proof is similar to the scalar case presented above, and starts by defining $u\in C^\infty(SM)$ and observing that it vanishes at the boundary.
    The left-hand side of the Pestov identity no longer vanishes, but it cancels one term on the right precisely, because $\aabs{\vgrad f}^2=(n-1)\aabs{f}$, where we have again identified $f$ as a function on $SM$.
    This leads to $\aabs{X\vgrad u}^2-\iip{R\vgrad u}{\vgrad u}=0$, which by lemma~\ref{lma:simple-if} implies that $\vgrad u=0$.
    Therefore there is a function $h\in C^\infty(M)$ so that $u=-\pi^*h$.
    The transport equation $Xu=-f$ is then equivalent with $dh=f$, so $h$ is the desired function.
\end{proof}

\subsubsection{Tensor tomography}

If $f$ is a tensor field of order $0$, the left-hand side of the Pestov identity of lemma~\ref{lma:pestov} vanishes.
If $f$ is of order $1$, the term precisely cancels the $\aabs{Xu}^2$ term. 
If the order is $m\geq2$, the Pestov identity no longer has this convenient positive definiteness.
However, using the Pestov identity not for the whole $u$ but for individual terms $u_k$ in its spherical harmonic decomposition has turned out to be useful.

In two dimensions the X-ray transform is solenoidally injective on simple manifolds for tensor fields of any order:

\begin{theorem}[{\cite{Paternain2011a}}]
    The geodesic X-ray transform is solenoidally injective on the space of smooth tensor fields of any order $m\geq0$ on a simple Riemannian surface.
\end{theorem}

One can also dispense with simplicity, if certain properties are assumed of the X-ray transform at ranks zero and one:

\begin{theorem}[{\cite{Paternain2011a}}]
    Let $M$ be a compact non-trapping surface with a strictly convex and smooth boundary, so that the X-ray transform is solenoidally injective for tensor fields of orders zero and one, and that the adjoint of the X-ray transform on scalars is surjective.
    Then the geodesic X-ray transform is solenoidally injective on the space of smooth tensor fields of any order $m\geq0$.
\end{theorem}

In higher dimensions it is not known whether solenoidal injectivity is always true on simple manifolds. However, with a stronger version of not having conjugate points, we can still formulate the result. Namely, given $\alpha\ge 0$, we say that the manifold $(M,g)$ is {\bf$\alpha$-controlled} if every $W\in\Z$ with zero boundary values satisfies
\begin{equation}
    \aabs{XW}^2-\iip{RW}{W} \geq \alpha\aabs{XW}^2.
\end{equation}
In this context, one may show that a simple manifold with strictly convex boundary is $0$ controlled. Then the theorem below gives a positive answer to TT($m$) under a stronger condition. 
\begin{theorem}[{\cite[Theorem 11.8]{Paternain2015}}]
    The geodesic X-ray transform is solenoidally injective on the space of smooth tensor fields of order $m\geq1$ on a simple $n$-dimensional Riemannian manifold which is $\alpha$-controlled for
    \begin{equation}
	\alpha\geq\frac{(m-1)(m+n-2)}{m(m+n-1)}.
    \end{equation}
\end{theorem}

Earlier results were given by Sharafudtinov in \cite[Theorem 4.3.3]{Sharafutdinov1994}. There, solenoidal injectivity holds over $m$-tensors under the curvature bound condition 
\begin{align*}
    k^+(M,g) := \sup_{(x,v)\in \partial_+ SM} \int_0^{\tau(x,v)} t K^+ (\varphi_t(x,v))\ dt <\frac{1}{m+1}, \qquad K^+ (x,v) := \max (0,K(x,v)),
\end{align*}
where $K(x,v)$ is the supremum of sectional curvatures over all two-planes of $T_x M$ containing~$v$. This bound was further improved by the same author from an $1/(m+1)$ bound to an $(m+2n-1)/m(m+n)$ one in his lecture notes \cite{Sharafutdinov1999a}. Such conditions on $k^+(M,g)$ allow to relate the criterion of absence of $\beta$-conjugate points with the geometry (curvature), via sufficient but not necessary conditions. 

There are a number of different ways to use the Pestov identity to obtain tensor tomography results.
The basic idea is to show that the integral function $u$ (solution to the transport equation $Xu=-f$) is a tensor field of order $m-1$ (has spherical harmonic content only up to degree $m-1$) if $f$ has order $m$ and $I^mf=0$.
In a certain sense, it is trivial that there is a potential, but the non-trivial part is to show that it is a tensor field.

In two dimensions one can conveniently use (anti)holomorphic integrating factors and reduce the problem to showing that certain shifted versions of $u$ are (anti)holomorphic~\cite{Paternain2011a}.
Alternatively, one can use the iterated Beurling transform to bound high order harmonic content of $u$ as outlined in section~\ref{sec:Beurling}.
Careful analysis shows that the products of constants of continuity constants for the Beurling transforms are uniformly bounded~\cite{Paternain2015,Lehtonen2017}.
Simpler estimates can be used to derive weaker bounds; as long as the products grow sufficiently slowly, one can still conclude that $u$ has finite degree~\cite{Ilmavirta2018}.
The Pestov identity may also be combined with other integral identities.

Let us briefly outline the idea with uniformly bounded constants.
Iterating the estimate for Beurling transforms, one can conclude that $\aabs{X_+u_k}\lesssim\aabs{X_+u_{k+2N}}$ when $k\geq m$ and $N\in\Nm$.
By regularity considerations $\aabs{X_+u_{k+2N}}_{L^2}\to0$ as $N\to\infty$, so in fact $X_+u_k=0$.
If $u_k$ satisfies $X_+u_k=0$, then it corresponds to a trace-free conformal Killing tensor field of order $k$ vanishing at the boundary.
Such tensor fields do not exist in simple geometry~\cite{Dairbekov2011}.
For more details on tensor tomography, see~\cite{Paternain2013,Paternain2015}.

\subsection{$\beta$-conjugate points, the terminator value $\Ter$ and $\alpha$-control}\label{sec:terminator}

The notion of $\alpha$-control just introduced above provides a continuous parameter which allows to encode previous geometric criteria (e.g. simplicity, conditions on curvature, \ldots) as threshold conditions on $\alpha$ (or its related so-called terminator value $\Ter$, as explained below), thereby allowing to refine previous statements. Before mentioning more results, we briefly visit the concepts of $\alpha$-control and terminator value now. 

Let $J$ be a vector field along a geodesic $\gamma$. We say that $J$ is a {\bf$\beta$-Jacobi field} if it satisfies the $\beta$-Jacobi equation
\begin{equation}
    D_t^2J(t) + \beta R(J(t),\dot\gamma(t))\dot\gamma = 0,
\end{equation}
where $D_t$ is the covariant derivative and $R$ the Riemann curvature tensor. 
The constant $\beta$ describes how sensitive these generalized Jacobi fields are to curvature.
These generalizations were introduced to X-ray tomography in~\cite{Pestov2003,Dairbekov2006} and are also extensively used on closed manifolds with Anosov geodesic flow \cite{Paternain2014,Paternain2015}. For $\beta=1$ we obtain the usual Jacobi fields.

We say that two points on $\gamma$ are {\bf$\beta$-conjugate} if there is a non-trivial $\beta$-Jacobi field vanishing at these two points. We then say that  $M$ is free of $\beta$-conjugate points if no two points are $\beta$-conjugate along any geodesic. As one may show that if $M$ is free of $\beta$-conjugate points, it is also free of $\beta'$-conjugate points for any $\beta' \in [0,\beta)$, this justifies the definition of {\bf terminator value} for the manifold $M$, given by
\begin{equation} 
    \Ter := \sup\{\beta\geq0;\ M \text{ is free of $\beta$-conjugate points}\} \in [0,\infty].
\end{equation}

As seen below, some classical geometric conditions can be reformulated as threshold conditions on $\Ter$. Namely: 
\begin{itemize}
\item If the manifold is compact and non-trapping, then $\Ter>0$.
\item There are no conjugate points if and only if $\Ter\geq1$. 
\item The manifold has non-positive curvature if and only if $\Ter=\infty$.
\end{itemize}

The main use of the terminator value is when relating it to $\alpha$-control on manifolds, as stated at the end of the previous section. Recall that the manifold $(M,g)$ is {\bf$\alpha$-controlled} if 
\begin{equation}
    \aabs{XW}^2-\iip{RW}{W} \geq \alpha\aabs{XW}^2, \qquad W\in \Z, \qquad W|_{\partial M} = 0.
\end{equation}
Then the terminator value is related to controllability as follows:

\begin{lemma}[{\cite[Proposition 7.1 and Remark 11.3]{Paternain2015}}]\label{lma:terminator-control}
    If a compact manifold (closed or with boundary) satisfies $\Ter\geq\beta$, then the manifold is $(\beta-1)/\beta$-controlled.
\end{lemma}
Tying this lemma with the comments above on $\Ter$, one may draw the following conclusions (see \cite[Lemma 11.2]{Paternain2015}): a non-trapping manifold with strictly convex boundary is $0$-controlled if it has no conjugate points, $\alpha$-controlled for $\alpha>0$ if it is simple, and $1$-controlled if and only if it has non-positive sectional curvature.

\subsection{Other spaces}\label{sec:PestovOther}

Pestov identities can also be used on other types of manifolds. We mention some examples.

\subsubsection{Closed manifolds}

Tensor tomography on closed manifolds (compact, without boundary) can be studied in a similar way with Pestov identities.
Simple manifolds have boundary, and the corresponding closed manifolds are Anosov manifolds.
On an Anosov surface, we have solenoidal injectivity for tensor fields of order zero~\cite{Dairbekov2003}, one~\cite{Dairbekov2003}, two~\cite[Theorem 1.1]{Paternain2014}, and order $m\geq3$ in \cite[Theorem 1.4]{Guillarmou2014}. For the case $m\ge 3$, earlier results were written if the terminator value is at least $\frac{m+1}2$~\cite[Theorem 1.3]{Paternain2014}, or if the manifold is negatively curved~\cite{Croke1998}.
More details are covered in the two-dimensional survey~\cite{Paternain2013}.

On Anosov manifolds of dimension $n\geq3$ the X-ray transform is solenoidally injective for tensor fields of order $m\geq2$ if the terminator value satisfies $\Ter >\frac{m(m+n-1)}{2m+n-2}$.
For tensor fields of order zero or one no such condition on the terminator value is needed~\cite{Dairbekov2003}.

\subsubsection{Pseudo-Riemannian manifolds}

On pseudo-Riemannian manifolds one does not usually study all geodesics, but only the light-like ones.
The X-ray transform restricted to this set of null geodesics is also known as the light ray transform.
Light rays as sets are conformally invariant, unlike Riemannian geodesics.

No Pestov identity is known on general pseudo-Riemannian or Lorentzian manifolds for X-ray or light ray transforms.
However, when the spacetime is a product of space and time, a Pestov identity can be used to prove injectivity of the light ray transform for scalars and one-forms~\cite{Ilmavirta2016}, under the assumption that both space and time are non-positvely curved Riemannian manifolds with strictly convex boundary and dimension at least two.
The dimension assumption rules out Lorentzian manifolds.
The Pestov identity is obtained by finding a Pestov-like identity on both space and time manifolds, and then combining them with suitable weights.

If the product $M=M_1\times M_2$ of two Riemannian manifolds $(M_i,g_i)$ is equipped with the pseudo-Riemannian product metric $g_1\oplus -g_2$, then the Pestov identity reads~\cite[Lemma 3]{Ilmavirta2016}
\begin{equation}
\begin{split}
&
(n_2-1)\aabs{\vgrad_1 Xu}^2
+
(n_1-1)\aabs{\vgrad_2 Xu}^2
\\&
=
(n_2-1)\aabs{X\vgrad_1u}^2
+
(n_1-1)\aabs{X\vgrad_2u}^2
\\&\quad
-(n_2-1)\iip{R_1\vgrad_1 u}{\vgrad_1u}
-(n_1-1)\iip{R_1\vgrad_2 u}{\vgrad_2u}
\\&\quad
+(n_1-1)(n_2-1)\aabs{Xu}^2.
\end{split}
\end{equation}
Here $n_i=\dim(M_i)$ and the various operators on the two underlying Riemannian manifolds are indicated by a subscript.
The identity is for smooth functions on the compact light cone bundle $LM\coloneqq SM_1\times SM_2$ on which the null geodesic flow lives.

\begin{theorem}[{\cite[Theorem 1]{Ilmavirta2016}}]
Suppose $(M_i,g_i)$, $i=1,2$, are two simple Riemannian manifolds with non-positive sectional curvature.
Equip the product $M=M_1\times M_2$ with the pseudo-Riemannian metric $c(x)(g_1\oplus-g_2)$, where $c\in C^\infty(M)$ is a smooth conformal factor.
Then a smooth function $f$ supported outside the edges ($\partial M_1\times\partial M_2$) integrates to zero over all null geodesics if and only if $f=0$, and a smooth one-form $\alpha$ supported outside te edges integrates to zero over all null geodesics if and only if $\alpha=dh$, where $h\in C^\infty(M)$ vanishes at the boundary.
\end{theorem}

\subsubsection{Convex obstacles}

Consider a compact manifold $M$ with strictly convex boundary with a strictly convex obstacle $O\subset M$.
Instead of integrating a function $f\colon M\to\Rm$ over all geodesics through $M$, we integrate a function $f\colon M\setminus O\to\Rm$ over all geodesics on $M\setminus O$ which have endpoints on $\partial M$ and reflect specularly on $\partial O$.
In two dimensions specular reflections can be characterized by saying that the angle of incidence equals the angle of reflection.

One can employ a similar approach to that of theorem~\ref{thm:simple,m=0}.
One defines a function $u$ by integrating the unknown function $f\in\ker I$ along the broken ray until $\partial M$ is met.
This function $u(x,v)$ vanishes when $x\in\partial M$ but not a priori when $x\in\partial O$.
This produces a boundary term in the Pestov identity, which can be simplified using the symmetry of $u$ under specular reflection.

To be precise, in two dimensions one obtains the Pestov identity~\cite[Lemma 6]{Ilmavirta2016b}
\begin{equation}
\begin{split}
\aabs{VXu}^2
&=
\aabs{XVu}^2
+\aabs{Xu}^2
\\&\quad
-\iip{KVu}{Vu}
-\iip{\kappa Vu}{Vu}_{L^2(\partial(SN))},
\end{split}
\end{equation}
valid for sufficiently regular functions on the manifold $N=M\setminus O$ which vanish at $\partial M$ and satisfy a reflection condition at $\partial O$.
Here the other terms are in the space $L^2(SN)$ as usual, and $\kappa$ is the curvature of $\partial O$.
If obstacle is strictly convex, then $\kappa<0$ and the boundary term has the correct sign.

Regularity is tricky in the presence of reflections; the function $u$ is not a priori smooth even for smooth $f$ and smooth geometry.
Singularities occur at tangential reflections at~$\partial O$.

This method was used to prove injectivity of the broken ray transform by Eskin~\cite{Eskin2004b} in the Euclidean plane with several reflecting obstacles and by the first author and Salo~\cite{Ilmavirta2016b} on non-positively curved Riemannian surfaces with a single reflecting obstacle.

The method was extended to any dimensions and tensors of any rank in~\cite{Ilmavirta2018}, still assuming a single reflecting obstacle.
The scalar-valued curvature~$K$ is replaced with the curvature operator~$R$ as described above.
In the boundary term the scalar curvature~$\kappa$ is replaced by the second fundamental form.
Assuming non-positive sectional curvature of the manifold and strict concavity of the reflector (as seen from the interior of~$N$, equivalent with the strict convexity of the obstacle~$O$) gives positivity to all terms.

\subsubsection{Non-compact manifolds}

Pestov identities have recently been successfully used on some non-compact context to prove tensor tomography and boundary rigidity. The identity looks exactly the same, but one needs to be far more careful with integrability and regularity.  

A positive answer to tensor tomography (Problem \ref{pb2}) was given in \cite{Lehtonen2016,Lehtonen2017} for some cases of {\bf Cartan--Hadamard}\footnote{A Cartan--Hadamard manifolds is a complete, simply connected Riemannian manifold with non-positive sectional curvature.} manifolds. For such manifolds, $\Ter = \infty$ so we may expect good $\alpha$-control, which helps to control terms of index form type appearing in Lemma \ref{lma:pestov}. The results extend to tensor fields with non-compact support, however suitable decay at infinity is needed, a non-artificial requirement since otherwise counterexamples exist even in the Euclidean case. Namely, the following two results are proved in \cite{Lehtonen2017}, generalizing earlier results in \cite{Lehtonen2016} to dimensions greater than $3$ and tensor fields of arbitrary order. Below, for $x\in M$ and $\Pi\subset T_x M$ a two-plane, we denote $K_x(\Pi)$ the sectional curvature of the two-plane $\Pi$. Using any distinguished point $o\in M$ as the ``origin'', we say that a function $f$ (or a tensor field) 
\begin{itemize}
    \item has polynomial decay at infinity of order $\eta$ if $x\mapsto |f(x)|(1+d(x,o))^\eta$ is bounded.
    \item has exponential decay at infinity of order $\eta$ if $x\mapsto |f(x)| e^{\eta d(x,o)}$ is bounded.
\end{itemize}
The first theorem is of an asymptotically hyperbolic flavor, while the second is asymptotically Euclidean. 
\begin{theorem}[{\cite[Theorem 1.1]{Lehtonen2017}}]\label{thm:CH1} Let $(M,g)$ be a Cartan--Hadamard manifold of dimension $n\ge 2$, and assume that there exists $K_0>0$ such that 
    \begin{align*}
	-K_0 \le K_x(\Pi) \le 0, \qquad x\in M, \quad \Pi\subset T_xM.
    \end{align*}
    If $f\in C^1(S^m(T^* M))$ and $\nabla f$ have exponential decay at infinity of order $\eta>\frac{n+1}{2}\sqrt{K_0}$, and if $I_m f= 0$, then $f = \sigma\nabla h$ for some $h\in C^1(S^m (T^* M))$ with exponential decay at infinity at rate $\eta-\varepsilon$ for any $\varepsilon>0$ (if $m=0$ then $f\equiv 0$).    
\end{theorem}

\begin{theorem}[{\cite[Theorem 1.2]{Lehtonen2017}}]\label{thm:CH2} Let $(M,g)$ be a Cartan--Hadamard manifold of dimension $n\ge 2$, and assume that the function 
    \[ {\cal K}(x) := \sup_{\Pi\subset T_x M} |K_x(\Pi)|, \qquad x\in M \] 
    decays strictly faster than quadratically at infinity. If $f\in C^1(S^m(T^* M))$ has polynomial decay at infinity of order $\eta>\frac{n+2}{2}$ and $\nabla f$ has decay of order $\eta+1$ and if $I_m f = 0$, then $f = \sigma\nabla h$ for some $h\in C^1(S^m (T^* M))$ with polynomial decay at infinity of order $\eta-1$.     
\end{theorem}

\takeout{
    \F{check theorems above}\J{I checked them and made small edits.} \F{thanks} \J{Now I'm worried that there might be an error in the paper. The statements above assume the paper is correct (which is fair), but I sent email to Jere to ask for clarification. It seems that they should argue that there are no conrofmal Killing tensors to show $X_+u_{m+1}=0\implies u_{m+1}=0$ but they don't.} \J{Mikko gave me an argument that cleared most of my confusion. It's fine with me as it stands.}
}

While the results on Cartan--Hadamard manifolds allow to treat more than one type of geometry at infinity, the recent work \cite{Graham2017} focuses on the {\bf asymptotically hyperbolic} context, however covers the case where a hyperbolic trapped set is present using the tools of Section~\ref{sec:trapping}, and also treats some non-linear results. It is convenient here to picture $(M,g)$ as a manifold-with-boundary where the metric $g$ has a specific singular behavior at the boundary such that geodesics never reach it in finite time. In addition, geodesics making it to the boundary as $t\to \pm \infty$ always hit the boundary normally, so one must look at second-order information to define the space of geodesics and other related objects like the scattering relation.

Under a no-conjugate points assumption, the authors prove a positive answer to Problem \ref{pb2} in \cite[Theorem 1]{Graham2017} for classes of tensors with suitable decay conditions at infinity which agree with those of Theorem \ref{thm:CH1}. The proof uses Pestov identities in the interior, after using the specific structure of the geodesics in a neighborhood of $\partial M$ to prove that a tensor field whose ray transform vanishes agrees to infinite order with a potential tensor near~$\partial M$.

The authors go on to studying non-linear inverse problems, determining features of the metric up to gauge from renormalized geodesic lengths, first at the boundary in \cite[Theorem 2]{Graham2017}, then globally if the manifold is real-analytic and such that $\pi_1 (M,\partial M) = 0$ in \cite[Theorem 3]{Graham2017}. Finally, \cite[Theorem 4]{Graham2017} establishes a deformation rigidity result (a variant of the Lens Rigidity Problem in Problem \ref{pb1}) in the case of non-trapping asymptotically hyperbolic metrics with non-positive sectional curvature.

\subsection{Unitary connections and skew-hermitian Higgs fields}\label{sec:PestovConnection}

The method of Pestov identities generalizes to the case of transport with hermitian connection and skew-hermitian Higgs field, as treated in the recent works \cite{Paternain2012,Guillarmou2015}. Let $E\to M$ a bundle as in Section \ref{sec:ConnectionsHiggs}, equipped with a Hermitian connection $\nabla$ and a skew-Hermitian Higgs field $\Phi$. 

To write Pestov identities, one must then generalize the Sasaki-related objects (e.g. horizontal/vertical gradients) to sections of the (pullback) bundle $E\to SM$. For $u\in C^\infty(SM; E)$, $\nabla^E u\in C^\infty(SM; T^*(SM)\otimes E)$ and using the Sasaki metric on $T(SM)$ we can identify this with an element of $C^\infty(SM; T(SM)\otimes E)$, splitting according to \eqref{eq:SMsplitting} into
\begin{align*}
    \nabla^E u = (\Xm u, \hgradE u, \vgradE u), \qquad \Xm u := \nabla_X^E u,
\end{align*}
where $\hgradE u$ and $\vgradE u$ can be viewed as smooth sections of $N\otimes E$.

Then the following Pestov identity can be derived for any $u\in C^\infty(SM; E)$ vanishing at $\partial SM$, see \cite[Proposition 3.3]{Guillarmou2015}: 
\begin{align}
    \|\vgradE \Xm u\|^2 = \|\Xm \vgradE u\|^2 - \iip{R\vgradE u}{\vgradE u} - \iip{F^E u}{\vgradE u} + (n-1) \|\Xm u\|^2,
    \label{eq:PestovConnection}
\end{align}
where $R$ is the Riemann curvature tensor of $(M,g)$, viewed as an operator on the bundles $N$ and $N\otimes E$ over $SM$ by the actions
\begin{align*}
    R(x,v) w := R_x(w,v)v, \qquad R(x,v) (w\otimes e) := (R_x (w,v)v) \otimes e, \quad e\in E_{x},
\end{align*} 
and where $F^E\in C^\infty (SM; N\otimes \text{End}_{\text{sk}} (E))$ is the curvature of the connection $\nabla^E$, see \cite[Eq. (3.5)]{Guillarmou2015}.

On simple surfaces, where $E$ is trivial and we globally represent $\Xm = X+ A$ for some skew-hermitian matrix of one-forms $A$, Equation \eqref{eq:PestovConnection} takes the form (see  \cite[Lemma 6.1]{Paternain2012})
\begin{align}
    \|V (X+A)u\|^2 = \|(X+A) Vu\|^2 - \iip{\kappa Vu}{Vu} - \iip{\star F_A u}{Vu} + \|(X+A)u\|^2,
    \label{eq:PestovConnection2}
\end{align}
where $\star F_A\in C^\infty(SM)$ is again related to the curvature of the connection. We now explain two possible ways to exploit the identities \eqref{eq:PestovConnection}-\eqref{eq:PestovConnection2} to produce positive answer to tensor tomography questions. 

\begin{theorem}[{\cite[Theorem 1.3]{Paternain2012}}]\label{thm:Conn1} Let $(M,g)$ a simple surface and $E$ a bundle over $M$ with hermitian connection $A$ and skew-Hermitian Higgs field $\Phi$. If $f\in C^\infty(SM,\Cm^n)$ is of the form $f = f_0 + f_1$, and if $I_{A,\Phi} f = 0$, then $f_1 = (d+ A)p$ and $f_0 = \Phi p$ for some smooth function $p:M\to \Cm^n$ vanishing at~$\partial M$.
\end{theorem}

\begin{proof} (sketch) {\bf Case $\Phi=0$.} To prove this, we assume that $u$ and $f$ are as in \eqref{eq:trans2} with $u|_{\partial SM} = 0$, and the result is proved when we show that $Vu = 0$ (so that $p:= u_0$ does the trick). We explain how $v := \sum_{k<0}u_k$ is zero, as the proof for $\sum_{k>0} u_k$ is similar. Equation \eqref{eq:PestovConnection2} applied to $v$ becomes
    \begin{align*}
	\|(X+A) Vv\|^2 - \iip{\kappa Vv}{Vv} - \iip{\star F_A v}{Vv} + \|((X+A) v)_0\|^2 = 0.
    \end{align*}
    The sum of the first two terms is non-negative due to the simplicity of the metric. Moreover, one can establish, using holomorphic integrating factors for scalar connections, that the injectivity of $I_A$ does not depend on perturbing the connection by a scalar one. In particular, $A$ can be perturbed into $A'$ by a scalar term for which $i \star F_{A'}\le 0$ in the sense of Hermitian operators. Assuming this is the case, we obtain that $- \iip{\star F_A v}{Vv} \ge 0$. This forces all terms to be zero and thus $v = 0$. \\

    {\bf Case $\Phi\ne 0$.} In the setting of the proof above, if a Higgs field is present, one must write a Pestov identity in the form of \eqref{eq:PestovConnection2} for the operator $V(X+A+\Phi)$ instead of $V(X+A)$. Some additional terms appear in the identity, which can be controlled by $C_{A,\Phi} \sum_{k=-\infty}^{-1} |k||v_k|^2$, and perturbing the connection with a scalar one so that $- \iip{\star F_A v}{Vv} >> 0$, one can again control these terms in a coercive fashion and enforce $v=0$.
\end{proof}

The second setting is that of manifolds with negative sectional curvature, where the answer to the tensor tomography problem can be made positive for tensors of arbitrary order. 

\begin{theorem}[{\cite[Theorems 4.1, 4.6]{Guillarmou2015}}]\label{thm:Conn2}
    Let $(M,g)$ a compact manifold with negative sectional curvature and $E$ a hermitian bundle with hermitian connection $\nabla^E$ and $\Phi$ a skew-Hermitian Higgs field. If $u\in C^\infty(SM;E)$ satisfies $(\Xm+\Phi)u = -f$ where $f$ has finite degree and if $u|_{\partial SM} = 0$, then $u$ has finite degree.  
\end{theorem}

\begin{proof} (sketch) In the case $\Phi=0$, the proof consists in applying \eqref{eq:PestovConnection} to the high-frequency content of $u$ (say, $T_{\ge m} u := \sum_{k\ge m} u_k$ for $m$ large enough), and show that for $m$ large enough, the curvature term $\iip{-R\vgradE u}{\vgradE u} $ overtakes the contribution of the connection term $\iip{-F^E u}{\vgradE u}$. In particular, \cite[Lemma 4.2]{Guillarmou2015} shows that for $m$ large enough, $\vgradE (T_{\ge m} u)$ is controlled by $\vgradE (T_{\ge m+1} \Xm u) = -\vgradE (T_{\ge m+1} f)$, but since the latter vanishes identically for $m$ large enough, so does the former. If a skew-Hermitian Higgs field $\Phi$ is added, when controlling the high frequencies of $u$, the terms involving $\Phi$ can still be overtaken thanks to the negative curvature. 
\end{proof}

\subsection{Magnetic and thermostat flows}\label{sec:PestovMagnetic}

The method by energy identities has been generalized to other types of non-geodesic flows on manifolds, which we will discuss here. 

Fixing a Riemannian manifold $(M,g)$, geodesic trajectories can be viewed as zero-acceleration curves in the Levi-Civita connection $\nabla$, governed by the equation $\nabla_{\dot\gamma}\dot \gamma = 0$. A way to consider other flows can be done by adding a mechanically motivated force field, characterized by a bundle map $F\colon TM\to TM$ for which the trajectories evolve under Newton's second law
\begin{align}
    \nabla_{\dot \gamma} \dot \gamma = F(\gamma,\dot\gamma).
    \label{eq:Newton}
\end{align}
If, in addition, this force field is skew-hermitian, then the quantity $|\dot\gamma|$ is preserved along trajectories, and we obtain a flow $\psi_t$ on $SM$ again, whose generator $G:= \frac{d}{dt} \psi_t|_{t=0}$ can be shown to take the form $G = X + \lambda V$ (e.g., in two dimensions), where $\lambda:SM \to \Rm$ incorporates information about the force field. One may then define associated ray transforms of functions and tensor fields via solving transport equations of the form $Gu = -f$. Many objects then depend on the flow under consideration, namely: the kernel of the ray transform over tensor fields, the notion of convexity at the boundary, the notion of conjugate points, etc\dots One may also lift this flow to a bundle $E\to M$ with a connection $A$ and Higgs field $\Phi$ and consider the associated notions of scattering data $C_{A,\Phi}$ and X-ray transforms $\I_{A,\Phi}$. In this context, the following results have been derived. 

\subsubsection{Magnetic ray transforms} 

A magnetic field on $M$ is a closed two-form $\Omega$, and this gives rise to the magnetic force $F:TM\to TM$ uniquely defined by 
\[ \Omega_x (\xi,\eta) = g(F_x(\xi), \eta), \qquad x\in M, \qquad \xi,\eta\in T_xM, \]
see \cite{Dairbekov2007,Ainsworth2013,Merry2011}. Here the function $\lambda(x,v) = -g(F_x(v), v_\perp)$ in fact does not depend on $v$. One can then define the concept of a magnetically convex boundary, and being ``simple'' with respect to the magnetic flow, and consider Problems \ref{pb1}--\ref{pb4}. 

The first results for such transforms were given in \cite{Dairbekov2007}, where energy methods (formulated there in the language of semi-basic tensor fields) are used in \cite[Section 5]{Dairbekov2007} to prove injectivity of the magnetic ray transform over functions, one-forms \cite[Theorem 5.3]{Dairbekov2007} and two-tensors \cite[Theorem 5.4]{Dairbekov2007} under certain curvature conditions. Further results are established (generic injectivity, magnetic boundary rigidity) using methods discussed later in this article.

On simple magnetic surfaces, positive answers to Problems \ref{pb3} (tensor tomography problem) and \ref{pb4} (determination up to gauge of $(A,\Phi)$ from their scattering data) are obtained in \cite[Theorems 1.2, 1.4, 1.5]{Ainsworth2013} in the presence of a unitary connection and a skew-Hermitian Higgs field, as in Section \ref{sec:PestovConnection}. The schemes of proof of \cite[Theorems 1.2, 1.4]{Ainsworth2013} follow \cite{Paternain2012}, where the new key step is to derive Pestov identities for the operator $V(X+\lambda V+A+\Phi)$.

\subsubsection{Thermostat ray transforms on surfaces} 

Another example of external field is given by a Gaussian thermostat, characterized by a smooth vector field $E$ on $M$, see \cite{Assylbekov2014,Merry2011}. The force field in this case is given by 
\begin{align*}
F(\gamma,\dot\gamma) = E(\gamma) - \frac{g_\gamma(E(\gamma), \dot \gamma)}{|\dot\gamma|^2} \dot \gamma = \frac{g_\gamma(E(\gamma), \dot \gamma_\perp)}{|\dot\gamma|^2} \dot \gamma_\perp
\end{align*}
and the function $\lambda(x,v) = -g(F_x(v), v_\perp)$ is now linear in $v$. Upon defining an associated thermostat ray transform over tensor fields, and defining a notion of terminator value $\Ter$ with respect to a thermostat-Jacobi equation, it is proved in \cite[Theorem 1.5]{Assylbekov2014} that the thermostat ray transform is injective (up to natural obstruction) over $m$-tensors if $\Ter \ge \frac{m+1}{2}$. In particular, for this notion of terminator value, we again have $\Ter = \infty$ if the thermostat curvature is non-positive, in which case the previous result holds for any tensor order $m$, see \cite[Corollary 1.6]{Assylbekov2014}. The proofs are based on deriving Pestov identities for the operator $V(X+\lambda V)$. Associated results in the case of closed surfaces without boundary are given there as well, see \cite[Theorem 1.2, Corollaries 1.3, 1.4]{Assylbekov2014}.

It is worth pointing out that in the geodesic case, the condition on $\Ter$ is not necessary (namely $\Ter>1$ implies the result for any $m$) because one can use holomorphic integrating factors for scalar connections to move from any harmonic level to any other. It may be of interest to seek a similar construction here.

\np
\section{Inversion formulas and another route to injectivity in two dimensions}\label{sec:reconstruction}

In this section, we present constructive inversion approaches for various integrands and geometric contexts. 

\subsection{Pestov--Uhlmann inversion formulas on simple surfaces}

Recall the scattering relation $\SS\colon\partial SM\to \partial SM$ as follows: if $(x,v)\in \partial_+ SM$, $\SS(x,v) := \varphi_{\tau(x,v)}(x,v)\in \partial_- SM$; if $(x,v)\in \partial_- SM$, $\SS(x,v) := \varphi_{-\tau(x,-v)}(x,v)\in \partial_+ SM$. Recall the following definitions of $A_\pm\colon L^2_\mu (\partial_+ SM) \to L^2_{|\mu|}(\partial SM)$  and their adjoints: 
\begin{align*}
    A_\pm w (x,v) = \left\{
    \begin{array}{cc}
	w(x,v) & (x,v)\in \partial_+ SM, \\
	\pm w\circ\SS (x,v) & (x,v)\in \partial_- SM,
    \end{array}
    \right. \qquad A_\pm^* u := (u \pm u\circ\SS)|_{\partial_+ SM}.
\end{align*}
Recall also the definition of the {\em fiberwise Hilbert transform} $H\colon L^2(SM) \to L^2(SM)$, defined on the fiberwise harmonic decomposition by 
\begin{align*}
    H u_k = -i\sgn{k} u_k, \quad u_k\in \Omega_k \qquad \text{with the convention } \sgn{0} = 0.
\end{align*}
Introducing this transform, Pestov and Uhlmann obtained the following formulas in \cite{Pestov2004} (written with slight updates as in \cite[Proposition 2.2]{Monard2015}), inverting the ray transform over functions and solenoidal vector fields:
\begin{align}
    \begin{split}
	f + W^2 f &= \frac{1}{8\pi} I_\perp^* (A_+^* H A_-) I_0 f, \qquad f\in L^2(M), \\
	h + W_\perp^2 h &= \frac{-1}{8\pi} I_0^* (A_+^* H A_-) I_\perp h, \qquad h\in H^1_0(M),
    \end{split}    
    \label{eq:reconsfh}
\end{align}
where the operators $W,W_\perp$ are $L^2(M)\to L^2(M)$-adjoints, and compact smoothing. Such equations take the form of {\em filtered-backprojection} formulas, where $A_+^* H A_-\colon L^2(\partial_+ SM)\to L^2(\partial_+ SM)$ is a continuous `filter' and the adjoints $I_0^*$ or $I_\perp^*$ are viewed as 'backprojection' operators, see Figs. \ref{fig:f_geodesics}--\ref{fig:FBP} for an example.

\begin{figure}[htpb]
    \centering
    \includegraphics[trim=40 50 40 10, clip, width = 0.24\textwidth]{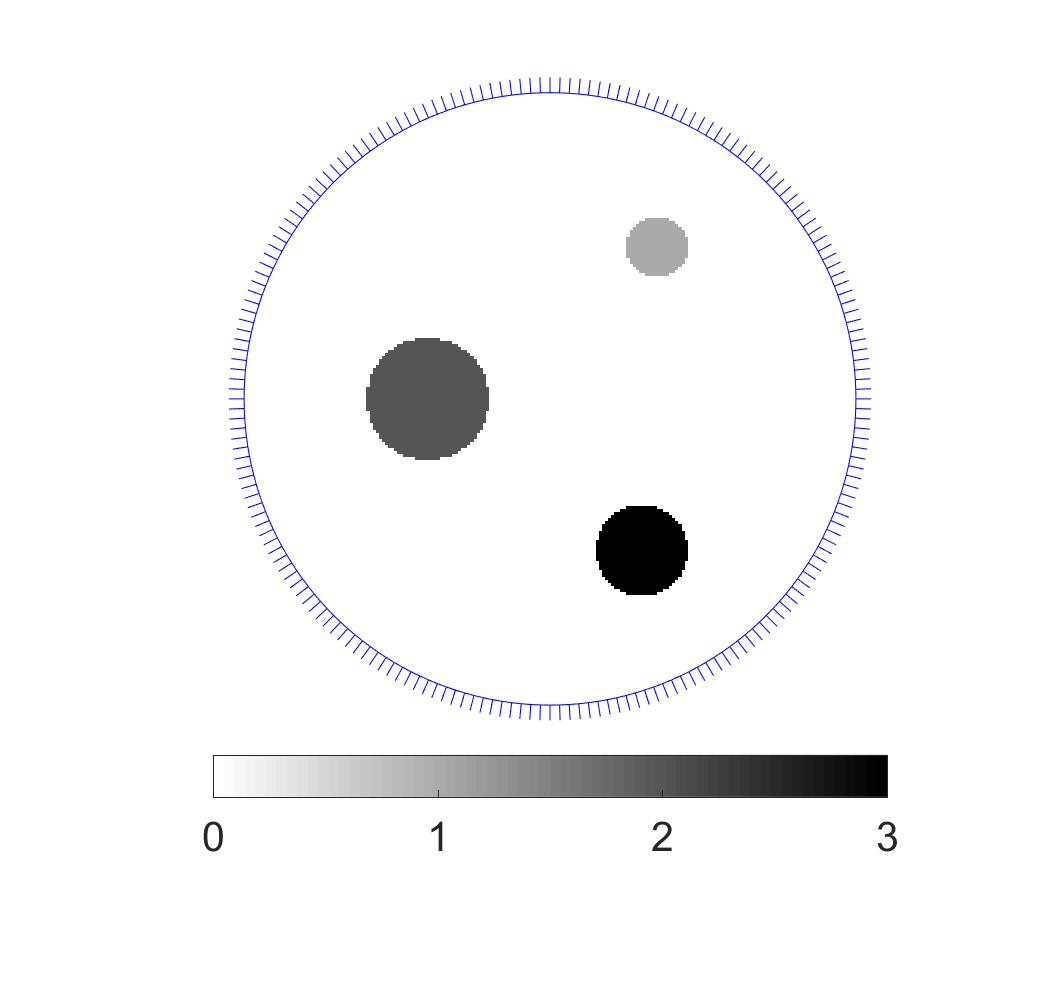}
    \includegraphics[trim=30 10 120 10, clip, width = 0.22\textwidth]{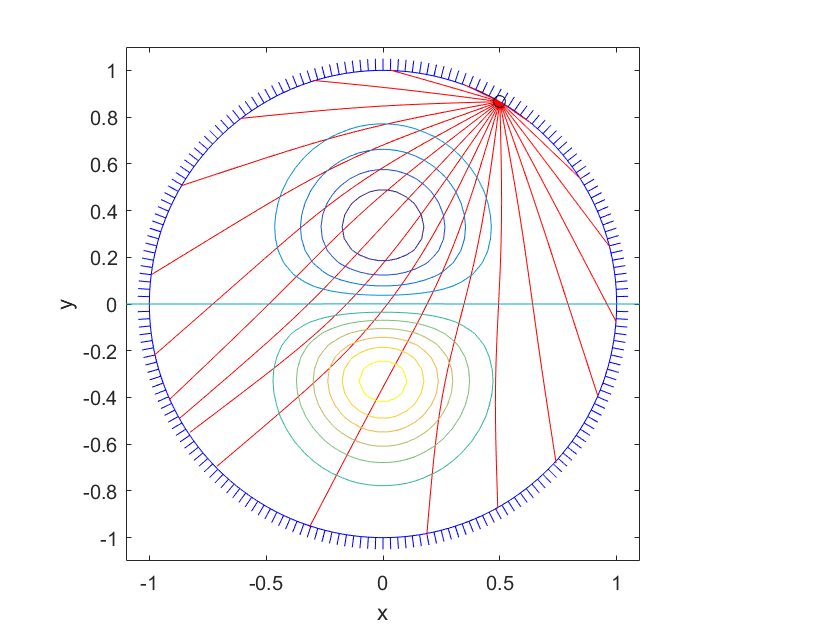}
    \includegraphics[trim=30 10 120 10, clip, width = 0.22\textwidth]{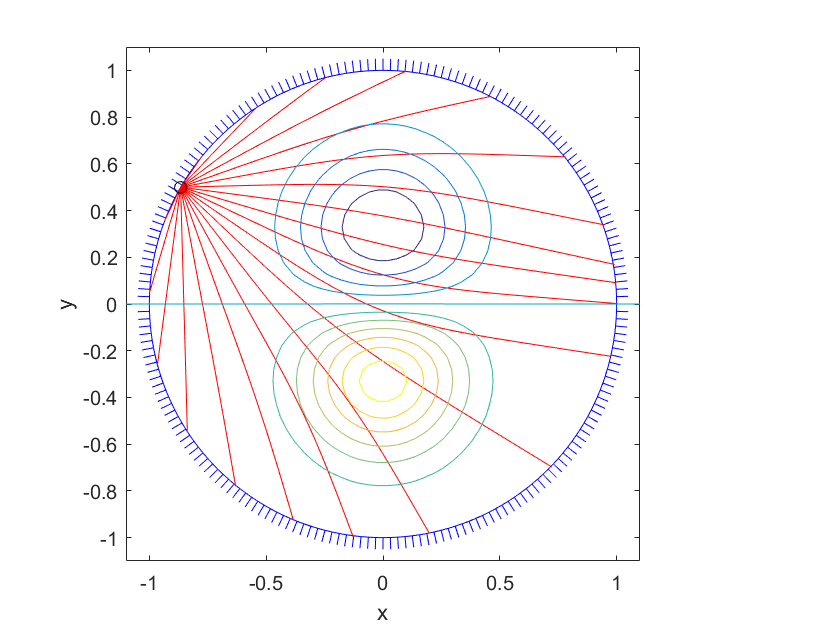}
    \includegraphics[trim=30 10 10 10, clip, width = 0.28\textwidth]{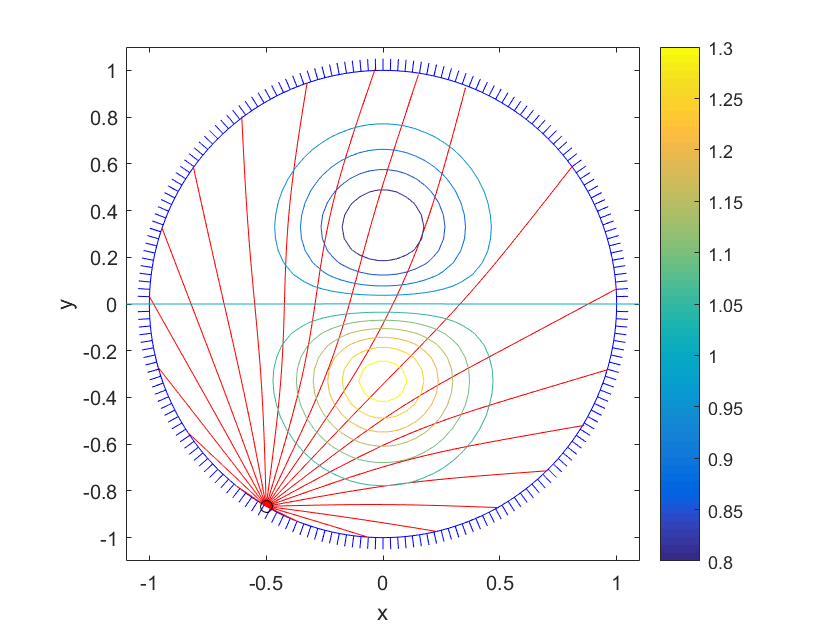}
    \caption{Example of a function $f$ defined on the unit disk (left). The domain is endowed with a scalar metric $g = c^{-2} Id$ where the ``sound speed'' $c$ is contour-plotted, and some geodesics are superimposed on three pictures on the right.}
    \label{fig:f_geodesics}
\end{figure}

\begin{figure}[htpb]
    \centering
    \includegraphics[trim=60 30 60 0, clip, width = 0.35\textwidth]{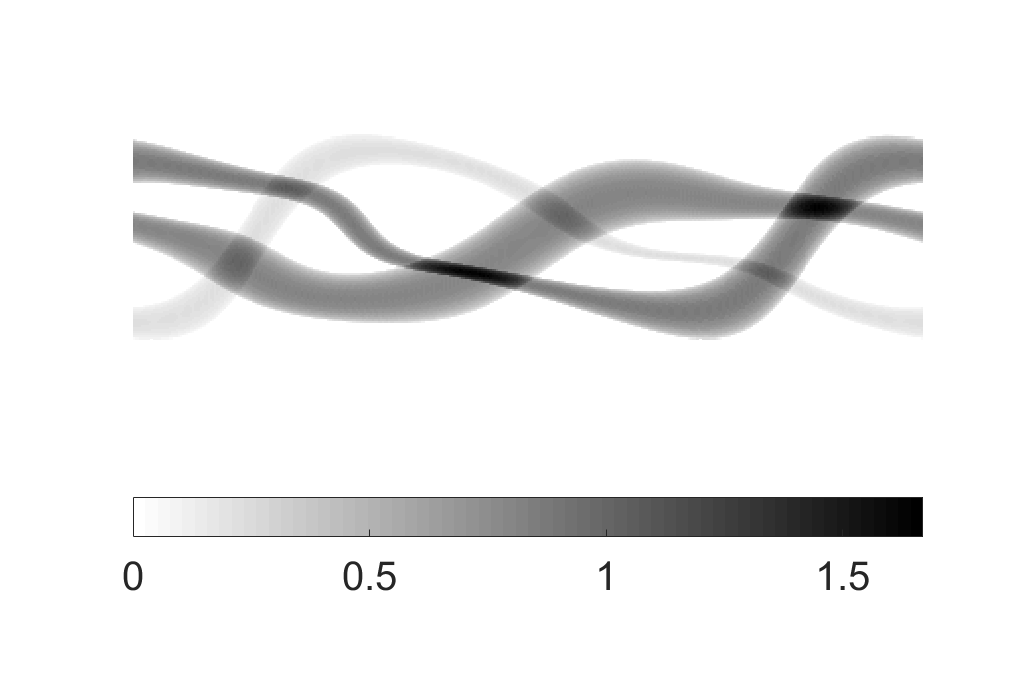}
    \includegraphics[trim=60 30 60 0, clip, width = 0.35\textwidth]{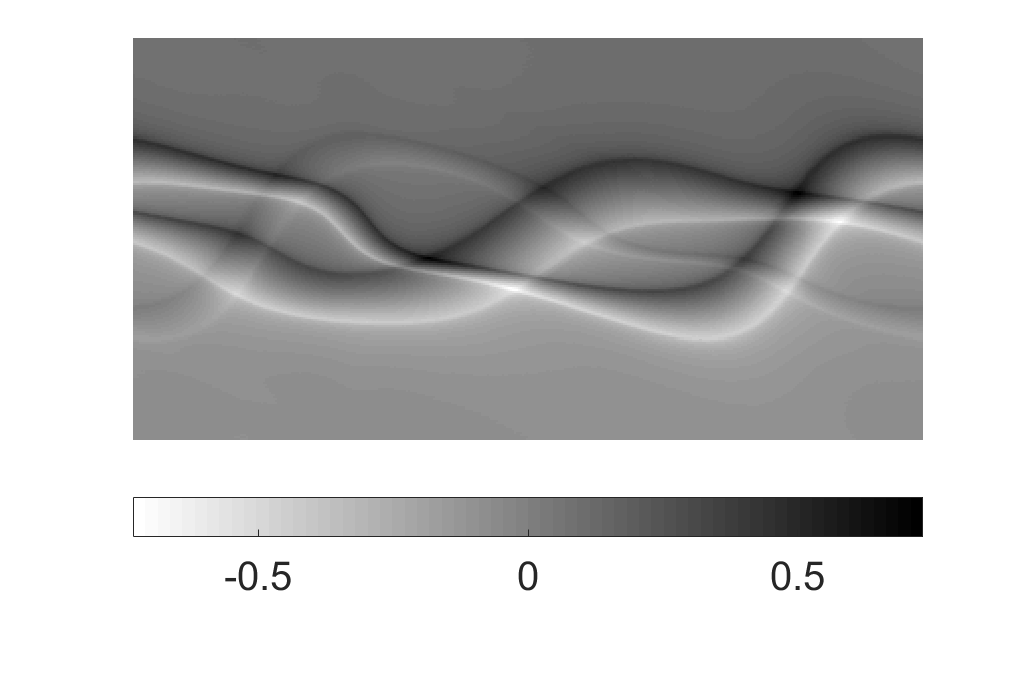}    
    \includegraphics[trim=30 50 30 0, clip, width = 0.28\textwidth]{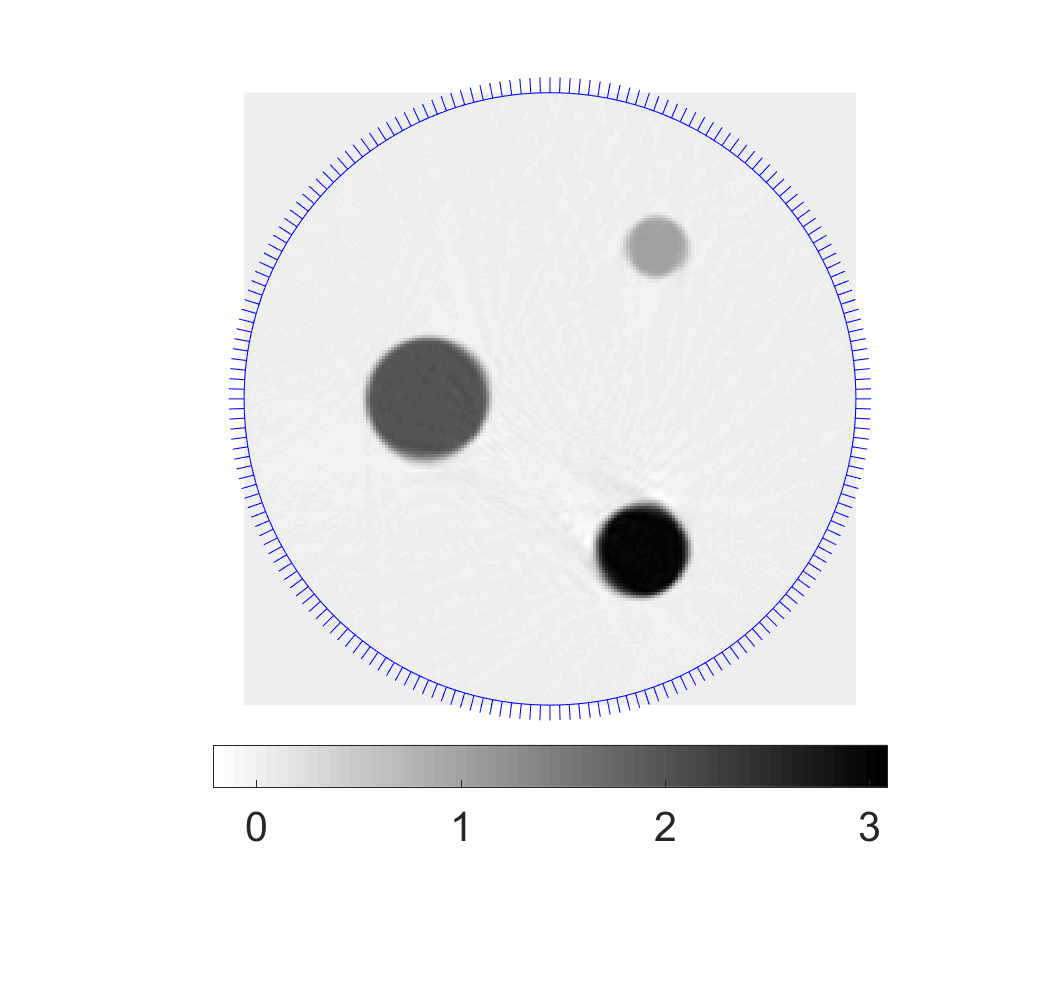}
    \caption{An example of approximate inversion. Left: $I_0 f$; Middle: $A_+^* H A_- I_0 f$ (the ``filtering'' step); Right: $-(1/8\pi) I_\perp^* A_+^* H A_- I_0 f$ (the ``backprojection'' step). On the left two pictures, the horizontal axis describes the boundary point from where geodesics are cast; the vertical axis describes the shooting direction. See \cite{Monard2013} for details of the implementation.}
    \label{fig:FBP}
\end{figure}

\begin{remark} Formulas \eqref{eq:reconsfh} hint at us that the classical {\em filtered-backprojection} formula inverting the 2D Radon transform $f = \frac{1}{4\pi} R^* H \partial_s R f$, contains two formulas for the price of one: indeed, $\partial_s Rf$ can also be viewed as $R_\perp f := R(X_\perp f)$, where $X_\perp f = \sin\theta \partial_x f - \sin\theta \partial_y f$ is the restriction to $SM$ of the solenoidal one-form $- \partial_y f dx + \partial_x f dy$, generated by the solenoidal potential $f$. In that case, the solenoidal potential $X_\perp f$ can be reconstructed by first reconstructing $f$ via $f = \frac{1}{4\pi} R^* H R_\perp f$.     
\end{remark}

The proof of \eqref{eq:reconsfh} can be found in \cite{Pestov2004} and in the recent form in \cite{Monard2015}, and is based on the interaction of the fiberwise Hilbert transform with the transport equation \eqref{eq:transport}, in particular relying heavily on the commutator formula below, first derived in~\cite{Pestov2005} 
\begin{align}
    [H,X] u = (X_\perp u)_0 + X_\perp u_0, \qquad u\in C^\infty(SM).
    \label{eq:commutator}
\end{align}
This is a commutator formula involving a distinguished non-local zeroth order $\Psi$DO, and is used in a similar fashion to the local commutator formulas of lemma~\ref{lma:hd-commutator}.
No higher dimensional analogue for this non-local formula is known.

\subsubsection{Analysis of $W$ and $W_\perp$}

Equations \eqref{eq:reconsfh} in fact make sense on any non-trapping surface. On the other hand, one must further assume that the surface is simple to establish that $W,W_\perp$ are compact. In general, as is the case with other operators emanating from this context, the operators $W,W_\perp$ are integral operators with Schwartz kernels naturally written in exponential coordinates. Namely, such operators take the form 
\begin{align*}
    Kf(x) = \frac{1}{2\pi} \int_{S_x} \int_0^{\tau(x,v)} k(x,v,t) f(\gamma_{x,v}(t))\ dt\ dS(v),
\end{align*}
where $k(x,v,t)$ is the Schwartz kernel of $K$ in exponential coordinates. In the absence of conjugate points, the mapping $(v,t) \mapsto \gamma_{x,v}(t)$ is a global diffeomorphism onto $M$ with jacobian $dM_x = b(x,v,t)\ dt\ dS(v)$ and inverse denoted $M\ni x' \mapsto (v_x(x'), t = d_g(x,x'))$, so that the actual Schwartz kernel of $K$ is, up to a constant, ${\cal K}(x,x') = \frac{k(x,v_x(x'), d_g(x,x'))}{b(x,v_x(x'), d_g(x,x'))}$. A sufficient condition for $K$ to be compact is if ${\cal K}\in L^2(M\times M)$.

In the setting of equations \eqref{eq:reconsfh}, one may show (see e.g. \cite{Pestov2004}) that $W,W_\perp$ have respective kernels
\begin{align*}
    w(x,v,t) = V \left( \frac{a}{b} \right) (x,v,t), \qquad w_\perp(x,v,t) = V \left( \frac{1}{b} \right) (x,v,t),
\end{align*}
with $(a,b)$ solving \eqref{eq:scalarJacobi}. Then by theory of parameter-dependent ordinary differential equations, it is easy to show that $w$ and $w_\perp$ vanish of order $1$ at $t=0$, so that $\frac{w}{b}$ and $\frac{w_\perp}{b}$ are both bounded and continuous on $\D$. If further, $M$ is simple, one may express $W$ and $W_\perp$ in terms of well-defined Schwartz kernels bounded and continuous on $M\times M$, making both operators $L^2(M)\to L^2(M)$ compact (in fact, one may show that they are $C^\infty$ smoothing, see \cite{Pestov2004,Guillarmou2017}).

To obtain injectivity of the equations \eqref{eq:reconsfh}, if $\kappa$ is constant, then $w,w_\perp \equiv 0$, and if $\sup_{M} |d\kappa|$ is small enough, then $W,W_\perp$ are contractions (see \cite{Krishnan2010}), and recovery of $f$ or $h$ from \eqref{eq:reconsfh} can be done via Neumann series. In such cases, this in fact gives another mechanism to prove injectivity of $I_0$ and $I_\perp$ than energy identities, though for now, it is open as to whether $Id + W^2$ and $Id + W_\perp^2$ are injective for all simple surfaces. A quantitative bound estimating the norm of $W$ in a neighborhood of constant curvature, simple surfaces was given in \cite[Appendix A]{Monard2016d}: if $\tau_\infty$, $\Vol{M}$ denote diameter and volume of $M$, and if $(M,g)$ is simple with constants $(C_1,C_2)$ in the sense that 
\begin{align}
    C_1 t \le |b(x,v,t)| \le C_2 t, \qquad (x,v,t)\in \D,
    \label{eq:simpleConstants}
\end{align}
then one may obtain the estimate
\begin{align}
    \|W\|_{L^2\to L^2} \le \frac{C_2^3 \tau_\infty^2}{24 C_1^{5/2}} \left( \frac{\Vol{M}}{2\pi} \right)^{\frac{1}{2}} \|d\kappa\|_\infty.
    \label{eq:estW}
\end{align}

\subsection{Transforms over $k$-differentials on simples surfaces}
Following the template outlined above, the second author generalized in \cite{Monard2013a} the inversion of $I_0$ and $I_\perp$ to inversion formulas for the recovery of sections of $\Omega_k$ and their horizontal derivatives, with $k\in \Zm$ fixed for this paragraph. Namely, for $u \in \Omega_k$ (locally of the form $f(x,y) e^{ik\theta}$ in isothermal coordinates), one may define
\begin{align*}
    I_k u := \I [u], \qquad I_{k,\perp} u := \I [X_\perp u],
\end{align*}
and upon introducing a {\em shifted Hilbert transform} $H_{(k)}$ by the formula 
\begin{align*}
    H_{(k)} u_\ell := -i \sgn{\ell-k} u_\ell, \qquad u_\ell \in \Omega_\ell,
\end{align*}
one may prove a commutator relation $[H_{(k)},X] u = (X_\perp u)_k + X_\perp u_k$ which allows Fredholm equations of the form \eqref{eq:reconsfh} pseudo-inverting $I_k$ and $I_{k,\perp}$, modulo compact error operators $W_k$ and $W_{k,\perp} = W_k^*$. The kernel of $W_k$ in exponential coordinates is given by
\begin{align*}
    w_k (x,\theta,t) = \left( -\partial_\theta \left( \frac{a}{b} \right) + ik \frac{a-1}{b} \right) e^{ik(\alpha_{x,\theta}(t)-\theta)},
\end{align*}
where $\alpha_{x,\theta}(t)$ is the angle of $\dot \gamma_{x,\theta}(t)$ with $\partial_x$ in isothermal coordinates. Further study of these kernels gives exact reconstructions via Neumann series in \cite[Corollary 5.9]{Monard2013a} under certain assumptions on the curvature.

\subsection{Transforms with non-unitary connections on simple surfaces}
A case encompassing both previous ones is to consider transforms with connections. Namely, given $M\times \Cm^n\to M$ the trivial\footnote{All vector bundles are trivial when $M$ is simply connected.} vector bundle of rank $n$ and a connection $A$ given by a $n\times n$ matrix of one-forms on $M$, and $f\in L^2(SM,\Cm^n)$, one may now solve the equivalent of a coupled system of transport equations for $u\colon SM\to \Cm^n$ by
\begin{align*}
    Xu + A_x(v) u = -f \qquad (SM), \qquad u|_{\partial_- SM} =0,
\end{align*}
and define $\I_A f := u|_{\partial_+ SM}$, the ray transform of $f$ with connection $A$. Note in what follows that we will abuse notation $A(x,v) \equiv A_x(v)\in \Cm^{n\times n}$. Naturally, $\I_A$ is continuous in the $L^2(SM,\Cm^n) \to L^2_{\mu} (\partial_+ SM, \Cm^n)$ setting, though a dimension count shows that one may not recover all of $f$ from $\I_A f$, and must therefore restrict to certain classes of $f$. In particular one may define
\begin{align*}
    I_{A,0} f &:= \I_A [f\circ \pi], \quad f\in L^2(M,\Cm^n), \\
    I_{A,\perp} h &:= \I_A [(X_\perp - A_V) (h\circ \pi)], \quad h\in H^1_0(M,\Cm^n),
\end{align*}
the restrictions of $\I_A$ to functions and certain one-forms (of the form $\star d h - A_V h$). In the same way that curvature of $(M,g)$ has an impact on X-ray transforms, the curvature of the connection $F_A := dA + A\wedge A$ (i.e., a matrix of $2$-forms with components $(F_A)_{ij} = dA_{ij} + \sum_{k=1}^n A_{ik}\wedge A_{kj}$) will have an impact on X-ray transforms via the function $\star F_A\colon M\to \Cm^{n\times n}$. 

Such transforms generalize the case of symmetric differentials because if $\phi$ denotes a non-vanishing section of $\Omega_1$ (in isothermal coordinates, take $e^{i\theta}$), then every element of $\Omega_k$ can be uniquely written as $f \phi^k$ for some function $f:M\to \Cm$, and the transport equation $Xu = -f \phi^k$ (with $u|_{\partial_{-SM} = 0}$), upon setting $v = \phi^{-k}u$, is equivalent to the transport equation $Xv + k (\phi^{-1} X\phi) v = -f$ (with $v|_{\partial_{-SM}}=0$) on the trivial bundle $SM\times \Cm$ with connection $k \phi^{-1} d\phi$. In particular, we have
\begin{align*}
    I_{k,0} [f\phi^k] = u|_{\partial_+ SM} = \phi^k|_{\partial_+ SM} v|_{\partial_+ SM} = \phi^k|_{\partial_+ SM} I_{k\phi^{-1} d\phi, 0} f,
\end{align*}
and the two transforms are strictly equivalent for injectivity and inversion purposes.

Upon introducing the fiberwise Hilbert transform, acting this time on each component of $\Cm^n$, one may derive the commutator formula (see \cite{Paternain2012})
\begin{align*}
    [H,X+A]u = (X_\perp - A_V)u_0 + ((X_\perp - A_V)u)_0, \qquad u\in C^\infty(SM,\Cm^n),
\end{align*}
and derive the inversion formulas (see \cite[Theorem 1]{Monard2016d}):
\begin{align}
    \begin{split}
	f + W_A^2 f &= \frac{1}{8\pi} I_{-A^*,\perp}^* B_{A,+} H Q_{A,-} I_{A,0} f, \qquad f\in C^\infty(M,\Cm^n), \\
	h + W_{A,\perp}^2 h &= \frac{-1}{8\pi} I_{-A^*,0} B_{A,+} H Q_{A,-} I_{A,\perp} h, \qquad h\in C_0^\infty(M,\Cm^n),	
    \end{split}
    \label{eq:FredIA}    
\end{align}
extendible by density to $f\in L^2(M,\Cm^n)$ and $h\in H^1_0(M,\Cm^n)$, and where the error operators admit respective kernels 
\begin{align*}
    w_A(x,v,t) &= \left(X_\perp - A_V + \frac{a}{b} V + V\left( \frac{a}{b} \right)\right) E_A^{-1}(x,v,t), \\
    w_{A,\perp}(x,v,t) &= E_A^{-1} (x,v,t) \left( V \left( \frac{1}{b} \right) - A_V (\varphi_t) \right) + \frac{1}{b} V(E_A^{-1}(x,v,t)),
\end{align*}
where $E_A(x,v,t)$ is the attenuation matrix solving the $(x,v)$-dependent ODE
\begin{align*}
    \frac{d}{dt} E_A(x,v,t) + A(\varphi_t(x,v)) E_A (x,v,t) = 0, \quad (x,v,t) \in \D, \quad E_A(0,x,v) = I_n.
\end{align*}
Note that $w_A,w_{A,\perp}\colon \D\to \Cm^{n\times n}$ are again such that $\frac{w_A}{b}$ and $\frac{w_{A,\perp}}{b}$, composed with the inverse of the exponential map $(v,t)\mapsto \gamma_{x,v}(t)$ are uniformly bounded on $M\times M$, in particular the kernels of $W_A,W_{A,\perp}$ belong to $L^2(M\times M)$ so that $W_A,W_{A,\perp}$ are compact.

By the Fredholm alternative, this implies that $\ker I_{A,0}\subset \ker (Id + W_A^2)$ and $\ker I_{A,\perp}\subset \ker (Id + W_{A,\perp}^2)$ are finite-dimensional. Varying connections with a complex parameter, one may combine this with Analytic Fredholm Theory to enlarge the known cases of injective ray transforms with connection. The steps go as follows: 

\begin{theorem}[{\cite[Theorem 3]{Monard2016d}}]\label{thm:AFT} For any analytic $C^1(M,(\Lambda^1)^{n\times n})$-valued family of connections $\lambda\mapsto A_\lambda$, the corresponding $L^2(M,\Cm^n)\to L^2(M,\Cm^n)$-valued families of operators $\lambda\mapsto W_{A_\lambda}$ and $\lambda\mapsto W_{A_\lambda,\perp}$ are analytic.    
\end{theorem}

For $A\in C^1(M,(\Lambda^1)^{n\times n})$ a fixed connection, applying this to a family of the form $\lambda \mapsto \lambda A$ with $A$ implies that if there is $\lambda_0\in \Cm$ such that $Id + W_{\lambda_0 A}^2$ is injective, this remains so for all $\lambda\in \Cm$ except possibly over a discrete set of $\lambda$, and implies the same conclusions for the ray transforms $I_{\lambda_0 A,0}$. An obvious choice is $\lambda =0$, which reconducts the question to a transform with no connection. In that case, we are left inquiring whether $Id + W^2$ is injective. 

\begin{theorem}[{\cite[Theorem 4]{Monard2016d}}]\label{thm:est} Let $(M,g)$ be a simple Riemannian surface with constants $C_1,C_2$ as in \eqref{eq:simpleConstants} and Gaussian curvature $\kappa(x)$. Given the $C^1$ connection $A$ with curvature $F_A$, let us denote $\alpha_A = \sup_{(x,v)\in SM} \{ \|(A+A^*)/2\|(x,v) \}$ and $\tau_\infty$ the diameter of $M$. There exist constants $C,C'$ depending on $(n,C_1,C_2,\tau_\infty,\alpha_A)$ such that 
    \begin{align}
	\|W_A\|_{L^2\to L^2}, \|W_{A,\perp}\|\le \left( \frac{\Vol{M}}{2\pi} \right)^{\frac{1}{2}} \sqrt{ C \|\star F_A\|_\infty^2 + C' \|d\kappa\|_\infty^2}. 
	\label{eq:estWA}
    \end{align}    
\end{theorem}
This gives us a few settings where transforms with connections defined over any structure group are injective. 

\begin{theorem}[{\cite[Theorem 5]{Monard2016d}}]\label{thm:injIA}
    Let $(M,g)$ be a simple surface and $A$ a $C^1$ connection. Then the following conclusions hold: 
    \begin{itemize}
	\item[$(i)$] If $\kappa$ is constant and $A$ is flat, the operators $W_A$ and $W_{A,\perp}$ vanish identically and \eqref{eq:FredIA} implies that the transforms $I_{A,0}$, $I_{A,\perp}$, $I_{-A^*,0}$ and $I_{-A^*,\perp}$ are all injective, with explicit, one-shot inversion formulas.
	\item[$(ii)$] Injectivity still holds if $(n,C_1,C_2,\tau_\infty,\alpha_A,\|\star F_A\|_\infty, \|d\kappa\|_\infty, \Vol{M})$ are such that the right-hand side of \eqref{eq:estWA} is less than 1, with a Neumann series type inversion.
	\item[$(iii)$] If $(M,g)$ is such that the operator $Id + W^2$ in \eqref{eq:reconsfh} is injective, then for every $\lambda\in \Cm$ outside a discrete set, the transforms $I_{\lambda A,0}$, $I_{\lambda A,\perp}$, $I_{-\lambda A^*,0}$ and $I_{-\lambda A^*,\perp}$ are all injective.
    \end{itemize}
\end{theorem}

The fact that the conclusions apply jointly to connections $A$ and $-A^*$ comes from exploring the symmetries in the inversion formulas \eqref{eq:FredIA} and establishing for instance that $(W_A)^* = W_{-A^*,\perp}$, see \cite[Lemma 13]{Monard2016d}.


\subsection{Surfaces with no conjugate points and hyperbolic trapping}
As mentioned in Section~\ref{sec:trapping}, some of the arguments can be adapted to situations when hyperbolic trapping is present, a special case of which is surfaces with negative curvature. In \cite{Guillarmou2017}, Guillarmou and the second author provided reconstruction formulas for functions and solenoidal vector fields from knowledge of their ray transform, in the case of surfaces with convex boundary, no conjugate points and hyperbolic trapping, see \cite[Theorem 1.1]{Guillarmou2017}. The derivation of the formulas is similar to obtaining \eqref{eq:reconsfh}, as it relies on the commutator formula \eqref{eq:commutator} which holds locally and independently of the presence of trapping. The added technicalities are then: the control of the regularity and wavefront sets when writing transport equations as explained in Section \ref{sec:trapping}; the presence of infinite-length geodesics in the kernel of the error operator, requiring further control for the sake of continuity estimates. In particular, one can write explicit estimates in the neighborhood of constant negative curvature metrics, making \eqref{eq:reconsfh} invertible via Neumann series. 

\begin{theorem}[{\cite[Theorem 1.2]{Guillarmou2017}}]\label{thm:controlW}
    Let $(M,g_0)$ be a manifold with strictly convex boundary and constant neative curvature $-\kappa_0$ and trapped set $K$, and let $\delta = \frac{1}{2} (\text{dim}_{\text{Haus}}(K)-1)\in [0,1)$\footnote{In constant negative curvature, the trapped set $K$ is a fractal set with $\text{dim}_{\text{Haus}}(K)\in [1,3)$ when $K\ne\emptyset$.}. Then for each $\lambda_1, \lambda_2\in (0,1)$ so that $1\ge \lambda_1\lambda_2>\max (\delta,\frac{1}{2})$, there is an explicit constant $A(\delta,\lambda_1,\lambda_2)>0$ depending only on $\delta,\lambda_1,\lambda_2$ such that for all metrics $g$ on $M$ with strictly convex boundary and Gauss curvature $\kappa(x)$ satisfying 
	\[ \lambda_1^2 g_0 \le g\le \lambda_1^{-2}g_0, \quad \kappa(x) \le -\lambda_2^2 \kappa_0, \quad |d\kappa|_{\infty} \le A(\delta,\lambda_1,\lambda_2) \kappa_0^{3/2},   \]
	the remainder operator $W$ in \eqref{eq:reconsfh} is an $L^2(M,g)$-contraction and hence $f$ and $h$ are reconstructible from \eqref{eq:reconsfh} via a convergent series. When $\delta<1/2$, the constant $A(\delta,\lambda_1,\lambda_2)$ does not depend on~$\delta$.    
\end{theorem}

Let us mention in passing that proving Theorem \ref{thm:controlW} requires writing continuity estimates for the normal operator $I_0^* I_0$ on a surface of a constant negative curvature, and this leads to a striking expression of $I_0^* I_0$ in terms of the (negative) Laplacian on that surface $\Delta_M$, given by 
\begin{align}
    I_0^* I_0 = 4\frac{\Gamma(\frac{1}{4}-S)\Gamma(\frac{1}{4}+S)}{\Gamma(\frac{3}{4}-S)\Gamma(\frac{3}{4}+S)}, \qquad \text{where} \quad S:= \frac{i\sqrt{-\Delta_{M}-\frac{1}{4}}}{2},
    \label{eq:Pi0}
\end{align}
$\Gamma(\cdot)$ the Euler Gamma function, and with the convention that $\sqrt{s(1-s)-\frac{1}{4}}=i(s-1/2)$ when $s(1-s)\in(0,\frac{1}{4})$ is an $L^2$-eigenvalue of $-\Delta_{M}$ with $s\in (1/2,1)$, see \cite[Lemma A.1]{Guillarmou2017} for more detail. This is to be contrasted with the Euclidean case where $I_0^* I_0 = \sqrt{-\Delta}$, a fact which is often also stated at the level of principal symbols in Riemannian settings, though {\em global} relations such as \eqref{eq:Pi0} show that the relation between $I_0^*I_0$ and the Laplace--Beltrami operator can be intricate.

\subsection{Reconstruction of higher-order tensor fields}

To reconstruct tensor fields from their X-ray transform, it was realized in \cite{Monard2015a} that the solenoidal representative of a tensor field may not lead to the most efficient reconstructions. To this end, we first describe a different gauge than the solenoidal one, which we call the Killing gauge here. The first appearance known to the authors is in \cite[Theorem 1.5]{Dairbekov2011}, stating that every $m$-tensor $f\in H^k(S^m(T^*M))$, $(m\ge 0, k\ge 1)$ can be uniquely represented in the form 
\begin{align*}
    f = \sigma\nabla v + \sigma(g\otimes \lambda) + \tilde f,
\end{align*}
where $v\in H^{k+1}(S^{m-1}(T^*M))$ is trace-free and vanishes at $\partial M$, $\lambda\in H^{k} (S^{m-2}(T^* M))$ and $\tilde f\in H^k(S^m(T^*M))$ is trace-free and divergence-free, and the decomposition is continuous in the appropriate spaces. When restricting all tensors to $SM$, one equivalently has $f\in H^k(\Omega_{m}) \cap \ker X_-$. On the data side, we immediately have $\I[\sigma\nabla v] = 0$ and thus 
\begin{align*}
    \I f = \I[\sigma(g\otimes \lambda)] + \I[\tilde f] = \I[\lambda] + \I[\tilde f].
\end{align*}
An advantage of this decomposition for inversion purposes, found in \cite{Monard2015a}, is that in the Euclidean case, the two components in the last right hand side are orthogonal for the $L^2(\partial_+ SM)$ topology. While not orthogonal in the usual target space $L^2_\mu(\partial_+ SM)$, this still gives us a direct decomposition, allowing one to break down the tensor tomography problem into smaller pieces, both for inversion and range characterization purposes. To see the gist of the method on the Euclidean unit disk, given a $2m$-tensor field $f$ restricted to $SM$ with $L^2$ components, one may iterate the decomposition above to find a $2m$-tensor $g$ with $L^2$ components such that $\I f = \I g$, and $g$ is of the form 
\begin{align*}
    g = g_0 + \sum_{k=1}^m g_{2k}, \qquad g_0 \in L^2(M), \qquad g_{2k} \in L^2(\Omega_{2k}) \cap \ker X_-, \qquad 1\le k\le m.
\end{align*}
Moreover, the ray transforms of each component are $L^2(\partial_+ SM)$-orthogonal, see \cite[Theorems 2.1, 2.2]{Monard2015a}. Then each $g_{2k}$ can be reconstructed explicitly via Cauchy-type integrals (its components are complex-analytic or anti-analytic in the base coordinates), while $g_0$ (a full $L^2(M)$ function) can be inverted via applying $I_0^{-1}$ to the data, see \cite[Theorems 2.4]{Monard2015a}. A similar story holds for odd-order tensors. In the case of the Euclidean disk, the approach was generalized to transforms with arbitrary position-dependent attenuation in \cite{Monard2017a}, leading to explicit and efficient inversions of the attenuated X-ray transform over tensor fields of arbitrary order. To reconstruct the $g_{2k}$ components, special care must be paid in constructing special invariant distributions with specified harmonic content. This is the topic of the next section.

\np
\section{Tensor tomography and special invariant distributions}\label{sec:invariant}

\subsection{An equivalence principle}

As providing a potential alternate route toward proving tensor tomography on simple manifolds, the following theorem was provided in \cite{Paternain2016}. In the statement, $L_m$ denotes the $L^2$-$L^2$ adjoint of the operator $\ell_m$ defined in~\eqref{eq:ellm}.

\begin{theorem}[{\cite[Theorem 1.2]{Paternain2016}}]\label{thm:equivalence} Let $M$ a compact simple Riemannian manifold, then the following are equivalent
    \begin{enumerate}
	\item[(1)] $I_m$ is s-injective on $C^\infty(S^m(T^*M))$; 
	\item[(2)] for every $u\in L^2(S^m_{sol}(T^*M))$, there exists $f\in H^{-1}(\partial_+ SM)$ such that $u = I^*_m \varphi$. 
	\item[(3)] for every $u\in L^2(S^m_{sol}(T^*M))$, there exists $f\in H^{-1}(SM)$ satisfying $Xf = 0$ and $u=L_m f$.  
	\item[(4)] for every $u\in C^\infty(S^m_{sol}(T^*M))$, there exists $\varphi\in C^\infty_\alpha(\partial_+ SM)$ such that $u = I_m^* \varphi$. 
	\item[(5)] for every $u\in C^\infty(S^m_{sol}(T^*M))$, there exists $f\in C^\infty(SM)$ with $Xf = 0$ such that $L_m f = u$. 
    \end{enumerate}
\end{theorem}

In the theorem above, $f$ is an invariant distribution for the geodesic flow, and the condition $L_m f = u$ is a prescription on its harmonic content. Constructing such invariant distributions not only provides another way to look at the problem, but their explicit construction provides an immediate way to reconstruct explicit features of the unknown tensor: if $u$ is an unknown solenoidal $m$-tensor and one knows an invariant distribution $f = I^* \varphi$ such that $L_m f = v$ for some $v\in C^\infty(S^m_{sol}(T^*M))$, then the inner product $\dprod{u}{v}_{S^m(T^*M)}$ is known from the chain of equalities
\begin{align*}
    \iip{u}{v}_{S^m(T^*M)} = \iip{u}{L_m f}_{S^m(T^*M)} = \iip{\ell_m u}{f}_{SM} = \iip{\ell_m u}{I^* \varphi}_{SM} = \iip{I_m u}{\varphi}_{\partial_+ SM}.
\end{align*}

The construction of such invariant distributions, or equivalently the surjectivity of adjoints to integral transforms has appeared in several places as explained in the next paragraphs. 

\subsection{Surjectivity of adjoints using microlocal arguments}

On simple surfaces, the surjectivity of $I_0^*\colon C_\alpha^*(\partial_+ SM) \to C^\infty(M)$ was proved in \cite[Theorem 1.4]{Pestov2005} via microlocal arguments, where 
\begin{align*}
    C_\alpha^\infty(\partial_+ SM) := \{h\in C^\infty(\partial_+ SM),\ h_\psi \in C^\infty(SM)\}.
\end{align*}
Specifically, the normal operator $I_0^* I_0$ is elliptic on a simple open neighborhood of $M$, and can be extended into an invertible operator on the closed ``double'' of $M$. After appropriate restriction, this allows to construct a right-inverse for $I_0^*$. This scheme of proof for simple surfaces was further used in the following contexts: 
\begin{itemize}
    \item In \cite[Lemma 4.5]{Salo2011}, it is proved that the operator $I_\perp^*\colon C_\alpha^\infty(\partial_+ SM) \to C^\infty(M)$ is surjective for $M$ a simple surface.
    \item In \cite[Theorem 5.4]{Paternain2013a}, it is proved that the operator $I_{0,A}^*\colon \SS^\infty (\partial_+ SM, \Cm^n)\to C^\infty(M,\Cm^n)$ is surjective for $M$ a simple surface and $A$ a unitary connection, where the space $\SS^\infty (\partial_+ SM, \Cm^n)$ is a natural generalization of $C_\alpha^\infty(\partial_+ SM)$ to the case with connection. 
\end{itemize} 

\subsection{Iterated Beurling series} 
\label{sec:Beurling}

A building block toward fulfilling (3) or (5) in Theorem \ref{thm:equivalence} amounts to seeking an invariant distribution of the form $w := w_{k_0} + w_{k_0+2} + w_{k_0+4} + \dots$ such that $w_{k_0} = u \in H_{k_0}$ is prescribed and $Xw = 0$. In light of \eqref{eq:tridiagonal}, this implies 
\begin{align*}
    X_- w_{k_0} = 0, \quad \text{and} \quad X_- w_{k_0 + 2(m+1)} = - X_+ w_{k_0 + 2m}, \qquad m = 0,1,\dots
\end{align*}
The first equation is a requirement on $u$, while the second family of equations suggests to construct $w_{k_0+2}$ from $w_{k_0}$, then $w_{k_0+4}$ from $w_{k_0+2}$, etc\dots provided that one can 'invert' $X_-$ on $X_+ \Omega_p$ in some sense. This is the purpose of the {\bf Beurling transform}, defined for any $p\ge 0$ as
\begin{align*}
    B\colon \Omega_p\to \Omega_{p+2}, \qquad f_p\mapsto f_{p+2},
\end{align*}
where $f_{p+2}\in \Omega_{p+2}$ is the unique solution to the equation $X_- f_{p+2} = -X_+ f_p$ that is orthogonal to $\ker^{p+2} X_-$ (equivalently, the unique solution with minimal $L^2$ norm). That this is well-defined follows from \cite[Lemma 11.1]{Paternain2015}. With $B$ defined as above, the following theorem is an example where the construction of such invariant distributions is well-understood, done via formal iterated Beurling series. 

Let $C_n(m)$ denote the constant
\begin{equation}
C_n(m)
=
\begin{cases}
\sqrt2,& m=0\text{ and }n=2\\
\left[1+\frac1{(2m+1)(m+2)^2}\right]^{1/2},& m\geq0\text{ and }n=3\\
1,& \text{otherwise}.
\end{cases}
\end{equation}

\begin{theorem}[{\cite[Theorem 11.4]{Paternain2015}}]\label{thm:Beurling} Let $(M,g)$ be a compact manifold with boundary, and assume that the sectional curvatures are non-positive. Then 
    \[  \|X_- u\|_{L^2} \le C_n(m) \|X_+ u\|_{L^2}, \qquad u\in \Omega_m, \quad u|_{\partial SM} = 0, \quad m\ge 0.   \]
    The Beurling transform satisfies
    \[ \|Bf\|_{L^2} \le C_n(m) \|f\|_{L^2}, \qquad f\in \Omega_m, \quad m\ge 0.   \]
    If $k_0\ge 0$ and if $f\in \Omega_{k_0}$ satisfies $X_- f = 0$, then there exists a solution of $Xw = 0$ in $SM$ such that $w_{k_0} = f$, given by $w = \sum_{k=0}^\infty B^k f$. One has $w\in L^2_x H_v^{-\frac{1}{2}-\varepsilon}(SM)$ for any $\varepsilon>0$, and the Fourier coefficients of $w$ satisfy 
    \begin{align*}
	\|w_k\|_{L^2}\le \|f\|_{L^2}, \quad k\ge k_0.
    \end{align*}    
\end{theorem}
In the theorem above, we have introduced the mixed norm spaces
\begin{align}
    L^2_x H_v^s (SM) = \left\{u\in {\cal D}'(SM)\colon \|u\|_{L^2_x H^s_v}<\infty\right\}, \qquad \|u\|_{L^2_x H^s_v}^2 := \sum_{p=0}^\infty \langle p \rangle^{2s} \|u_p\|_{L^2}^2,     
    \label{eq:L2xHsv}
\end{align}
where as usual $\langle p \rangle = (1+p^2)^{1/2}$.

We point out that the estimate $\aabs{X_-u_k}_{L^2}\lesssim\aabs{X_+u_k}_{L^2}$ can be useful and true even in settings where the Beurling transform itself is not needed.

\subsection{Explicit constructions on the Euclidean disk} 

The constructions above via formal Beurling series is explicit if one understands the Beurling transform explicitly, which is only well-understood in some cases yet little documented. Another approach toward building such invariant distributions is done in the case of the Euclidean disk, provided by the second author in \cite[Theorem 5.2]{Monard2017a}. There, it is enough to consider finding invariant distribution with prescribed complex-analytic, $L^2(\Dm)$ average, and this serves as a crucial building block to write an explicit inversion of the attenuated tensor tomography problem over tensors of any order.

When $M=\Dm$, parameterizing $\partial_+ SM$ in fan-beam coordinates $(\beta,\alpha) \in \Sm^1\times (-\pi/2,\pi/2)$, borrowing notation from \cite{Monard2017a}, we define for $k=0,1,\dots$
\begin{align*}
    Z_k (x,y) &:= \frac{\sqrt{k+1}}{2\pi^2} (x+iy)^k, \qquad (x,y)\in \Dm, \\
    W_k(\beta,\alpha) &:= (-1)^k \frac{\sqrt{k+1}}{2\pi\sqrt{2}} e^{ik\beta} (e^{i(2k+1)\alpha}+ (-1)^k e^{-i\alpha}), \qquad (\beta,\alpha)\in \partial_+ SM.
\end{align*}

Then \cite[Proposition 4]{Monard2017a} shows that for every $k$, $\left( \left( \frac{W_k}{\cos\alpha} \right)_\psi \right)_0 = Z_k$. Moreover, we have the following theorem, reformulated here in terms of invariant distributions.

\begin{theorem}[{\cite[Theorem 5.2]{Monard2017a}}]\label{thm:I0star}
    For any $f\in L^2(\ker \overline{\partial})$, given by $f = \sum_{k=0}^\infty \iip{f}{Z_k}_{SM} Z_k$, the function $W_f\in \ell^2( \{W_k\}_{k=0}^\infty )$ given by  
    \begin{align}
	W_f := \sum_{k=0}^\infty \iip{f}{Z_k}_{SM} W_k, \quad \text{satisfies} \qquad \left( \left( \frac{W_f}{\cos\alpha} \right)_\psi \right)_0 = f. 
	\label{eq:Wf}
    \end{align}
    Moreover, the distribution $\left( \frac{W_f}{\cos\alpha} \right)_\psi$ is fiberwise holomorphic and orthogonal to $\ker^k \eta_-$ for any $k>0$.
\end{theorem}

The last claim implies that such invariant distributions have minimal norm in some sense. In addition, one may show that such distributions make sense in $L^2_x H_v^{-1/2-\varepsilon}(SM)$ (as defined in \eqref{eq:L2xHsv}) for every $\varepsilon>0$, and that this is sharp. Notice also that each invariant function $\left( \frac{W_k}{\cos\alpha} \right)_\psi$ belongs to $C^\infty(SM)$. For $k>1$, since $\ker^k X_- = \ker^k \eta_- + \ker^{-k} \eta_+$, the theorem above can be viewed as a building block to construct invariant distribution with prescribed $k$th moment in $\ker^k X_-$.

\subsection{Anosov flows on closed manifolds}

While not covered in detail in this review, invariant distributions also play a role in solving integral geometric problems on a closed {\bf Anosov}\footnote{The geodesic flow is Anosov if there is a continuous invariant splitting $T(SM) = \Rm X \oplus E^u \oplus E^s$ and constants $C>0$ and $0<\rho<1<\eta$ such that for all $t>0$, $\|d\varphi_{-t}|_{E^u}\|\le C \eta^{-t}$ and $\|d\varphi_{t}|_{E^s}\|\le C \rho^t$.} Riemannian manifold $(M,g)$. Namely, there is a notion of X-ray transform, where one integrates a function or tensor field over all possible closed geodesics, which appears when considering the linearization of spectral rigidity questions ("is a metric uniquely determined up to gauge, from the spectrum of its Laplace--Beltrami operator? or from its so-called marked length spectrum?"), see e.g.~\cite{Guillemin1980}. 

In this context, injectivity of the ray transform considered is again linked in \cite{Paternain2014} to the existence of certain invariant distributions. Specifically, in the context of Anosov surfaces in \cite[Theorems 1.4, 1.5, 1.6]{Paternain2014}, the authors establish the existence of distributional solutions of $Xw = 0$ with prescribed zeroth moment in $C^\infty(M)$ (\cite[Theorems 1.4]{Paternain2014}) and first moment as a prescribed smooth solenoidal one-form (\cite[Theorems 1.4]{Paternain2014}). Under additional conditions on $\Ter$, \cite[Theorem 1.6]{Paternain2014} establishes the existence of invariant distributions with prescribed fiberwise moments on order $2$. This is the case of interest for the non-linear problem as it relates to the X-ray transform over second-order tensors. 

All distributions mentioned above live in the space $H^{-1}(SM) = (H^1(SM))'$, where $H^1(SM)$ is the completion of $C^\infty(SM)$ with respect to the norm 
\begin{align*}
    \|u\|_{H^1}^2 := \|u\|^2 + \|Xu\|^2 + \|X_\perp u\|^2 + \|Vu\|^2.    
\end{align*}
This is to be contrasted to the previous paragraphs (involving an $L^2_x H_v^{-1/2-\varepsilon}(SM)$ norm) which may indicate that the norm can be sharpened. However the present method, based on Pestov identities (and duality arguments {\it \`a la} Hahn-Banach), is not yet amenable to fractional Sobolev norms. 

The results above, namely regarding the existence of invariant distributions with prescribed zeroth or first fiberwise moments, were generalized in \cite[Theorems 1.7, 1.8]{Assylbekov2014} to the case of Anosov thermostat flows (see also Section~\ref{sec:PestovMagnetic}) on closed Riemannian surfaces.

\np

\section{Microlocal Methods}
\label{sec:microlocal}

This section is devoted to microlocal methods, first introduced in integral geometry by Guillemin. Here the results are based on the description of a ray transform $I$ (over functions or tensor fields of a fixed degree) or its corresponding normal operator $I^* I$ in certain classes of operators, namely the first one as a Fourier Integral Operator (FIO), and the second one as a pseudo-differential operator ($\Psi$DO) when the family of curves is simple, or in more general classes if the geometry is more complex. Here, microlocal analysis helps one to prove finiteness theorems or Fredholmess (invertibility of the problem up to a finite-dimensional kernel made of smooth ghosts) under geometric restrictions, to study how to recover the singularities\footnote{By 'singularity' here we mean element of the wavefront set.} of the unknown from the singularities of the data, and to describe how and when this is not possible. Such analysis ultimately provides stability estimates, allows to work with weights in the transform, and to study partial data problems. One can also prove injectivity results in the case where the family of curves and weights is real-analytic, using analytic microlocal analysis. 

\subsection{Results for complete families of curves} \label{sec:microlocalsimple}

A sufficient-but-not-necessary microlocal condition for stability and sometimes injectivity, first formulated in \cite{Stefanov2008}, is that the family of geodesics be {\bf geodesically complete}, in the sense that for every $x\in M$ and every $\omega\in T^* M$, there exists a geodesic $\gamma$ free of conjugate points passing through $x$ and normal\footnote{Specifically, there exists $t$ such that $\gamma(t) = x$ and $\omega(\dot \gamma(t)) =0$.} to $\omega$.  Such a condition ensures that integrals over geodesics in a neighbourhood of $\gamma$ allow to resolve possible singularities at $\omega$ without creating artifacts elsewhere. If the metric is analytic, one may use analytic microlocal analysis to prove injectivity of the ray transform over any complete complex of geodesics, see \cite[Theorem 1]{Stefanov2008}. Then using stability estimates which remain true under $C^k$-perturbation of the metric for $k$ large enough, one may promote this to a generic injectivity result for functions but also for solenoidal tensor fields of order up to two, see \cite[Theorems 2, 3]{Stefanov2008}.

Such results have been generalized to the case of magnetic flows in \cite{Dairbekov2007} and for generic general families of curves and weights including analytic ones in \cite{Frigyik2008}. For instance in \cite{Frigyik2008}, the completeness condition can also be written for a general family of curves, upon defining simple curves and assuming that the conormal bundle of all simple curves covers $T^* M$, and the analysis can be made robust to transforms with smooth weight $\phi\colon SM\to \Rm$
\begin{align*}
    I_{\phi,0} f(\gamma) = \int_{\gamma} f(\gamma(t)) \phi(\gamma(t), \dot \gamma(t))\ dt, \qquad \gamma\in \G.
\end{align*}
If the curves are geodesics and the metric is simple, the associated normal operator $I_{\phi,0}^* I_{\phi,0}$ is a $\Psi$DO of order $-1$, with principal symbol
\[  a_{-1}(x,\eta) = 2\pi \int_{\ker\eta} |\phi(x,v)|^2\ dS(v), \qquad (x,\eta)\in T^* M,  \]
where we have defined $\ker \eta := \{v\in S_x M, \eta(v) = 0\}$. In particular, if $\phi$ is {\bf admissible} in the sense that for every $x\in M$ and $\eta\in T^*_x M$, there exists $v\in S_xM$ such that $\eta(v) = 0$ and $\phi(x,v)\ne 0$, the normal operator is elliptic therefore the problem is Fredholm and H\"older-stable on a complement of the (at most finite-dimensional) kernel, see also \cite{Stefanov2004,Stefanov2005}. If $I_\phi$ is defined over tensor fields instead, the normal operator is elliptic over divergence-free tensors in the interior of~$M$.

\subsection{Mapping properties of the normal operator and an Uncertainty Quantification result on simple manifolds} \label{sec:UQ}

\paragraph{Mapping properties of $I_0^* I_0$ in the simple case.}

In the case of simple manifolds, recent sharp mapping properties of the normal operator $I_0^* I_0$ were obtained in \cite{Monard2017}. Specifically, calling $d_M$ any positive $C^\infty$ function that equals $dist(\cdot,\partial M)$ near the boundary, the following mapping properties were derived: 
\begin{theorem}[{\cite[Theorem 4.4]{Monard2017}}]\label{thm:mappingN} The operator $I_0^* I_0$ is an isomorphism in the following functional settings: 
    \begin{align}
	I_0^* I_0 &\colon d_M^{-1/2} C^\infty(M) \to C^\infty(M). \label{eq:mapping2} \\
	I_0^* I_0 &\colon H^{-1/2(s)} (M) \to H^{s+1}(M), \qquad s>-1, \qquad (\text{bi-continuous}) \nonumber
    \end{align}
\end{theorem}
The theorem above also applies to X-ray transforms with attenuation whenever they are injective. Looking at Eq. \eqref{eq:mapping2}, we see that it is in fact natural for the integrand to have a certain blow-up at the boundary. The spaces $H^{-1/2(s)} (M)$, so-called H\"ormander $\frac{-1}{2}$-transmission spaces, are Hilbert spaces whose elements are morally $H^{s}$ inside of $M$, with some special behavior near the boundary, see \cite{Grubb2014,Grubb2015}. The proof of Theorem \ref{thm:mappingN} is based on the fact that an extension $P$ of $I_0^* I_0$ satisfies a so-called $\mu$-transmission condition\footnote{a generalization of Boutet de Monvel's {\em transmission condition} which corresponds to the case $\mu =0$.} with $\mu=-1/2$ with respect to $\partial M$ in the sense that 
\begin{align*}
    \partial_x^\beta \partial_\xi^\alpha p_j (x,\nu_x) = e^{\pi i (m-2\mu-j-|\alpha|)} \partial_x^\beta \partial_\xi^\alpha p_j(x,-\nu_x), \qquad x\in \partial M, \quad j\ge 0,
\end{align*}
for all multi-indices $\alpha, \beta$, where $p \sim \sum_{j=0}^\infty p_j$ is the full symbol of $P$, of order $m=-1$ so that $p_j(x,t\xi) = t^{-1-j} p_j(x,\xi)$. Such a condition makes the operator $P$ Fredholm in the functional settings above, after which one proves that the kernel and co-kernel are trivial. 

\paragraph{Statistical interlude: regularization of noisy inversions and Uncertainty Quantification.} 

Theorem \ref{thm:mappingN} allows to give the first rigorous statistical approach to regularization in the case of inversions with noisy data in a Riemannian context. The setting is as follows: given $\psi\in C^\infty(M)$, suppose that one attempts at estimating a ``smooth aspect'' $\langle f, \psi \rangle_{L^2(M)}$ from noisy measurements 
\[ Y = I_0 f + \varepsilon \Wm, \] 
where $\varepsilon >0$ is the noise level and $\Wm$ is a Gaussian white noise on the space $L^2_\mu(\partial_+ SM)$. One then adopts a Bayesian approach where, upon a choice of Gaussian prior distribution and noise model, Bayes' formula gives the density of the posterior random variable $f|Y$. One may then attempt to understand where posterior densities concentrate depending on the choice of prior and, under the assumption that a true $f_0$ generated the noisy data, whether the posterior distributions concentrate their mass near~$f_0$. 

In this regard, \cite[Theorems 2.5, 2.7]{Monard2017} provide a limiting behavior of the posterior distribution that is independent of the choice of Gaussian prior for $f$, chosen among a rather loosely constrained family (including all priors modeling Sobolev smoothness), corrected at the boundary to account for the boundary behavior as in Eq. \eqref{eq:mapping2}. Upon choosing such a prior, if $\bar{f}$ is the mean of the posterior distribution and $f_0$ the ground truth, not only do we have convergence in distribution\footnote{${\cal N}(\mu,\sigma^2)$ denotes a Gaussian law (or ``normal density'') of mean $\mu$ and variance $\sigma^2$.} 
\begin{align*}
    \frac{1}{\varepsilon} \langle \bar{f}-f_0, \psi \rangle_{L^2(M)} \to Z \sim {\cal N} (0, \|I_0 (I_0^*I_0)^{-1} \psi\|_{L^2_\mu (\partial_+ SM)}^2), \qquad \text{as } \varepsilon\to 0,
\end{align*}
but in fact the posterior distribution of $\langle f,\psi \rangle_{L^2(M)}$ is approximately a normal density centered at $\langle \bar f,\psi \rangle_{L^2(M)}$ with variance $\|I_0 (I_0^*I_0)^{-1} \psi\|_{L^2_\mu (\partial_+ SM)}^2$. This variance is optimal and its expression justifies why sharp mapping properties of $I_0^* I_0$ were required. Note that since the posterior law is Gaussian here, the posterior mean $\bar f$ agrees with the Maximum A Posteriori (i.e., the argmax of the posterior distribution) which can also be obtained by Tychonov regularization. 

The result above tells us how the mass of the posterior distribution concentrates about the mean, and allows one to use {\bf Bayesian credible intervals} (intervals of the posterior distribution centered at the mean and containing $1-\alpha$ of the total mass with $\alpha\in (0,1)$ a fixed threshold) as approximate {\bf frequentist confidence sets} (sets which contain the ground truth with some given probability). While the latter are usually hard to compute, the former can be obtained by vizualizing ensembles of posterior draws. For uncertainty quantification purposes, the results above tell us that one may use the former intervals to infer where the true unknown lies with a certain probability.

\subsection{Microlocal analysis of cases with conjugate points}\label{sec:microlocalCP}

In the presence of conjugate points, $I_0^* I_0$ is no longer an elliptic $\Psi$DO, and the presence of a conjugate locus can destroy stability. One may view this microlocally by seeing that the Schwarz kernel of $I_0^* I_0$ develops singularities away from the diagonal, and this generates artifact singularities in reconstructions. In some case, they can be removed, either because they are ``weaker'' than the original singularity which generated them, or because the geometry allows to reconstruct the original singularity using integrals over other curves than the ones which generate the artifacts. The latter case only occurs in dimensions greater than three, where the problem is overdetermined and some geometries with conjugate points still allow to restrict the data in a way that fulfills the completeness condition, thereby restoring stability. While this is a somewhat global consideration, the aim of the following sections is to locally describe the ray transform when defined on a neighbourhood of a curve containing conjugate points, and possibly study what these local descriptions imply globally.

The first result in this direction appeared in \cite{Stefanov2012a} in the case of fold caustics, and the analysis was done by studying the structure of the normal operator $I_0^* I_0$. Following this approach, the results were extended to more general types of conjugate points in \cite{Holman2015}. In two dimensions, the results in \cite{Stefanov2012a} were refined in \cite{Monard2013b,Holman2017} by studying the operator $I_0$ as an FIO directly. We present these results in the next few sections.

\subsubsection{FIO's and the clean composition calculus}

One way to make our way up to this description is by expressing $\I$ and $I_0$ as composites of push-forwards and pull-backs by smooth submersions, following \cite{Holman2015,Stefanov2012a}, or equivalently, exploiting the double fibration structure below:
\begin{center}
\begin{tikzcd}[]
&  SM \arrow{dl}[swap]{F}\arrow{dr}{\pi} & \\
\partial_+ SM && M
\end{tikzcd}
\end{center}
Namely, if $(M,g)$ is non-trapping with strictly convex boundary, the canonical projection $\pi\colon SM\to M$ and basepoint map $F\colon SM\to \partial_+ SM$ defined by $F(x,v) = \varphi_{-\tau(x,-v)}(x,v)$ are both smooth submersions with $\pi$ proper, and as such define pull-backs $\pi^*\colon C_c^\infty(M)\to C_c^\infty (SM)$ ($\pi^* f(x,v):= f(x)$) and $F^*\colon C_c^\infty(\partial_+ SM)\to C^\infty(SM)$ ($F^* g (x,v) = g(F(x,v))$), as well as push-forwards $\pi_*\colon C_c^\infty(SM)\to C_c^\infty(M)$ and $F_*\colon C_c^\infty(SM)\to C_c^\infty(\partial_+ SM)$ via the relations
\begin{align*}
    \int_{SM} (\pi^* f) g\ d\Sigma^3 = \int_M f (\pi_* g)\ dVol_g, \qquad  \int_{SM} (F^* h) g\ d\Sigma^3 = \int_{\partial_+ SM} h\ (F_* g) \mu\ d\Sigma^2, 
\end{align*}
to be true for all $f\in C_c^\infty(M)$, $g\in C_c^\infty(SM)$ and $h\in C_c^\infty (\partial_+ SM)$. Note in particular that $\I = F_*$. As explained in \cite{Holman2015}, $\pi_*$ and $\pi^*$ are FIO's of order $\frac{1-n}{4}$ with canonical relations
\begin{align}
    \begin{split}
	\C_{\pi_*} &= \left\{ (\omega, d\pi|_{(x,v)}^T \omega):\ (x,v)\in SM,\ \omega \in T^*_x M \backslash\{0\} \right\} \subset T^* M \times T^* SM,	\\
	\C_{\pi^*} &= \left\{ (d\pi|_{(x,v)}^T \omega, \omega):\ (x,v)\in SM,\ \omega \in T^*_x M \backslash\{0\} \right\} \subset T^* SM \times T^* M,
    \end{split}    
    \label{eq:CRpi}
\end{align}
while $F_*$ and $F^*$ are FIO's of order $\frac{-1}{4}$ with canonical relations
\begin{align}
    \begin{split}
	\C_{F_*} &= \left\{ (\eta, dF|_{(x,v)}^T \eta):\ (x,v)\in SM,\ \eta \in T^*_{F(x,v)} \partial_+ SM \backslash\{0\} \right\} \subset T^* (\partial_+ SM) \times T^* SM,	\\
	\C_{F^*} &= \left\{ (dF|_{(x,v)}^T \eta, \eta):\ (x,v)\in SM,\ \eta \in T^*_{F(x,v)} \partial_+ SM \backslash\{0\} \right\} \subset T^* SM\times T^* (\partial_+ SM).
    \end{split}
    \label{eq:CRF}
\end{align}

Note that ${\cal I} = F_*$ and as such one can tell which singularities are ``collapsed'': given $\eta\in T^*_{\xi} (\partial_+ SM)$, all the singularities $dF|_{\varphi_t(\xi)}^T \eta$ for $t\in (0,\tau(\xi))$ are collapsed into $\eta$. 
Now, to understand $I_0$ as an FIO, one may view $I_0$ as $I_{0} = F_*\circ \pi^*$ and use the {\em clean composition calculus} of Duistermaat and Guillemin to obtain the canonical relation of $I_0$ by computing the composition $\C_{I_0} = \C_{F_*} \circ \C_{\pi^*} \subset T^* (\partial_+ SM) \times T^*M$.

Traditionally, similarly to geodesic completeness, it is of interest that the so-called {\bf Bolker condition} be satisfied in the sense that on the microlocal diagram below
\begin{center}
\begin{tikzcd}[]
    &  \C_{I_0} \arrow{dl}[swap]{}\arrow{dr}{} & \\
T^*(\partial_+ SM) & & T^* M
\end{tikzcd}
\end{center}
the projection $\C_{I_0}\to T^*(\partial_+ SM)$ is an injective immersion. When this the case (for $I_0$ or any other integral operator of interest, see e.g. \cite{Finch2003}), one may use the clean composition calculus again to compute the associated normal operator and show that this is in fact a $\Psi$DO, the most favorable scenario for inversion purposes. For geodesic X-ray transforms, the presence of conjugate points is precisely what invalidates the Bolker condition, and further analysis is required. Similar issues occur in \cite{Felea2007,Stefanov2013a,Zhang2017}. The next two sections summarize what can be said in cases with conjugate points, in dimension two, then three and higher.

\subsubsection{Two dimensions}\label{sec:FIO2D}

In two dimensions, $M$ and $\partial_+ SM$ have the same dimension, and some statements can be made more precise. In this section, we present results from \cite{Stefanov2012a,Monard2013b,Holman2017}. The results in \cite{Monard2013b,Holman2017} are presented in the context of the double fibration of the point-geodesic relation, though for consistency of the present article, we will present the ideas using the $SM$ notation. 

Let $(X,X_\perp,V)$ the canonical frame of $SM$ and $(X^\flat,X_\perp^\flat,V^\flat)$ its dual co-frame with respect to the Sasaki metric\footnote{It corresponds to $(\alpha,-\beta,\gamma)$ in \cite{Merry2011}.}. A basis of $T_{(x,v)}(\partial_+ SM)$ is given by $\left(V_{(x,v)}, T_{(x,v)} := \frac{1}{\dprod{v}{\nu_x}}\nabla_T|_{(x,v)}\right)$, where $\nabla_T|_{(x,v)} = \dprod{-(\nu_x)_\perp}{\nabla}$ (horizontal derivative along the tangent vector). Let us denote $(V^\flat, T^\flat)$ the dual co-frame to $(V,T)$ on $T^*_{(x,v)}\partial_+ SM$. Then the composition $\C_{I_0} = \C_{F_*} \circ \C_{\pi^*}$ can be computed explicitly as follows (the proof is relegated to Appendix~\ref{app:composition-proof}): 

\begin{lemma}\label{lem:composition} The composition $\C_{I_0} = \C_{F_*} \circ \C_{\pi^*}\subset T^* (\partial_+ SM)\times T^* M$ is given by 
    \begin{align}
	\begin{split}
	    \C_{F_*} \circ \C_{\pi^*} &= \Big\{ (\lambda \eta_{x,v,t}, \lambda \omega_{x,v,t}),\ (x,v)\in \partial_+ SM,\ t\in (0,\tau(x,v)),\ \lambda\in \Rm \Big\}, \qquad \text{where} \\
	    \eta_{x,v,t} &:= -b(x,v,t) V^\flat_{(x,v)} + a(x,v,t)T_{(x,v)}^\flat \in T^*_{(x,v)} \partial_+ SM, \\
	    \omega_{x,v,t} &:= (\dot\gamma_{x,v}(t))_\perp^\flat \in T^*_{\gamma_{x,v}(t)} M.
	\end{split}	
	\label{eq:relation}
    \end{align}    
\end{lemma}
Checking the Bolker condition now is just to ask whether, along a geodesic $\varphi_t(x,v)$, two (or more) singularities of the form $\lambda_1 \omega_{x,v,t_1}$ and $\lambda_2 \omega_{x,v,t_2}$ are mapped to the same $\eta\in T^*_{(x,v)} (\partial_+ SM)$? This is equivalent to asking whether the mapping $t\mapsto \eta_{x,v,t}$ in \eqref{eq:relation} is injective. Such a condition is violated precisely when $a(t_2)b(t_1)-a(t_1)b(t_2) = 0$ for some $0<t_1< t_2<\tau(x,v)$, but then this occurs precisely when the non-trivial Jacobi field
\begin{align*}
    J(t) := (a(x,v,t) b(x,v,t_1) - b(x,v,t) a(x,v,t_1)) \dot \gamma_{x,v}(t)_\perp,
\end{align*}
vanishes at $t_1$ and $t_2$, i.e., when $\varphi_{t_1}(x,v)$ and $\varphi_{t_2}(x,v)$ are conjugate, see also \cite[Theorem 4.1]{Monard2013b}. 

In such a case, upon defining $\omega_1 := \lambda_1 \omega_{x,v,t_1}$, $\omega_2 := \lambda_1 \frac{a(t_1)}{a(t_2)} \omega_{x,v,t_2}$ and $\eta\in T^*_{(x,v)}(\partial_+ SM)$ uniquely defined by $\eta_V = -b(t_1)\lambda_1$ and $\eta_T = a(t_1) \lambda_1$, we have that 
\begin{align*}
    (\eta, \omega_1) \in \C_{I_0}, \qquad \text{and} \quad (\eta,\omega_2) \in \C_{I_0},
\end{align*}
describing how both singularities at $\omega_1$ and $\omega_2$ collapse into $\eta$ upon applying~$I_0$. 

\paragraph{Cancellation of singularities and artifact-generating operators.} 

Next, one may exploit that $\C_{I_0}$ is a local graph to construct artifact-generating operators which will produce a proof of global instability of the problem. In the setting above, there exists a conical neighborhood $V$ of $\eta$ in $T^* (\partial_+ SM)$ and $V_{1,2}$, conical neighborhoods of $\omega_{1,2}$ in $T^* M$ such that $\C_{1,2}:= \C \cap (V\times V_{1,2})$ are diffeomorphisms. Then $\C_{21}:= \C_2^{-1}\circ \C_1$ is a canonical relation itself, and a diffeomorphism, and it provides the basis for constructing an artifact-generating FIO $F_{21}\colon H^s (V_1)\to H^s(V_2)$ (for all $s\in \Rm$) with canonical relation $\C_{21}$, see \cite[Theorem 4.3]{Monard2013b}. $F_{21}$ will be such that given $f_1$ a compactly supported distribution with wavefront set included in $V_1$, one may construct the artifact $F_{21} f_1$, whose wavefront set is included in $V_2$, and such that
\[ I(f_1 - F_{21} f_1) \in C^\infty(\partial_+ SM), \]
see also \cite[Corollary 4.1]{Monard2013b}. Such an approach therefore allows to construct distributions whose image is a smooth function, thereby removing any hopes for a global stability estimate on the Sobolev scale, of the form 
\begin{align*}
    \|f\|_{H^{s_1}(M)}\le C \|If\|_{H^{s_2}(\partial_+ SM)} + C' \|f\|_{H^{s_3}(M)}, 
\end{align*}
no matter the choice of indices $s_1,s_2,s_3$. In practice, if $\omega_1$ sits above $x_1$ and $\omega_2$ sits above $x_2$, upon defining $U_{1,2}$ simple neighborhoods of $x_{1,2}$ in $M$, and denoting $r_{U_{1,2}}\colon L^2(M)\to L^2(U_{1,2})$ the restriction operators, the transforms $I_0 r_{U_1}$ and $I_0 r_{U_2}$ are injective and explicitly invertible (if the domain is small enough, the error operators in \eqref{eq:reconsfh} become contractions and thus the Fredholm equations are explicitly invertible), and $F_{21}$ may be constructed as $F_{21} := (I_0 r_{U_2})^{-1}\circ I_0 r_{U_1}$. One may also define $F_{12}:= (I_0 r_{U_1})^{-1}\circ I_0 r_{U_2}$. To express the fact that such operators generates artifacts ``of the same strength'', it was further proved in \cite[Theorem 2.1]{Holman2017} that the operators $F_{21}\colon H^{-1/2}(V^1)\to H^{-1/2}(V^2)$ and $F_{12}\colon H^{-1/2}(V^2)\to H^{-1/2}(V^1)$ were {\bf principally unitary}, in the sense that $F_{12} F_{21} - Id$ and $F_{21} F_{12} - Id$ are smoothing operators. 

\paragraph{Outcomes of the Landweber iteration.} 

In cases where the operator $I_0$ is proved unstable in theory, one may wonder what happens if we run an adjoint-based inversion (such as Landweber's iteration) to the data. Let us recall that given $\L\colon \H_1\to \H_2$ a bounded operator between two Hilbert spaces, and given data $m = \L f$, the Landweber iteration consists in choosing a constant $\gamma>0$ and iterating the following scheme: 
\begin{align*}
    f^{(0)} = 0, \qquad f^{(k)} = f^{(k-1)} - \gamma \L^* (\L f^{(k-1)}-m), \qquad k=1,2,\dots
\end{align*}
One can use $\L = I_0$, though $\L = (-\Delta_g)^{1/2} \chi I_0^* I_0$ is an equivalent choice which removes the smoothing properties of $I_0$ and may speed up convergence ($\L^*\L$ is a $\Psi DO$ of order zero while $I_0^* I_0$ is not).

The analysis made in \cite{Holman2017}, specifically regarding the principal unitarity of the operators $F_{12}$ and $F_{21}$, allows to draw conclusions on what happens to the iterations in the presence of conjugate points. To be specific, considering the situation where $f$ is supported on $U_1\cup U_2$ with $WF(f) \subset V^1\cup V^2$, we may identify $f$ with the couple $(f_1,f_2) := (r_{U_1}f,r_{U_2}f)$. Then the operator $\L^* \L$ can be regarded as a $2\times 2$ matrix of operators involving the operators $F_{12}$ and $F_{21}$, namely, modulo smoother operators, one may write
\begin{align*}
    \L^*\L = \left(
    \begin{array}{cc}
	Id + F_{12}F_{21} & 2 F_{12} \\ 2 F_{21} & Id + F_{21} F_{12}.
    \end{array}
    \right) \mod \Psi^{-1}.
\end{align*}
Picking $f_1\in L^2(U_1)$ arbitrary and $f_2 = 0$, and writing the Landweber iteration at leading order, the iterations converge to 
\begin{align*}
    \text{Landweber solution} = (Id - P) f_1 + F_{21} Pf_1, \qquad P := (Id + F_{21}^* F_{21}),
\end{align*}
see \cite[Sec. 3.2]{Holman2017}. This means that, starting from an input initially supported on $U_1$, the Landweber iteration reconstructs a portion of the input, but also generates the artifact $F_{21} Pf_1$, supported on $U_2$. This is not so much a defect of the method than a common issue that any adjoint-based method will share. Indeed $f_1$ and $F_{21}f_1$, while having disjoint wavefront sets, lead to the same singularities in data. Without additional prior knowledge, any linear combination of the two would be an acceptable reconstruction given the data.

\paragraph{Attenuated transforms.} 

So far, the statements were made for transforms without weight, although the analysis carries over to the case of a transform with weight of the form 
\begin{align*}
    I_\phi f(x,v) := \int_{0}^{\tau(x,v)} f(\gamma_{x,v}(t)) \phi(\varphi_t(x,v))\ dt, \qquad (x,v)\in \partial_+ SM,
\end{align*}
where $\phi:SM\to \Rm$ is a smooth weight, non-vanishing for simplicity of exposition. A specific example is that of the {\em attenuated X-ray transform}, where the weight is given by $\phi(x,v) = \exp \left( -\int_0^{\tau(x,v)} a(\gamma_{x,v}(t))\ dt \right)$, with $a$ some {\em attenuation} function. The presence of a weight can salvage stability in certain scenarios, even in the presence of conjugate points. Namely, the following results were established in~\cite{Monard2013b}: 
\begin{itemize}
    \item If conjugate points occur at most in pairs, and for any conjugate pair $(x_1,v_1)$, $(x_2,v_2)$, we have $\det \left[\begin{smallmatrix} \phi(x_1,v_1) & \phi(x_2,v_2) \\ \phi(x_1,-v_1) & \phi(x_2,-v_2) \end{smallmatrix}\right]\ne 0$, then stability is possible. 
    \item If the condition above does not hold at some conjugate pair, or if some conjugate points occur in triples (i.e. there exists a Jacobi field vanishing more than twice along some geodesic), then the problem is unstable.  
\end{itemize}

The work in \cite{Monard2013b, Holman2017} also covers the case of attenuations, namely in constructing corresponding artifact-generating operators in the unstable case, and studying the long-term behavior of the Landweber iteration when applied to the attenuated X-ray transform. A characteristic example of the locality of the discussion regarding the interplay between the geometry and the presence of a weight can be found Fig. \ref{fig:attenuated_conjugate} (these pictures also appear in~\cite{Holman2017}).

\begin{figure}[htpb]
    \centering
    \includegraphics[trim = 10 50 0 40, clip, width=0.32\textwidth]{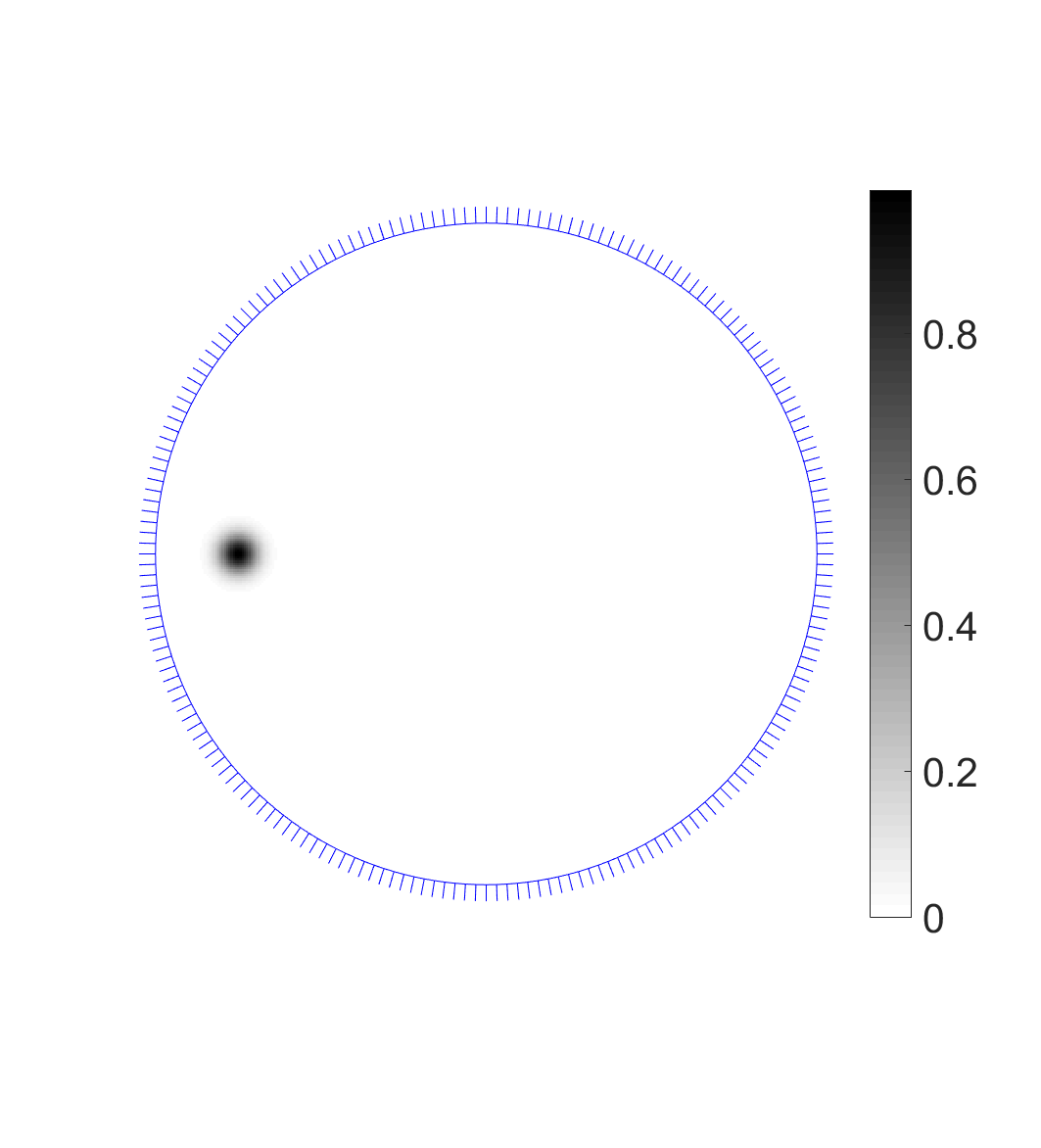}	    
    \includegraphics[trim = 10 50 0 40, clip, width=0.32\textwidth]{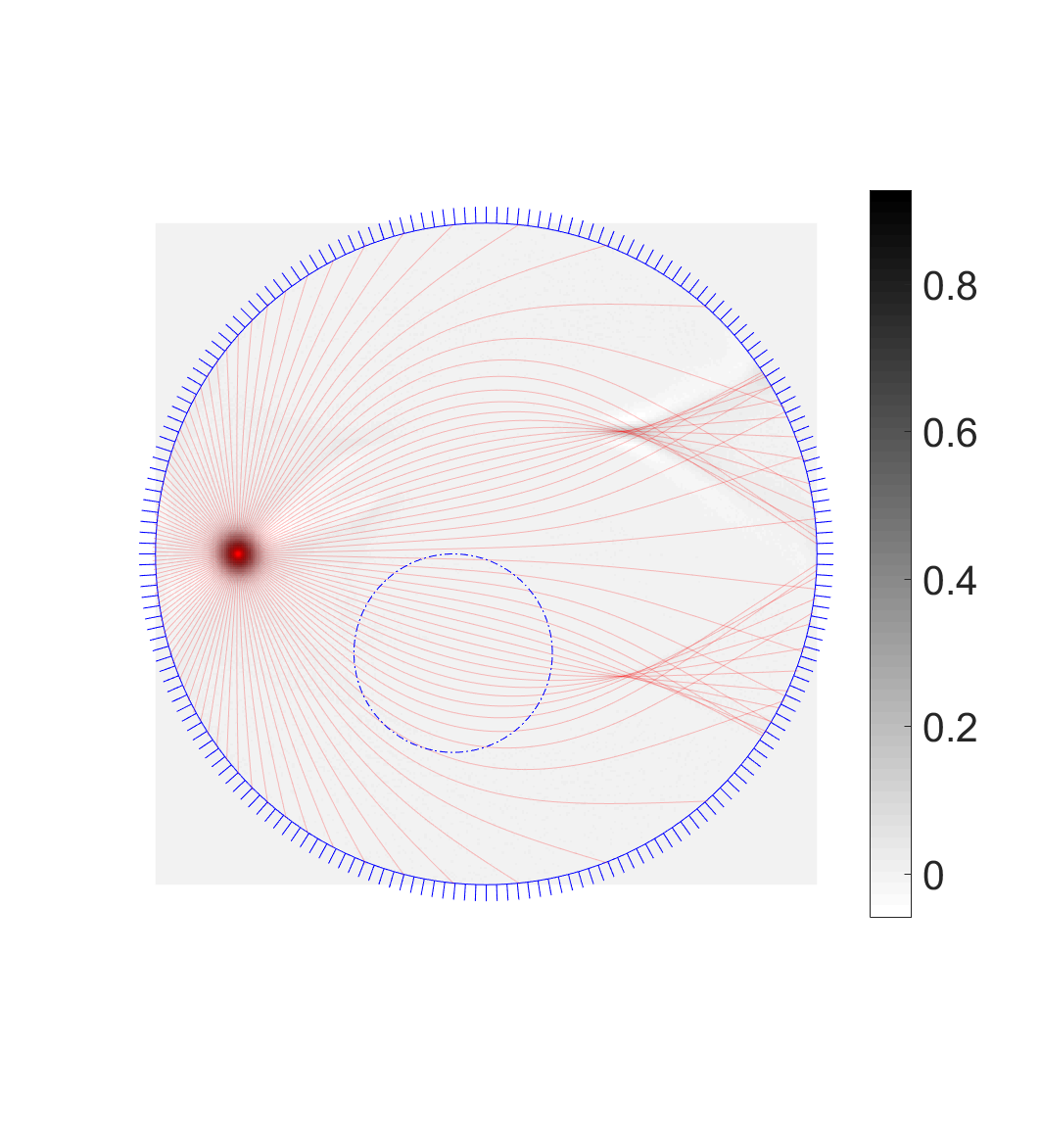}	    
    \includegraphics[trim = 10 50 0 40, clip, width=0.32\textwidth]{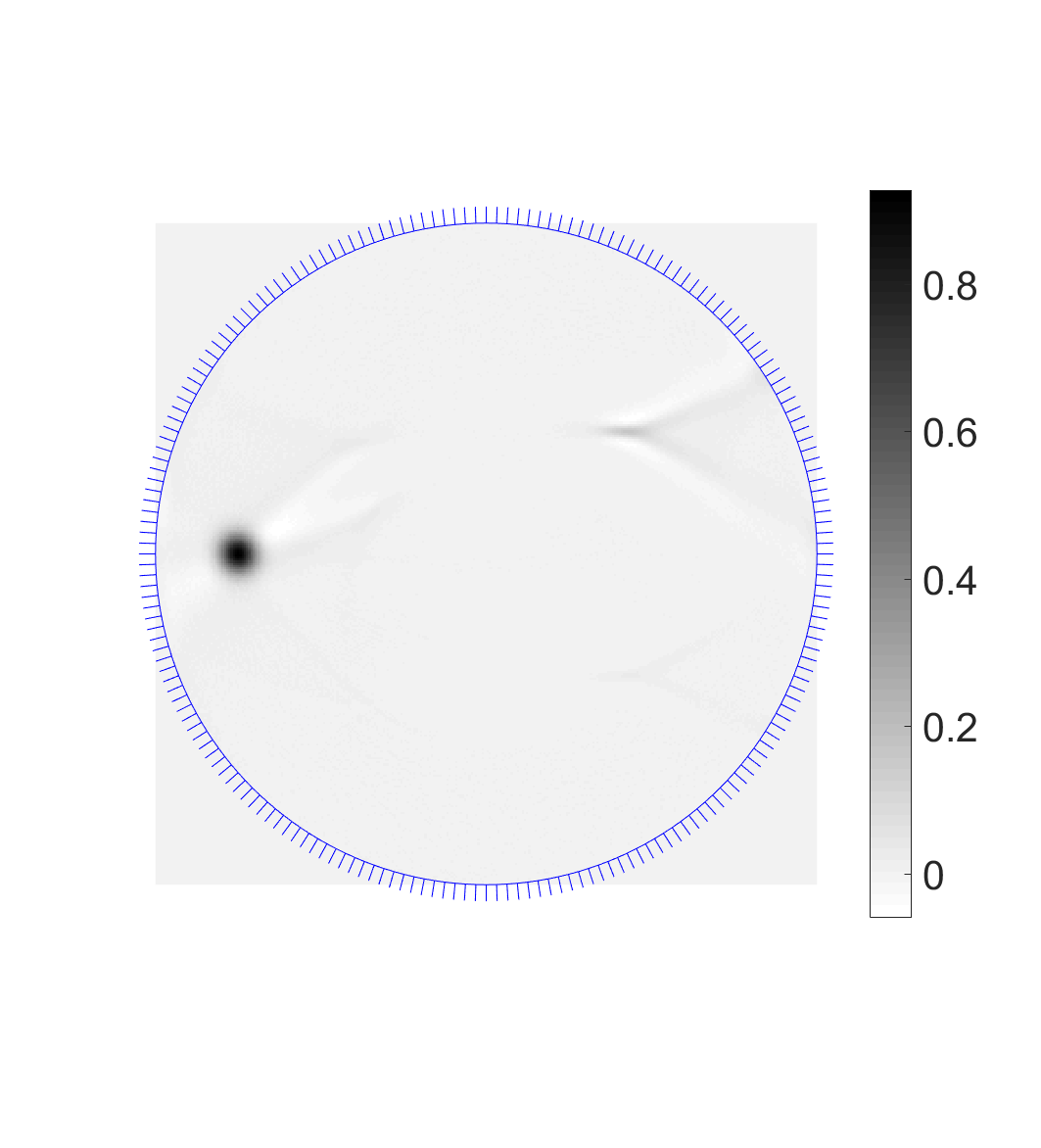}
    \caption{Left: a sharply peaked gaussian $f$ defined on the unit disk. Middle: some geodesics superimposed, showing where the conjugate locus of the (approximate) wavefront set of $f$ is. A positive attenuation $a$ supported inside the dashed circle is included in the transform. Right: reconstruction of $f$ from $I_a f$ after convergence of Landweber's iteration. Comments: In the absence of attenuation, one would expect that the Landweber iteration generates artifacts at both top and bottom conjugate loci. On the other hand, adding a positive attenuation supported inside the dashed circle on the middle picture ``stabilizes'' the reconstruction of certain singularities by removing the bottom artifact at convergence. }
    \label{fig:attenuated_conjugate}
\end{figure}

\np
\subsubsection{Three dimensions and higher}\label{sec:FIOnD}

In higher dimensions, $M$ has dimension $n$ while $\partial_+ SM$ has dimension $2n-2>n$, and the analysis from above is no longer formally determined. The analysis will also depend on the order of conjugate points, which can be of any order between $1$ and $n-1$. Following pioneering work in \cite{Stefanov2012a}, the most complete result known to date is given in \cite{Holman2015}, and applies to general weighted transforms of the form $\I_{\phi,0} := F_* \circ \phi \circ \pi^*$, where $\phi\colon SM\to \Rm$ is admissible in the sense of Section~\ref{sec:microlocalsimple}. We identify $\phi$ with the corresponding multiplication operator $f\mapsto \phi f$. As seen earlier, a natural approach to study $I_{\phi,0}$ is to study the normal operator
\begin{align*}
    N_\phi:= I_{\phi,0}^* I_{\phi,0} = \pi_* \circ \phi \circ F^* \circ F_* \circ\phi\circ \pi^*,
\end{align*}
using, again, the clean composition calculus of four FIOs. As mentioned in Section~\ref{sec:microlocalsimple}, in the absence of conjugate points, $N_\phi$ is an elliptic $\Psi$DO of order $-1$. In the case of conjugate points, two clean compositions can be carried out easily, while the third creates multiple connected components in the canonical relation, corresponding to differing orders of conjugate points. Before mentioning the main theorems, let us briefly introduce notation about the decomposition of the conjugate locus on~$M$. 

Two vectors on the same geodesic, say $(x,v)$ and $\varphi_t(x,v)$ for $t\ne 0$, are {\bf conjugate of order~$k$} ($1\le k\le n-1$) if 
\[  \dim \left(\V (x,v)\cap d\varphi_{-t}|_{\varphi_t(x,v)} \V (\varphi_t(x,v)) \right) = k, \]
and we write $C\subset SM\times SM$ the set of all conjugate pairs\footnote{This corresponds to the more traditional definition that on some geodesic $\gamma$, two points $x = \gamma(t_1)$ and $x' = \gamma(t_2)$ are conjugate of order $k$ along $\gamma$ if the space of Jacobi fields along $\gamma$ vanishing at $t_1$ and $t_2$ has dimension $k$, see \cite[Lemma 3]{Holman2015}.}. $C$ splits into $C = \bigcup_{k=1}^{n-1} C_{R,k} \cup C_S$ where $C_{R,k}$ consists of the {\bf regular conjugate pairs of order $k$} (those that have a neighborhood in $SM\times SM$ such that any pair is either non-conjugate or conjugate of order $k$), and $C_S$ (singular conjugate points) contains the rest. Each $C_{R,k}$ is an embedded submanifold of $SM\times SM$ of dimension $2n-1$ (see \cite[Theorem 3]{Holman2015}), and one may construct a vector bundle $J_{R,k}\subset T(SM)\times T(SM)$ of dimension $k$ on $C_{R,k}$, consisting of the pairs $\left((\gamma(t_1), \dot\gamma(t_1),\frac{dJ}{dt}(t_1)), (\gamma(t_2), \dot\gamma(t_2),\frac{dJ}{dt}(t_2))\right)$ whenever $\gamma(t_1), \gamma(t_2)$ are conjugate along $\gamma$ and $J$ is any Jacobi field along $\gamma$ vanishing at $t_1$ and $t_2$. Each such bundle maps into ${\cal C}_k (J_{R,k})\subset T^* M\times T^* M$, obtained by mapping each argument of $J_{R,k}$ using the connection map and the musical isomorphism, given at any point by $K_{(x,v)} \colon \V(x,v) \to \{v\}^\perp$ (where $\{v\}^\perp\subset T_x M$) and $\flat_g|_x\colon T_xM\to T_x^* M$, respectively. 

The main result is as follows. 
\begin{theorem}[{\cite[Theorem 4]{Holman2015}}]
\label{thm:Holman} Suppose $M$ has no self-intersecting geodesics and that $C_S = \emptyset$. Then the sets 
    \[  C_{A_k} = {\cal C}_k(J_{R,k}) \subset T^* M \times T^* M \]
    are either empty of are local canonical relations made out of $M_k$ connected components. On the level of operators, we have the decomposition
    \begin{align*}
	N_\phi = \Upsilon + \sum_{k=1}^{n-1} A_k, \qquad A_k = \sum_{m=1}^{M_k} A_{k,m},
    \end{align*}
    where $\Upsilon$ is a $\Psi$DO of order $-1$, elliptic where $\phi$ is admissible, and for each $k$, either 
    \[ A_{k,m} \in \I^{-( 1 + (n-1-k)/2)} \left( M\times M, C'_{A_{k,m}}; \Omega_{M\times M}^{1/2} \right),   \]
    where $C_{A_{k,m}}\subset C_{A_k}$ for each $m$, or $M_k = 1$ and $A_{k,1}=0$ if $C_{A_k} = \emptyset$. 
\end{theorem}

From this result, we see that the components $A_k$ for $1\le k<n-1$ are FIOs of order strictly less than $-1$. This hints at us that when $C_S = C_{R,n-1} = \emptyset$, and if one can prove Sobolev mapping properties of the form $A_k:H^s \to H^{s+ 1 + (n-1-k)/2}$ for $1\le k\le n-2$, the problem may still be stable in the sense that, when solving for $f$ the equation $N_\phi f = g$, writing $A = \sum_{k=1}^{n-2} A_k$ and applying a parametrix $Q$ for $\Upsilon$ such that $Q\Upsilon = I + K$ with $K$ a smoothing operator, we arrive at the equation $f + K f + QAf = Qg$, where $QA$ is also smoothing therefore compact. The problem is then Fredholm again, as such, solvable modulo a finite-dimensional kernel of smooth ghosts. 

The main challenge at this point is to prove said mapping properties. In \cite{Holman2015}, the argument is brought to completion under the following additional assumptions: suppose that there are only conjugate pairs of order $1$ and that no two points are conjugate along more than one geodesic, implying $M_1 = 1$; suppose in addition that $C_{A_1}$ is a local canonical graph. Under these additional assumptions, we give the second main theorem of \cite{Holman2015}.

\begin{theorem}[{\cite[Theorem 5]{Holman2015}}] \label{thm:Holman2} Suppose $n\ge 3$, let $(\wtM, \tilde g)$ a smooth extension of $M$ with no self-intersecting geodesic, conjugate pairs of order at most $1$, and such that $C_{A_1}$ is a local canonical graph. Let $\phi:SM\to [0,\infty)$ smooth and admissible at every $x\in M$. Then the kernel of $I_{\phi,0}$ acting on $L^2(\Omega_M^{1/2})$ is at most finite dimensional and contained in $C_c^\infty (\Omega_{M}^{1/2})$, and for any $f\in L^2(\Omega_M^{1/2})/\ker I_{\phi,0}$, 
	\[ \|\I_\phi f\|_{L^2(\Omega_{\partial_+ SM})} \sim \|f\|_{H^{-1/2}(\Omega_M^{1/2})}.  \]
\end{theorem}

The hypothesis that $C_{A_1}$ be a local canonical graph was first formulated in \cite[Eq. (4.4)]{Stefanov2012a}. Some conditions on the geometry guaranteeing the graph condition were formulated in \cite{Stefanov2012a,Holman2015}. To the authors' knowledge, there are no examples yet of metrics satisfying that condition, but other families that do are given in \cite{Stefanov2012a}. Other counterexamples were provided there, hinting that conjugate loci which either contain pairs of higher order, or those that violate the graph condition, require further analysis. In this context, some open questions are formulated in Section~\ref{sec:open}.

\np
\section{Methods exploiting convexity}\label{sec:convexity}

\subsection{Heuristics} 

Let us first give local considerations. Let a smooth hypersurface $S = \{\rho=0\}$ with $p$ a point on it and $U$ a neighborhood of $p$. Then $U\backslash S$ splits into two components $U_{\pm} = \{\pm \rho>0 \}$. Considering a family of smooth curves $\G$ across $U$ which are near tangential to $S$, call $\G_+$ those that pass through $U_+$ and $\G_-$ the other ones. Then $p$ is a convex point on $S$ (as viewed from $U_+$, or concave as viewed from $U_-$) if for all curves in $\G$, we have $\frac{d^2}{dt^2} \rho(\gamma(t)) \le -c_0 <0$ for some constant $c_0>0$. This imposes curves in $\G_+$ to be short, poking through $S$ twice, coming from and returning to $U_-$, see Fig.~\ref{fig:convex}. 
\begin{figure}[htpb]
    \centering
    \includegraphics[height=0.2\textheight]{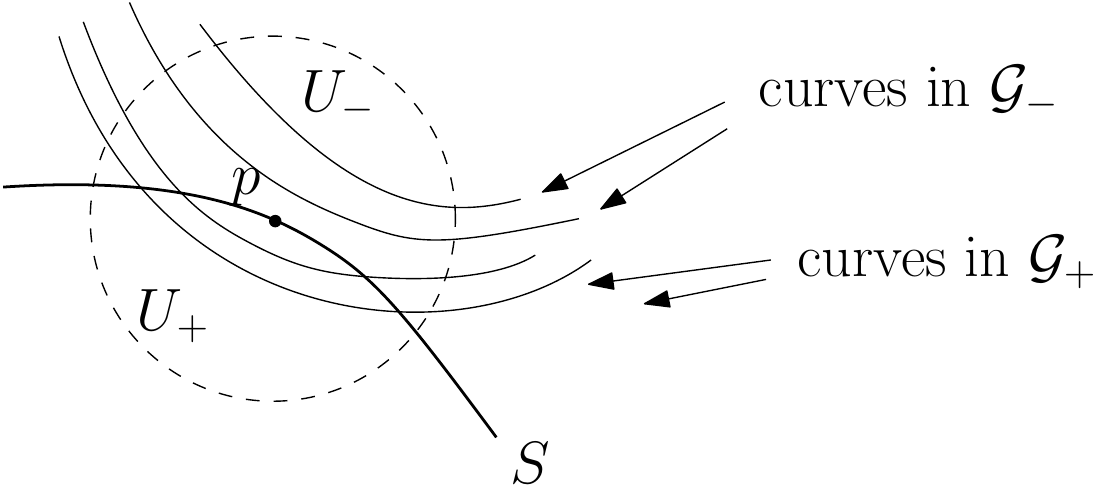}
    \caption{A locally convex setting}.
    \label{fig:convex}
\end{figure}
In this type of geometry, the problem of reconstructing $f$ supported in $U$ from $If|_{\G}$ admits the following triangular structure: 
\begin{align}
    \left[
    \begin{array}{c}
	If|_{\G_-} \\ If|_{\G_+}
    \end{array}
\right] = \left[
\begin{array}{ccc}
    A_1 & 0  \\ * & A_2 
\end{array}
\right] \left[
    \begin{array}{c}
	f|_{U_-} \\ f|_{U_+}
    \end{array}
\right].
\end{align} 
This implies that one may consider reconstructing $f|_{U_-}$ first, from $If|_{\G_-}$, then consider the reconstruction of $f|_{U_+}$ later. This is the key idea behind {\bf support theorems} (does $If|_{\G_-} = 0$ imply that $f$ is supported in $U_+$?), and several old and new inversion approaches. Such a condition makes the problem local in nature, and becomes global if one may exhaust the entire manifold in that way\footnote{The exhaustion can exclude a small set (zero measure, empty interior, or other) corresponding to the function space at hand.}, as one may reconstruct the function $f$ layer after layer. 

Convexity was initially assumed in \cite{Herglotz1905,Wiechert1907} in the case of radial, scalar metrics of the form $g = c^{-2}(r) Id$ where the sound speed $c$ satisfies $\frac{d}{dr} \left(r/c(r)\right) >0$, arising as a natural condition for invertibility. Such a condition is natural in that if $\frac{d}{dr} \left(r/c(r)\right)$ vanishes at $r_0$, then the sphere $\{|x| = r_0\}$ is totally geodesic, i.e. neither convex nor concave. Convexity conditions were exploited further in \cite[Sec 1.8]{Romanov1974} when the metric only depends on a single (not necessarily radial) variable, and in \cite{Sharafutdinov1997} for radial metrics to produce a positive answer to the tensor tomography problem. It was first used in a general context without symmetries in~\cite{Uhlmann2015}. 

In order to understand {\bf layer stripping} when a global foliation holds, split $\Omega$ into $\Omega_1\cup \Omega_2 \cup \Omega_3$ as in Fig. \ref{fig:layer} below, and for $1\le j\le 3$, denote $\G_j \subset \G$ the geodesics intersecting $\Omega_j$ but not $\Omega_{j+1}$, so that $\G = \G_1\cup \G_2\cup \G_3$. Then the convex foliation allows to establish that the forward operator has the triangular structure
\begin{align}
    \left[
    \begin{array}{c}
	If|_{\G_1} \\ If|_{\G_2} \\ If|_{\G_3}
    \end{array}
\right] = \left[
\begin{array}{ccc}
    A_1 & 0 & 0 \\ * & A_2 & 0 \\ * & * & A_3
\end{array}
\right] \left[
    \begin{array}{c}
	f|_{\Omega_1} \\ f|_{\Omega_2} \\ f|_{\Omega_3}
    \end{array}
\right],
\label{eq:layerStripping}
\end{align} 
where for each $1\le j\le 3$, the geodesics in $\G_j$ are geodesically complete over $\Omega_j$ (so the operators $A_j$ are Fredholm, i.e. injective up to a finite-dimensional space of smooth ghosts, and stable). In addition, adding more intermediate slabs preserves the triangular structure, and if the slabs are thin enough, then the error operators within the results of \cite{Uhlmann2015} becomes contractions, so that the operators $A_j$ are actually injective. 

\begin{figure}[htpb]4
    \centering
    \includegraphics[height=0.2\textheight]{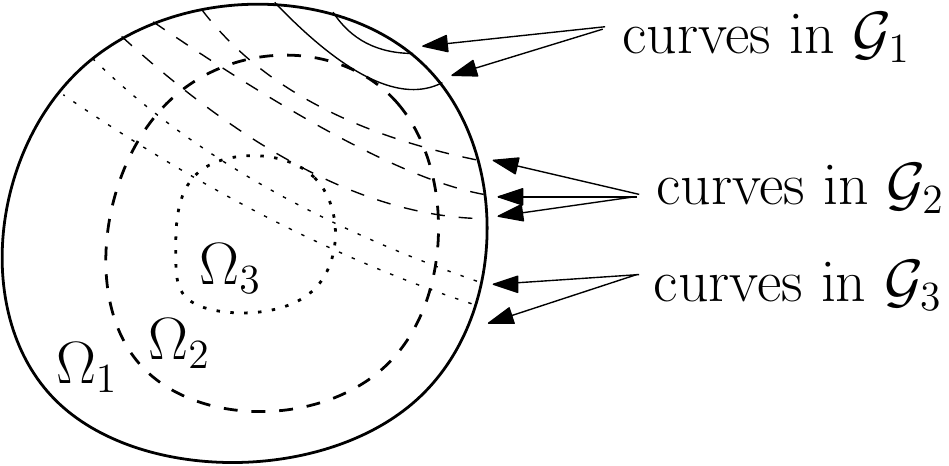}
    \caption{Layer stripping in action: reconstruct $f_1 = f|_{\Omega_1}$ from $If|_{\G_1}$ first, then $f_2 = f|_{\Omega_2}$ from $(If-If_1)|_{\G_2}$, then $f_3= f|_{\Omega_3}$ from $(If-If_1-If_2)|_{\G_3}$}.
    \label{fig:layer}
\end{figure}

In the following sections, we describe in more detail the literature based on exploiting convexity, first reviewing support theorems in Section~\ref{sec:support} then injectivity and reconstruction approaches in geometries with and without symmetries in the remaining sections.

\subsection{Support theorems}\label{sec:support}

A version of Helgason's support theorem states that if a compactly supported smooth function $f\colon\Rm^n\to\Rm$ integrates to zero over all lines that avoid a compact, convex obstacle $K\subset\Rm^n$, then $f=0$ in $\Rm^n\setminus K$.

There are a number of possible extensions of this result to Riemannian manifolds.
The mentioned Euclidean result is global: it concerns all rays that avoid a compact obstacle.
Typical support theorems on manifolds are local: a boundary point $x\in\partial M$ has a neighborhood $U$ so that if a sufficiently regular function $f\colon U\to\Rm$ integrates to zero over all maximal geodesics that stay in $U$, then $f$ vanishes in $U$.
If $f$ is a tensor field, then one can only expect it to vanish up to gauge, meaning that $f=\sigma\nabla h$ for a tensor field $f$ of one order lower.

On real analytic simple Riemannian manifold a local support theorem was obtained by Krishnan for both scalar~\cite{Krishnan2009a} fields and by Krishnan and Stefanov for second order tensor~\cite{Krishnan2009} fields using analytic microlocal analysis, an approach initiated by Boman and Quinto \cite{Boman1987}. The technique uses the complex stationary phase by Sj\"ostrand as applies for the first time in \cite{Frigyik2008}.
The result was later extended to tensor fields of all orders~\cite{Abhishek2017}.
A result without real analyticity was obtained by Uhlmann and Vasy in dimensions three and higher~\cite{Uhlmann2015}.

On radially symmetric manifolds it is natural to consider a spherical layer instead of a small neighborhood. 
The support theorem holds in any dimension when the boundary is strictly convex, as this guarantees that the Herglotz condition is satisfied locally~\cite{Sharafutdinov1997,Hoop2017}.

All of these results on Riemannian manifolds reproduce Helgason's theorem when applied to Euclidean geometry. Support theorems have also been obtained in Lorentzian geometry~\cite{Stefanov2017a,RabienaHaratbar2018}.

As explained in the previous section, local results of this kind lead to global injectivity results if one assumes a foliation condition compatible with the local support theorem.
A typical combination is a strictly convex foliation and a local support theorem near strictly convex boundary points.

\subsection{Metrics dependent on a single variable}

\subsubsection{Parallel layers}

V.G. Romanov studied integral geometric problems in \cite[Sec. I.4]{Romanov1974} on the slab $S = \Rm^{n-1} \times \{0\le x_n\le H\}$, integrating a function along a family of curves which is translation-invariant in the $\Rm^{n-1}$ factor, and with a diving behavior in the $x_n$ variable. Specifically, a crucial assumption is that through every point $x = (x', x_n)\in S$ and direction $v$ with $v_n =0$, there is a curve $\gamma$ passing through $(x,v)$, with both endpoints at $\{x_n=0\}$, and such that $x$ is the farthest point to $\{x_n=0\}$ on the curve $\gamma$. Together with additional curvature and smoothness conditions, this nothing but formulates that the planes $\{x_n = const.\}$ form a convex foliation of~$S$.

A way to achieve this with a geodesic flow is to consider a metric of the form $g = e^{2\lambda(x_n)} Id$, with $\lambda' <0$, a form of Herglotz condition adapted to parallel layers. 

In this case, a continuous function with compact support can be reconstructed from its integrals, see \cite[Theorem 1.5, Sec. I.4]{Romanov1974}, allowing for certain weights as well in \cite[Theorem 1.6, Sec. I.4]{Romanov1974}. To prove this, one notices that upon Fourier-transforming in $x'$, the problem diagonalizes frequency-wise, and for each frequency, the integral geometry problem looks like a Volterra equation of the second kind (integral equation with causal kernel). The case where curves are parabolas is also treated in \cite[Ch. 4.3]{Bukhgeuim2000} using Volterra operator equations.

\subsubsection{Spherical layers (radial metrics)} 

When the manifold is a ball equipped with a radial, scalar metric of the form $c^{-2}(r)Id$, the sphere of radius $r$ and center $0$ is strictly convex if and only if the Herglotz condition $\frac{d}{dr} \frac{r}{c(r)} >0$ is satisfied.
A rotation invariant Riemannian metric on an annulus may be written in this form~\cite[Proposition C.1]{Hoop2017b}.
In this context, Sharafutdinov has obtained a positive answer to the tensor tomography problem~\cite{Sharafutdinov1997}.

\begin{theorem}[{\cite[Theorem 1.1]{Sharafutdinov1997}}]
\label{thm:Sharafutdinov}
Let $g$ a Riemannian metric on the spherical layer 
    \[ D = \{x\in \Rm^n |\ \rho_0 \le |x|\le \rho_1\} \qquad (0<\rho_0<\rho_1,\ n\ge 2).  \]
    Assume $g$ invariant under all orthogonal transformations of $\Rm^n$ and such that the sphere $S_\rho = \{x|\ |x|=\rho\}$ is strictly convex for every $\rho\in [\rho_0,\rho_1]$. Let $G = S_{\rho_1}$. If a symmetric tensor field $f\in C^\ell(S^m(T^*D))$ ($\ell\ge 1, m\ge 0$) lies in the kernel of the ray transform $I_G$, then $f = dv$ for some $v\in C^\ell (S^{m-1} (T^* D))$ satisfying $v|_G=0$.
\end{theorem}

The same result for scalar fields was proven in lower regularity in~\cite{Hoop2017}, assuming only $f\in L^2$ and $g\in C^{1,1}$.
The metric may even have jump discontinuities.

\begin{theorem}[{\cite{Hoop2017}}]
Consider the spherical layer $D$ as in theorem~\ref{thm:Sharafutdinov}.
Suppose $c\colon[\rho_0,\rho_1]\to(0,\infty)$ is piecewise $C^{1,1}$ and satisfies the Herglotz condition.
At jump discontinuities the Herglotz condition is interpreted as positivity of the distributional derivative $\frac{d}{dr} \frac{r}{c(r)} >0$, meaning $c(r-)>c(r+)$ when $c$ is not continuous at $r$.
If $f\in L^2(D)$ intergates to zero over all geodesics between points on the outer boundary $S_{\rho_1}$, then $f=0$.
\end{theorem}

The two proofs are similar.
The problem is first reduced to the two-dimensional case.
In two dimensions the ray transform is block-diagonalized by the angular Fourier transform.
More precisely, the function can be written in polar coordinates as $f(r,\theta)$ and expanded as a Fourier series in $\theta\in S^1$.
Identifying a point in the annulus with the geodesic whose minimal radius is at that point, one can regard $If(r,\theta)$ as a function on the annulus as well.
The $k$th Fourier component of $If(r,\theta)$ only depends on the $k$th Fourier component of $f(r,\theta)$.
This one-dimensional dependence is encoded in Abel-type integral transforms depending on the index $k$.
Once one proves that the Abel transform is injective for all $k$, injectivity of the X-ray transform follows.
In the case of tensor fields the radial problem for each $k$ becomes a system of integral equations due to the gauge freedom.

The Herglotz condition does not prohibit conjugate points. At least in two dimensions, this allows cases where the problem is definitely unstable yet the ray transform is solenoidally injective in all orders.

\subsection{Geometric method for piecewise constant functions}

Using a layer stripping argument associated with a foliation is somewhat technical when there is no explicit symmetry.
One can relax the geometrical assumptions on the underlying manifold if one restricts to a smaller class of functions.
The argument can be used to show that a piecewise constant function in the kernel of the X-ray transform has to vanish on any manifold of dimension two or higher, provided that there is a strictly convex foliation~\cite{Ilmavirta2017}.
One has to be careful with the wording, since the set of piecewise constant functions is not a vector space unless the partition is fixed.
In two dimensions this is not covered by any existing result, and in higher dimensions this provides a simpler proof than the scattering calculus approach described next.

\subsection{The Uhlmann--Vasy method by scattering calculus}\label{sec:Melrose}

All the methods above rely to some extent on a ``local'' problem near a convex boundary point, where convexity allows to write integral equations with causal kernels as in, e.g., \cite{Romanov1974,Sharafutdinov1997}, or to set up an invertible linear problem in the case of \cite{Ilmavirta2017}. Another method, first introduced by Uhlmann and Vasy in \cite{Uhlmann2015} to study the local problem in dimensions three and higher, is to express a post-processed version of the X-ray transform into a ``normal-like'' operator with good ellipticity properties within Melrose's scattering calculus. This yields local injectivity of X-ray transforms over functions, which we first discuss including a sketch of the approach. We then outline the many generalizations (to higher-order tensor fields, other flows and weigthed transforms), for which the approach's handling of each case proves very robust.

\subsubsection{Geodesic X-ray transform over functions}

The main result hinges on a local integral geometric problem near a concave hypersurface $\{x=0\}$, and suppose that on the slab $S= \{0\le x\le c\}$, every level set $\{x=x_0\}$ is geodesically concave when viewed from the super-level set $\{x\ge x_0\}$. Let us denote a general point $z = (x,y)$, and coordinatize a unit tangent vector $v\in T_zS$ as $v = \lambda \partial_x + \omega$ with $g_z(\partial_x, \omega) = 0$ and $g_z (\omega,\omega) = 1-\lambda^2$, see Fig.~\ref{fig:Melrose}.

\begin{figure}[htpb]
    \centering
    \includegraphics[height=0.17\textheight]{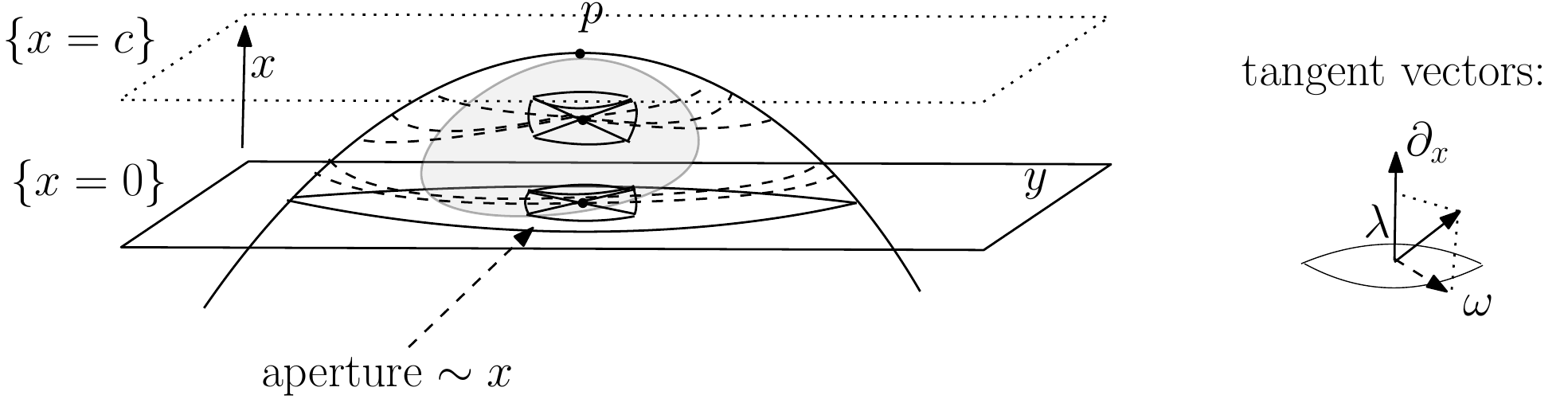}
    \caption{A local setting amenable to scattering calculus.}
    \label{fig:Melrose}
\end{figure}


Back to our case of a manifold with boundary, if $p$ is a convex point of $\partial M$, then there exists a neighborhood $O = M\cap \{0< x \le c\}$ of $p$ open in $M$ which can be described by a convex foliation as above. We denote by $\M_O$ the set of $O$-local geodesics. In this context, the main local result in \cite{Uhlmann2015} is:
\begin{theorem}\label{thm:UhlmannVasy}
    With the notation as above, if $c$ is small enough, the local transform $I$ is injective on $H^s(O)$ for any $s\ge 0$. Upon defining for $\digamma>0$
    \[ H_\digamma^s(O) = e^{\digamma/x} H^s = \{f\in H^s_{\text{loc}}(O): e^{\digamma/x}f\in H^s(O)\},   \]
    For any $s\ge 0$ there exists $C>0$ such that for all $f\in H_\digamma^s (O)$, 
    \begin{align*}
	\|f\|_{H_\digamma^{s-1} (O)} \le C \|If|_{\M_{O}}\|_{H^s(\M_{O})}.
    \end{align*}    
\end{theorem}

The main idea behind the theorem is to define a ``truncated'' normal operator
\begin{align*}
    Af (z) := x^{-1} \int_{S_z} If(\gamma_\nu) \chi (\lambda/x) d\mu(\nu),
\end{align*}
and to show that for $\digamma >0$, 
\begin{align*}
    A_\digamma = x^{-1} e^{-\digamma/x} A e^{\digamma/x} \in \Psi_{\text{sc}}^{-1,0} (\{x\ge 0\}),
\end{align*}
and is elliptic within Melrose's scattering algebra $\Psi_{\text{sc}}^{\cdot,\cdot} (\{x\ge 0\})$, a $\Zm\times \Zm$-graded algebra of pseudo-differential operators defined on a manifold with boundary, whose grading has two indices: a classical one measuring the growth rate in the momentum variable at ``fiber infinity'', and another measuring the rate of vanishing near the boundary\footnote{The boundary of the base is also called ``base infinity'', as the geometric model associated is usually such that geodesics do not attain the boundary in finite time.}. 

To prove ellipticity, one must compute the principal symbol of $A_\digamma$, and this accounts for much of the work. The Schwartz kernel $K_{A_\digamma}(z,z')$ is easy to compute via exponential coordinates; the main singularity of $K_{A_\digamma}$ is on the diagonal $z=z'$, and upon rewriting $K_{A_\digamma}$ as a function of $(z,\xi:= z'-z)$, the full symbol may be computed via Fourier transform $\sigma_{A_\digamma} (z,\zeta)\sim {\cal F}_{\xi\to \zeta} K_{A_\digamma} (z,\zeta)$. Once such an amplitude is computed, one must find the leading behavior in terms of $x$ near $x=0$ and in terms of $\zeta$ near $\zeta\to \infty$, and show that both leading-order terms vanish neither on $\{x=0\}\times \overline{\Rm^n_{\zeta}}$, nor on $\overline{O}\times \{|\zeta| = \infty\}$ ($\{|\zeta|=\infty\}$ is understood as the boundary of the radial compactification of $\Rm^n_\zeta$). 

Once ellipticity is proved, this shows Fredholm properties of such an operator, namely the existence of a $B$ such that $BA_{\digamma} = Id + K$, where $K$ is compact. In addition, if $c$ is small enough, then $K$ becomes a contraction, and the inversion of $Id + K$ can be done by Neumann series so that $A_\digamma$ is in fact injective and H\"older-stable. 
If the manifold can be globally foliated by convex hypersurfaces, with convexity constants bounded away from $0$, then compactness arguments allow to upgrade the local Theorem \ref{thm:UhlmannVasy} into a global one, see e.g. \cite[Corollary]{Uhlmann2015}.

\smallskip
Before discussing the latest results, let us mention some generalizations of this approach, the first two of which are covered in the topical review \cite{Uhlmann2017}. These generalizations attest to the robustness of the method for integral geometric problems in dimensions three and higher.

\subsubsection{General families of curves} 

In the appendix to \cite{Uhlmann2015}, Zhou generalized the results above to ray transforms of functions, where integration is performed along general families of curves on a Riemannian manifold $(M,g)$. Such curves are generated by fixing $(x,v)\in TM$ and solving the differential equation
\begin{align*}
    \nabla_{\dot\gamma} \dot\gamma = G(\gamma,\dot\gamma), \qquad \gamma(0) = x, \qquad \dot\gamma(0) = v, \qquad (\gamma:= \gamma_{x,v})
\end{align*}
where $\nabla$ is the Levi-Civita connection and $G$ is a fixed smooth bundle map from $TM$ to itself. The local result is first considered near a convex boundary point $p$, where by convexity here we mean that if $\rho$ is a boundary defining function near $p$ (such that $\rho =0$ on $\partial M$ and $\rho>0$ inside $M$), we have $\frac{d^2}{dt^2} \rho(\gamma_{p,v}(t))|_{t=0}<0$ for any $v\in T_p (\partial M)$. 

Related results are described in Section \ref{sec:microlocalsimple}.

\subsubsection{Injectivity over tensor fields and judicious choices of gauges} \label{sec:gauges}

The approach was then generalized by Stefanov--Uhlmann--Vasy in \cite{Stefanov2014} to prove solenoidal injectivity of geodesic X-ray transforms over tensor fields of order $1$ and $2$ (as a stepping stone toward proving boundary rigidity, see Section~\ref{sec:nonlinearMetric} below). In this case, the transform has a natural kernel, and given that one is working with two distinct metrics (the initial one and the scattering one), one must make a judicious choice of gauge. Upon choosing the solenoidal gauge with respect to the scattering metric, the work is then to show full ellipticity (up to gauge) of a certain operator defined out of appropriately restricted X-ray transforms. Ellipticity is here defined in the sense of a scattering algebra of operators defined on sections of bundles, namely, $\Psi^{\cdot,\cdot}_{\text{sc}} (X; {}^{\text{sc}}T^* M, {}^{\text{sc}}T^* M)$ and $\Psi^{\cdot,\cdot}_{\text{sc}} (X; \text{Sym}^2 ({}^{\text{sc}}T^* M), \text{Sym}^2 ({}^{\text{sc}}T^* M))$. Though not written, the results are expected to generalize to tensor fields of arbitrary order. 

The solenoidal gauge used above gives rise to an elliptic problem, which is an advantage to study Fredholmness in inverse problems, but a curse when considering local problems and extension issues, since elliptic gauges require solving a global problem in order to be computed. To gain more flexibility in the local problem and in fact tackle the non-linear ones (boundary and lens rigidity, see Section~\ref{sec:BRnormal}), the same authors improved the results above in \cite{Stefanov2017}, by working with tensor fields which are in the {\bf normal gauge} at the boundary (one-forms in the normal gauge only have tangential components, 2-tensors only have tangential-tangential components). One must then construct ``normal-like'' operators which land in this gauge instead of the solenoidal one. Such operators are only partially elliptic, and their Fredholm properties are obtained after a refined study of the characteristic directions. In the case of one-forms, the operator has real principal type with radial points, and those can be directly dealt with. The case of two-tensors is more delicate, as one must deal with double characteristics (i.e., second-order vanishing of the principal symbol at the characteristic set). The authors circumvent this by reducing the inversion to that in the elliptic gauge and working out the effect of the change of gauge. Such results are compatible with having microlocal weights in the ray transform, a necessary step toward using pseudo-linearization identities for the non-linear problem.

\subsubsection{Magnetic flows} 

Injectivity over tensor fields was also established in \cite{Zhou2016} in the case of magnetic flows. The added technical step is that the kernel of the magnetic ray transform couples pairs of tensor fields of consecutive order, and therefore invertibility is proved by jointly considering pairs (function, one-form) and pairs (one-form, symmetric two-tensor). A local injectivity result requires a ``magnetic convexity'' condition near a boundary point; a global injectivity result (over a gauge representative similar to a solenoidal gauge defined on pairs) can then be derived when the manifold can be foliated by magnetically convex hypersurfaces. The ellipticity of the system is established by working in the scattering algebras $\Psi^{\cdot,\cdot}_{\text{sc}} (M, {}^{\text{sc}}T^* M \times M, {}^{\text{sc}}T^* M \times M)$ and $\Psi^{\cdot,\cdot}_{\text{sc}} (M, \text{Sym}^2 ({}^{\text{sc}}T^* M)\times {}^{\text{sc}}T^* M, \text{Sym}^2 ({}^{\text{sc}}T^* M)\times {}^{\text{sc}}T^* M)$. Though not written, the results are expected to generalize to tensor fields of arbitrary order. 

\subsubsection{Transform with matrix weights, connections and Higgs fields}

Another generalization of the approach is to work with transforms with connections and Higgs fields, as done in \cite{Paternain2016a}, or in the most general form, on transforms with a matrix weight. Namely, given a (known) weight function $W \in C^\infty(SM, GL(N,\Cm))$, one may define the X-ray transform $\I_W \colon C^\infty(SM, \Cm^N) \to C^\infty(SM, \Cm^N)$
\begin{align*}
    \I_W h (x,v) = \int_0^{\tau(x,v)} W(\varphi_t(x,v)) h(\varphi_t(x,v))\ dt, \qquad (x,v)\in \partial_+ SM.
\end{align*}
Special examples of weights are those that arise from a pair $(A,\Phi)$ (connection, Higgs field) on the bundle $M\times \C^N$, where the weight is assumed to solve the transport equation
\begin{align*}
    XW = W (A+ \Phi) \qquad (\text{on } SM\times \Cm^N), \qquad W|_{\partial_- SM} = Id,
\end{align*}
in which case we exactly have $\I_W = \I_{A,\Phi}$. In \cite{Paternain2016a} the authors study the invertibility of $\I_W$ in the following cases: (i) $W$ is arbitrary and $h$ is a $\Cm^N$-valued function on $M$; (ii) $W$ arises from a pair $(A,\Phi)$, and $h$ is the sum of $\Cm^N$-valued functions and one-forms. 

The main resting assumption is that the manifold $(M,g)$ where the transform $\I_W$ is defined admits a strictly convex function. This implies the existence of a strictly convex foliation, allowing a successful study of the local problem by scattering calculus (see \cite[Theorem 1.5]{Paternain2016a}) followed by a global argument to prove injectivity of $\I_W$ in both settings described above (see \cite[Theorems 1.1, 1.6]{Paternain2016a}). In addition, an extended discussion is provided in \cite[Sec. 2]{Paternain2016a} regarding which manifolds admit a strictly convex function.
The injectivity result for $\I_{A,\Phi}$ over functions and one-forms also implies a positive answer to the non-linear Problem \ref{pb3}, as described in Section \ref{sec:nonlinearConnection} below.

\np
\section{From linear to non-linear results}\label{sec:nonlinear}

While linearization of non-linear inverse problems is a common approach to obtaining non-linear results in a neighborhood of ``favorable'' case, integral geometric problems have enjoyed striking identities, out of which one may derive a {\bf global} non-linear uniqueness result which ultimately relies on the injectivity of an X-ray transform. 

We give two examples of such identities below, one with applications to boundary/lens rigidity, the other with applications to inverse problems for connections and Higgs fields.

\subsection{From Problem 2 to Problem 1}\label{sec:nonlinearMetric}

In the case of inverse problems for metrics, a good example of a ``pseudo-linearization'' identity is given below, first appearing in \cite{Stefanov1998}. Let $N$ a manifold and $V,\widetilde{V}$ two vector fields on $N$. We denote $X(s,X^{(0)})$ the solution of $\dot X = V(X)$, similarly for $\widetilde{X}$ in terms of~$\widetilde{V}$. 

\begin{lemma}[{\cite{Stefanov1998,Stefanov2013}}]\label{lem:pseudolin} 
    For any $t>0$ and any initial condition $X^{(0)}$, if $\widetilde{X}(\cdot,X^{(0)})$ and $X(\cdot,X^{(0)})$ exist on the interval $[0,t]$, then 
    \begin{align}
	\widetilde{X}(t,X^{(0)}) - X(t,X^{(0)}) = \int_0^t \frac{\partial \widetilde{X}}{\partial X^{(0)}} (t-s,X(s,X^{(0)})) (\widetilde{V}-V)(X(s,X^{(0)}))\ ds.
	\label{eq:pseudolin}
    \end{align}    
\end{lemma}

See \cite{Stefanov2013} for a proof. To see how this relates to the Lens Rigidity Problem (see Problem \ref{pb1}), fix two metrics $g$ and $\tilde g$ on $M$ and suppose they have same lens data\footnote{i.e., same boundary distance function and same scattering relation.}. Both metrics give rise to geodesic vector fields $V$ and $\widetilde{V}$ on $T^* M$, each of which uniquely characterizes $g$ and $\tilde g$. If $g$ and $\tilde g$ have the same scattering relation at the boundary, then the left-hand side of \eqref{eq:pseudolin} vanishes for $X^{(0)}$ at the boundary and for $t>0$ is the length of the geodesic emanating from $X^{(0)}$ for either metric. Then this implies the vanishing of a weighted X-ray transform of $V-\widetilde{V}$ along the geodesics of $V$. If the weight has good properties and if the ray transform is injective, this solves the non-linear problem. 

Such an approach has been documented in the review \cite[Section 5.4.1]{Uhlmann2017} and the interested reader is invited to refer to it for earlier results. We now cover a couple of further recent uses of this approach. 

\subsubsection{Boundary and lens ridigity}\label{sec:BRnormal} 

The works \cite{Stefanov2013} and \cite{Stefanov2017} give the latest progress on boundary rigidity results to date. Specifically, \cite[Theorem 1.1]{Stefanov2017} states that on a Riemannian manifold $(M,g)$ of dimension three and higher, if $\partial M$ is stricly convex for two metrics whose boundary distance functions agree, then these metrics are gauge equivalent in a neighborhood of the boundary. Under a global foliation condition, a lens rigidity result is also established in \cite[Theorem 1.3]{Stefanov2017}, namely: if $(M,g)$ has a strictly convex foliation, and if $\hat g$ is another metric whose lens data agrees with that of $g$, then $g$ and $\hat g$ are gauge-equivalent. The results in \cite{Stefanov2013} were first established by the same authors in the case of the recovery of conformal a factor, a problem which does not require addressing ray transforms over tensor fields. 

To find the range of applicability of the results above, one must then study which manifolds admit strictly convex foliations, and this is discussed at length in \cite[Section 2]{Paternain2016a}. Note that it is still open whether simple manifolds of dimension three and higher admit strictly convex foliations, therefore the earlier result \cite{Stefanov2005}, establishing boundary rigidity for generic simple manifolds using microlocal methods, may cover cases which are not treated by \cite{Stefanov2013,Stefanov2017}.
A simple surface does admit a strictly convex foliation.

\subsubsection{Lens rigidity for Yang-Mills fields}

In \cite[Theorem 1.2]{Paternain2017a}, the authors prove the unique identifiability modulo gauge of a Yang-Mills potential from its scattering relation on manifolds of dimension three and higher satisfying a certain foliation condition. In this example, the scattering relation is defined in terms of the flow of a coupled dynamical system for a particle in $SM$ and a Lie algebra-valued ``color charge'', and the nature of the coupling is driven by the Yang-Mills potential. The approach combines the tools of Section \ref{sec:Melrose}, with a pseudolinearization identity similar to Lemma \ref{lem:pseudolin}. A similar problem for magnetic fields was considered earlier by Zhou in~\cite{Zhou2016a}.

\subsubsection{An inverse problem from condensed matter physics}

In \cite{Lai2017}, the authors consider the recovery of a potential from the dynamical behavior of vortex dipoles in an inhomogeneous Gross-Pitaevskii equation in the plane, a problem with applications to condensed matter physics. The inverse problem can be viewed as a lens rigidity problem where the measurements resemble a scattering relation for a flow perturbed by the unknown potential, and the recovery of the potential in \cite[Theorem 2]{Lai2017}) uses ideas such as Lemma~\ref{lem:pseudolin}. 

\subsection{From Problem 4 to Problem 3}\label{sec:nonlinearConnection}

The solenoidal injectivity of ray transforms $I_{A,\Phi}$ over sums of functions and one-forms implies reconstructibility of a connection and a Higgs field from their scattering data up to gauge, that is, a positive answer to Problem \ref{pb3}. 

\begin{theorem}[{\cite[Theorem 1.5]{Paternain2012}}]\label{thm:Conn3} Assume $M$ is a compact simple surface, let $A$ and $B$ be two Hermitian connections, and let $\Phi$ and $\Psi$ be two skew-Hermitian Higgs fields. Then $C_{A,\Phi} = C_{B,\Psi}$ implies that there exists a smooth $U\colon M\to U(n)$ such that $U|_{\partial M} = Id$ and $B = U^{-1} dU + U^{-1} A U$, $\Psi = U^{-1} \Phi U$. 
\end{theorem}

While similar theorems exist in a perturbative context (see references in \cite{Paternain2012}), Theorem \ref{thm:Conn3} is a striking example of how injectivity of the linearized problem implies {\em global} injectivity (modulo gauge) of the non-linear operator $(A,\Phi)\mapsto C_{A,\Phi}$. The argument is as short as it is powerful and we repeat it here. 

\begin{proof}[Proof of Theorem \ref{thm:Conn3}] The equality $C_{A,\Phi} = C_{B,\Psi}$ implies that the fundamental matrix solutions $U_{A,\Phi},U_{B,\Psi}\colon SM\to U(n)$, satisfying 
    \begin{align*}
	(X+ A + \Phi) U_{A,\Phi} = 0, \qquad (X+B+\Psi)U_{B,\Psi} = 0, \qquad U_{A,\Phi}|_{\partial_+ SM} = U_{B,\Psi}|_{\partial_+ SM} = Id,
    \end{align*}
    agree on $\partial SM$. Then the proof consists in showing that $U := U_{A,\Phi} (U_{B,\Psi})^{-1}$, which is smooth by construction, only depends on the basepoint, and thus fulfills the conclusion of the theorem. 
    Looking at $W:= U - Id$, $W$ is a matrix solution, vanishing at $\partial SM$, of the transport equation on~$SM$
    \begin{align*}
	XW + AW - WB + \Phi W - W\Psi = B-A + \Psi-\Phi.
    \end{align*}
    This equation can be viewed as a transport equation on the bundle $M\times \Cm^{n\times n}$ with (Hermitian) connection $\hat A(R):= AR - RB$ and (skew-hermitian) Higgs field $\hat\Phi(R) := \Phi R - R\Psi$, and the vanishing of $W$ at $\partial SM$ expresses that $I_{\hat A, \hat \Phi} (B-A + \Psi-\Phi) =0$. Theorem \ref{thm:Conn1} then implies that $W$ is only a function on the basepoint, and thus $U = Id + W$ fulfills all the desired properties.
\end{proof}

Following these ideas, we briefly describe similar contexts where such results have been obtained.

\subsubsection{Skew-hermitian pairs on simple magnetic surfaces} 

A similar scheme of proof was used in \cite{Ainsworth2013} to show that a skew-hermitian pair $(A,\Phi)$ is determined by the scattering data $C_{A,\Phi}$ defined through a simple magnetic flow, see \cite[Theorem 1.4]{Ainsworth2013}. The proof relies on the injectivity of all magnetic ray transforms with skew-hermitian pairs, proved in this context via energy identities as described in Section~\ref{sec:Pestov}.

\subsubsection{Skew-hermitian pairs on manifolds with negative sectional curvature} 

The conclusion of Theorem \ref{thm:Conn3} also holds on arbitrary bundles with Hermitian connection and skew-Hermitian Higgs field over a manifold $(M,g)$ with negative sectional curvature and strictly convex boundary, as stated in \cite[Theorem 1.2]{Guillarmou2015}. The scheme of proof mimicks that of Theorem \ref{thm:Conn3} by relying on the injectivity of the linear problem as stated in Theorem \ref{thm:Conn2}. The added technicality is the presence of trapping, which in the case of negative sectional curvature, is hyperbolic and allows to control appropriately the regularity of transport solutions, as explained in Section~\ref{sec:trapping}.  

\subsubsection{General pairs on manifolds admitting a strictly convex function} 

In general, one may notice that the scheme of proof above also works for pairs $(A,\phi)$ which are not necessarily skew-Hermitian. Namely, on a fixed Riemannian manifold of dimension at least three (say contractible with strictly convex boundary), if $\I_{A,\Phi}$ is injective over functions and one-forms for any smooth connection $A$ and Higgs field $\Phi$, then for any smooth pair $(A,\Phi)$, the scattering data $C_{A,\Phi}$ determines the pair $(A,\Phi)$ up to gauge. 

In this context, on Riemannian manifolds of dimension three and higher admitting a strictly convex function, and for any smooth pair $(A,\Phi)$, $\I_{A,\Phi}$ is proved injective in \cite{Paternain2016a}, then the authors also prove there that general pairs $(A,\Phi)$ are determined (up to gauge) by their scattering data, see e.g. \cite[Theorem 1.1]{Paternain2016a}.

\section{Open questions} \label{sec:open}

We conclude this review with some open questions. 

\begin{enumerate}
    \item On simply connected, non-trapping Riemannian surfaces with strictly convex boundary, it is commonly conjectured that the tensor tomography problem is solvable. 
    \item Is the X-ray transform solenoidally injective on tensor fields of all orders on a simple Riemannian manifold of any dimension? There is a complete answer in dimension two, but only partial results in higher dimensions.
    \item On a simple Riemannian manifold with boundary in dimension three and higher, the boundary rigidity problem is still open. 
    \item Regarding injectivity of the X-ray transform with conjugate points in Section~\ref{sec:microlocalCP}:
	\begin{enumerate}
	    \item What happens in the presence of regular conjugate points of order $n-1$ when $n\ge 3$? What about singular conjugate points? In dimensions three and higher, singular conjugate pairs (e.g., of type $D_4$, as defined in \cite{Arnold1972}) occur generically and correspond to singular conjugate pairs of order $2$.
	    \item What happens in the presence of regular conjugate points of any order violating the graph condition? (And when is the graph condition satisfied in the first place?) Some works studying FIOs where the graph condition fails may be found in \cite{Felea2005,Felea2013}, and examples of metrics of product type in \cite{Stefanov2012a} shows that failure of this condition can destroy stability in some cases. The question is however open in general.
	\end{enumerate}
    \item On simple Riemannian surfaces, is the ray transform injective when considering it over a bundle with any connection and Higgs field with structure group $GL(n,\Cm)$?
    \item Is there a local support theorem for tensor fields of all orders without real analytic metrics?
    \item Can a Pestov identity be used to prove an estimate which is localized in space? Localization in frequency is a recent observation~\cite{Paternain2018}.
    \item What happens to the various linear and non-linear integral geometry problems when the metric is not $C^\infty$? Namely, the results that hold for smooth metrics would be expected to hold with roughly $C^2$ regularity. 
    \item In dimensions three and higher, do simple manifolds admit strictly convex foliations?
    \item On which manifolds is the geodesic X-ray transform and variants thereof \emph{not} (solenoidally) injective?
\end{enumerate}

\section*{Acknowledgement}

J.I.\ was supported by the Academy of Finland (decision 295853). F.M.\ was partially supported by NSF grants DMS-1712790 and DMS-1814104.

We are grateful for discussions and feedback with 
Colin Guillarmou,
Sean Holman,
Thibault Lefeuvre,
Jere Lehtonen,
Gabriel Paternain,
Mikko Salo,
Vladimir Sharafutdinov,
Plamen Stefanov,
Gunther Uhlmann,
Andr\'as Vasy, and
Hanming Zhou.

\appendix

\section*{Proof of Lemma \ref{lem:composition}}
\label{app:composition-proof}

\begin{proof}[Proof of Lemma \ref{lem:composition}]
    Given $\omega\in T_{x'}^* M$, we compute immediately
    \begin{align*}
	d\pi|_{(x',v')}^T \omega = \omega(v') X^\flat_{(x',v')} - \omega({v'}_\perp) X_{\perp,(x',v')}^\flat.    
    \end{align*}
    Given $\eta\in T^*_\xi (\partial_+ SM)$, we compute, for any $t\in (0,\tau(\xi))$, 
    \begin{align*}
	dF|^T_{\varphi_t(\xi)} \eta = \eta\left( dF|_{\varphi_t(\xi)} (X_{\perp,\varphi_t(\xi)}) \right) X_{\perp,\varphi_t(\xi)}^\flat + \eta\left( dF|_{\varphi_t(\xi)} (V_{\varphi_t(\xi)}) \right) V^\flat_{\varphi_t(\xi)}.
    \end{align*}
    Given $(x',(x,v))\in M\times \partial_+ SM$ and $(\omega,\eta)\in T_{x'}^* M \times T_{(x,v)}^* (\partial_+ SM)$, to find whether $(\eta,\omega)\in \C_{I_0}$, the only point in $SM$ where canonical relations can compose is $(x',v')$, where $x' = \gamma_{x,v}(t)$ for some $t\in (0,\tau(x,v))$, and $v' = \dot \gamma_{x,v}(t)$. There, writing the condition $d\pi|_{(x',v')}^T \omega = dF|^T_{\varphi_t(x,v)} \eta$ gives
    \begin{align}
	0 = \omega(v'), \qquad \eta\left( dF|_{\varphi_t(x,v)} (X_{\perp,\varphi_t(x,v)}) \right) = -\omega(v'_\perp),\qquad \eta\left( dF|_{\varphi_t(x,v)} (V_{\varphi_t(x,v)}) \right) = 0. 
	\label{eq:linsys}
    \end{align} 
    This imposes $\omega = \lambda (v'_\perp)^\flat$ for some $\lambda\in \Rm$, $\lambda = \omega({v'}_\perp)$. We now use the following fact, proved after the conclusion
    \begin{align}
	\begin{split}
	    dF|_{\varphi_t(x,v)} \left( X_{\perp, \varphi_t(x,v)} \right) &= \dot a(x,v,t) V_{(x,v)} + \frac{\dot b(x,v,t)}{\dprod{\nu_x}{v}} \nabla_T|_{(x,v)}, \\
	    dF|_{\varphi_t(x,v)} \left( V_{\varphi_t(x,v)} \right) &= a(x,v,t) V_{(x,v)} + \frac{b(x,v,t)}{\dprod{\nu_x}{v}} \nabla_T|_{(x,v)},    	
	\end{split}
	\label{eq:dF}
    \end{align}
    and upon writing $\eta_V = \eta (V_{(x,v)})$ and $\eta_T = \eta (T_{(x,v)})$, the linear system \eqref{eq:linsys} has a unique solution 
    \begin{align}
	\eta_V = - b(x,v,t) \lambda, \qquad \eta_T = a(x,v,t) \lambda,
	\label{eq:eta}
    \end{align}    
    hence the proof.

    To prove \eqref{eq:dF}, we want to compute
    \begin{align*}
	dF|_{\varphi_t(x,v)} (X_{\perp,\varphi_t(x,v)}) \qquad \text{and} \qquad dF|_{\varphi_t(x,v)} (V_{\varphi_t(x,v)}).    
    \end{align*}
    A basis of $T_{(x,v)}(\partial_+ SM)$ is given by $V_{(x,v)}$ and $\nabla_T|_{(x,v)} = \dprod{T_x}{\nabla}$ (horizontal derivative along the tangent vector), where $T_x := -(\nu_x)_\perp$, in particular expressed in the frame $(X,X_\perp,V)$ as 
    \begin{align*}
	\nabla_T|_{(x,v)} = \dprod{T_x}{v}\dprod{v}{\nabla} + \dprod{T_x}{v_\perp}\dprod{v_\perp}{\nabla} &= \dprod{T_x}{v} X_{(x,v)} - \dprod{T_x}{v_\perp} X_{\perp(x,v)} \\
	&= \dprod{\nu_x}{v_\perp} X_{(x,v)} + \dprod{\nu_x}{v} X_{\perp(x,v)}.
    \end{align*}
    Now for every $(x,v)\in \partial_+ SM$ and $t\in (0,\tau(x,v))$, we have $F(\varphi_t(x,v))= (x,v)$, so that the following identity holds
    \begin{align}
	dF|_{\varphi_t(x,v)}\left( d\varphi_t|_{(x,v)} Y\right) = Y, \qquad Y \in T_{(x,v)} (\partial_+ SM).
	\label{eq:tmp10}
    \end{align}
    In addition, we can compute directly that 
    \begin{align*}
	d\varphi_t|_{(x,v)} (V_{(x,v)}) &= -b(x,v,t) X_{\perp, \varphi_t(x,v)} + \dot b(x,v,t) V_{\varphi_t(x,v)} \\
	d\varphi_t|_{(x,v)} (\nabla_T|_{(x,v)}) &= \dprod{\nu_x}{v_\perp} X_{\varphi_t(x,v)} + \dprod{\nu_x}{v} (a(x,v,t) X_{\perp, \varphi_t(x,v)} - \dot a(x,v,t) V_{\varphi_t(x,v)}).
    \end{align*} 
    Applying $dF|_{\varphi_t(x,v)}$, using \eqref{eq:tmp10} and the fact that $dF|_{\varphi_t(x,v)}( X_{\varphi_t(x,v)}) = 0$, we obtain the relations
    \begin{align*}
	dF|_{\varphi_t(x,v)} \left( -b(x,v,t) X_{\perp, \varphi_t(x,v)} + \dot b(x,v,t) V_{\varphi_t(x,v)} \right) &= V_{(x,v)} \\
	dF|_{\varphi_t(x,v)} \left( a(x,v,t) X_{\perp, \varphi_t(x,v)} - \dot a(x,v,t) V_{\varphi_t(x,v)} \right) &= \frac{1}{\dprod{\nu_x}{v}} \nabla_T|_{(x,v)},    
    \end{align*}
    which, via linear combinations and using $a\dot b- b\dot a = 1$, yields \eqref{eq:dF}. 
\end{proof}

\bibliographystyle{siam}
\bibliography{./bibliographyTR}

\end{document}